\documentclass[aos]{imsart}

\RequirePackage{amsthm,amsmath,amsfonts,amssymb}
\RequirePackage[numbers]{natbib}
\RequirePackage[colorlinks,citecolor=blue,urlcolor=blue]{hyperref}
\RequirePackage{graphicx}
\RequirePackage{xcolor}
\RequirePackage{booktabs}
\RequirePackage{tabularx}

\makeatletter
\def\journal@name{}%
\def\journal@url{}%
\makeatother

\startlocaldefs
\theoremstyle{plain}
\newtheorem{theorem}{Theorem}[section]
\newtheorem{lemma}[theorem]{Lemma}
\newtheorem{corollary}[theorem]{Corollary}
\newtheorem{proposition}[theorem]{Proposition}
\theoremstyle{definition}
\newtheorem{assumption}{Assumption}
\newtheorem{definition}{Definition}[section]
\theoremstyle{definition}
\newtheorem{remark}{Remark}[section]
\renewcommand{\appendixname}{Supplement}
\endlocaldefs

\begin{document}

\begin{frontmatter}
\title{When Does Dynamic Preconditioning Preserve the Polyak--Ruppert CLT? A Stabilization Threshold}
\runtitle{Stabilization Threshold for Dynamically Preconditioned SGD}
\runauthor{S.\ An and X.\ Huo}

\begin{aug}
\author[A]{\fnms{Sunyoung}~\snm{An}\ead[label=e1]{san49@gatech.edu}}
\author[A]{\fnms{Xiaoming}~\snm{Huo}\ead[label=e2]{huo@gatech.edu}}

\address[A]{H.\ Milton Stewart School of Industrial and Systems Engineering,
Georgia Institute of Technology\printead[presep={,\ }]{e1,e2}}
\end{aug}

\begin{abstract}
Polyak--Ruppert averaging yields an asymptotically normal estimator
with sandwich covariance $H^{-1}SH^{-1}$, the foundation of online
inference.  When the gradient step is preconditioned by a data-driven
matrix $P_t$, we ask how fast
$P_t$ must stabilize for the central limit theorem (CLT) to remain valid.

We resolve this via an exact \emph{preconditioner-isolating
decomposition} of the averaged error that confines $P_t$ to a dynamic
remainder $R_n$, leaving the martingale and Taylor terms
preconditioner-free.  Let $M_t=(P_tH)^{-1}$ denote the effective inverse
drift matrix, with $\|M_t-M_{t-1}\|_{\mathrm{op}}\lesssim t^{-\beta}$
and step-size exponent $\alpha\in(1/2,1)$.  We identify a
stabilization-rate threshold $\beta>(\alpha+1)/2$ and prove that,
within the class of polynomial rate hypotheses used in our upper bound,
it cannot be weakened: the dynamic
remainder $\sqrt{n}\,R_n$ vanishes in $L^2$ whenever
$\beta>(\alpha+1)/2$, and we exhibit sequences satisfying those
hypotheses for which it does not vanish when
$\beta\le(\alpha+1)/2$.

A single stabilization argument certifies three SA
variants---SA-AdaGrad, SA-RMSProp, and SA-ONS---with gain $\rho_t=c/t$,
each delivering one-step $L^2(\mathrm{op})$ stabilization of order
$t^{-1}$, yielding the CLT
$\sqrt{n}(\overline{x}_n-x^*)\xrightarrow{d}\mathcal{N}(0,H^{-1}SH^{-1})$;
under bounded inputs the pathwise rate $\beta=1$ further preserves the
$n^{-1/6}$ Wasserstein rate at $\alpha^*=2/3$.
Under standard regularity conditions, Wald-type online inference
remains valid for dynamically preconditioned averaged SGD whose
stabilization rate exceeds the threshold.
\end{abstract}

\begin{keyword}[class=MSC]
\kwd[Primary ]{62L20}
\kwd{62F12}
\kwd[; secondary ]{60F05}
\end{keyword}

\begin{keyword}
\kwd{Polyak--Ruppert averaging}
\kwd{asymptotic normality}
\kwd{stochastic approximation}
\kwd{central limit theorem}
\kwd{stabilization threshold}
\end{keyword}

\end{frontmatter}

\section{Introduction}
\label{sec:intro}

The central limit theorem (CLT) is a foundation of statistical inference:
it provides the asymptotic distribution needed for confidence intervals,
hypothesis tests, and efficiency comparisons
\citep{lehmann1998theory,vaart1998asymptotic}.
For iterate-averaged stochastic gradient methods, it specifies both
a Gaussian limit and its sandwich covariance in a single theorem
statement.

This foundation now underpins inference in streaming and online
settings---online A/B testing, continual monitoring of treatment
effects, and streaming M-estimation, for example---where the estimator
is updated one observation at a time and inference must be performed
in real time.  A line of recent work develops online inference procedures
for averaged SGD
\citep{chen2020statistical,lee2021fast,zhu2023online}.  In practice,
one-pass stochastic optimization is routinely combined with
adaptive preconditioning, which improves computational efficiency and
is believed to sharpen the resulting Gaussian approximation in finite
samples.
If the CLT fails or the asymptotic variance is altered by the adaptive preconditioning,
all downstream inference---coverage of confidence intervals, size of
hypothesis tests, consistency of plug-in covariance estimators---is
compromised.
A rigorous understanding of \emph{when adaptive preconditioning
preserves the CLT} is, therefore, a prerequisite for reliable inference
in these settings.

We formalize the setting.  Let $x \in \mathbb{R}^d$ denote the parameter to be
estimated, let $\zeta$ be a random data observation, and let
$f(x,\zeta)$ be a loss function.  The population risk (objective function) is
\[
  F(x) \;:=\; \mathbb{E}_{\zeta}[f(x,\zeta)],
  \qquad
  \nabla F(x) = \mathbb{E}_{\zeta}[\nabla f(x,\zeta)],
\]
where the interchange of differentiation and expectation is assumed
throughout.  The
estimation target is the unique minimizer $x^* := \arg\min_{x\in\mathbb{R}^d} F(x)$,
satisfying $\nabla F(x^*) = 0$.  We denote the population Hessian by
$H := \nabla^2 F(x^*)$ and the gradient covariance at the minimizer by
$S := \mathbb{E}[\nabla f(x^*,\zeta)\nabla f(x^*,\zeta)^\top]$,
which coincides with $\mathrm{Cov}(\nabla f(x^*,\zeta))$ since
$\nabla F(x^*) = 0$.

Stochastic gradient descent (SGD) \citep{robbins1951stochastic},
$x_{t+1} = x_t - \eta_t \nabla f(x_t, \zeta_t)$,
converges to~$x^*$ under standard step-size conditions, but its last
iterate is asymptotically suboptimal: to remove this inefficiency,
Polyak and Juditsky \citep{polyak1992acceleration} and Ruppert
\citep{ruppert1988efficient} introduced iterate averaging
$\overline{x}_n = n^{-1}\sum_{t=1}^{n} x_t$ and established
$\sqrt{n}(\overline{x}_n - x^*)
\xrightarrow{d}
\mathcal{N}(0, H^{-1}SH^{-1})$.
In regular well-specified likelihood models the information matrix
equality gives $S = H$, so the sandwich covariance reduces to $H^{-1}$,
the inverse Fisher information, and the averaged iterate is asymptotically
efficient \citep{vaart1998asymptotic}.

While Polyak--Ruppert averaging guarantees the target asymptotic distribution,
its finite-sample behavior can be poor when the objective is ill-conditioned.
When the eigenvalues of the Hessian $H$ span several orders of magnitude,
unpreconditioned SGD oscillates along high-curvature directions and
makes slow progress along low-curvature ones; the resulting transient
behavior delays convergence of the averaged iterate to its asymptotic regime
\citep{bottou2018optimization}.
Preconditioning addresses this through a matrix $P_t$ that rescales the
gradient step along different directions:
\[
  x_{t+1} = x_t - \eta_t P_t \nabla f(x_t, \zeta_t),
\]
where $P_t$ is built from information available before step~$t$
(a predictability condition formalized in Section~\ref{sec:problem}).
A well-chosen $P_t$ approximately equalizes curvature across directions,
so that each converges at a comparable rate, bringing the
finite-sample distribution of the averaged iterate closer to its Gaussian
limit.
From an inferential standpoint, one would expect that confidence intervals
based on the preconditioned average reach their nominal coverage at
smaller sample sizes than under unpreconditioned SGD---provided the CLT
itself is not disrupted by the evolving preconditioner.
Classical examples include AdaGrad \citep{duchi2011adaptive}, RMSProp
\citep{tieleman2012rmsprop}, and online Newton-type methods
\citep{hazan2007logarithmic}; adaptive preconditioning of this kind
reduces the effective condition number of the problem.

In all these methods, however, $P_t$ is itself estimated
from the data and, therefore, evolves with the iterates, introducing
additional time variation into the recursion.  This creates a fundamental
tension: preconditioning improves finite-sample convergence by reducing the
anisotropy of the problem geometry, but the time-varying preconditioner
could alter the leading term of the averaged error and invalidate
the CLT on which all downstream inference depends.
Indeed, online inference procedures developed for averaged SGD
\citep{chen2020statistical,lee2021fast,zhu2023online} are designed for
the unpreconditioned case and do not cover adaptively preconditioned
methods.

This tension leads to the central question of the paper:
\begin{quote}
\emph{How fast must a data-dependent preconditioner $P_t$ stabilize for
Polyak--Ruppert averaged SGD to preserve its CLT, and how tight is this
requirement under standard polynomial-rate proof techniques?}
\end{quote}
Existing non-asymptotic analyses of averaged SGD do not incorporate
dynamically evolving preconditioners
\citep{gadat2023optimal,mou2020linear}; conditioned-SGD CLTs
\citep{leluc2023asymptotic,boyer2023stochastic} do not derive the
stabilization rate required for the CLT; and optimization-oriented
analyses of adaptive methods do not study Polyak--Ruppert averaging
\citep{godichon2024adagrad,godichon2025adaptive}.
A detailed comparison is given in Section~\ref{sec:related}.

\paragraph*{Contribution~I: exact preconditioner-isolating decomposition}
Our starting point is an exact, pathwise decomposition of
the averaged error (Lemma~\ref{lem:decomp}):
\[
  \overline{x}_n - x^*
  \;=\;
  \underbrace{-\frac{1}{n}H^{-1}\sum_{t=1}^{n}\xi_t}_{\text{martingale term }\Xi_n}
  \;\underbrace{-\;\frac{1}{n}H^{-1}\sum_{t=1}^{n}u_t}_{\text{Taylor remainder }T_n}
  \;+\; R_n(\{P_t\}),
\]
where $\xi_t$ is the martingale innovation at step $t$ and $u_t$ is the
second-order Taylor remainder from linearizing $\nabla F$ at $x^*$.
The key property is that \emph{no preconditioner matrix appears in either
the martingale term $\Xi_n$ or the Taylor remainder $T_n$}; all explicit
dependence on the evolving preconditioner sequence is confined to the dynamic
remainder $R_n(\{P_t\})$.  (The terms $\Xi_n$ and $T_n$ can still depend
on~$P_t$ implicitly through the iterate trajectory~$\{x_t\}$, but no
preconditioner matrix enters their algebraic form.)
This is a stronger separation than existing analyses provide:
prior CLT proofs for preconditioned SGD
\citep{leluc2023asymptotic,boyer2023stochastic,godichon2024hessian} typically
retain the time-varying preconditioner inside the leading term and invoke
an almost-sure limit $P_t \to P$ to identify the Gaussian covariance.
Our decomposition yields the exact leading term
$-n^{-1}H^{-1}\sum\xi_t$ without requiring $P_t$ to converge.  It thereby
reduces the CLT problem under dynamic preconditioning to a concrete
analytic question---\emph{how fast must $P_t$ stabilize for
$\sqrt{n}\,R_n \to 0$?}---which is exactly what the next contribution answers.

\paragraph*{Contribution~II: a stabilization-rate threshold with a rate-saturating construction}
The decomposition above converts the CLT question into a rate condition
on the preconditioner sequence through the dynamic remainder $R_n$, and
our central contribution answers it with a threshold that determines
when Polyak--Ruppert averaging preserves its CLT---and, hence, its
inferential validity---under dynamic preconditioning.  The condition
is expressed through a \emph{stabilization rate} $\beta > 0$, defined by
$\|M_t - M_{t-1}\|_{\mathrm{op}} \lesssim t^{-\beta}$ where
$M_t = (P_tH)^{-1}$ is the effective inverse drift matrix, with
step-size exponent $\alpha \in (1/2,1)$ in $\eta_t \propto t^{-\alpha}$.
We prove:
\begin{enumerate}
\item[\emph{(i)}] \emph{Sufficiency.}\;
  Under bounded $\|M_t\|_{\mathrm{op}}$,
  $\beta > (\alpha+1)/2$ implies
  $\sqrt{n}\,R_n(\{P_t\}) \to 0$ in $L^2$
  (Theorem~\ref{thm:sharp_threshold}).
\item[\emph{(ii)}] \emph{A saturating construction.}\;
  We exhibit a deterministic sequence satisfying the boundedness and
  stabilization hypotheses with $\beta \le (\alpha+1)/2$ for which
  $\sqrt{n}\,R_n$ does not vanish (Proposition~\ref{thm:sharpness}),
  showing that the exponent $(\alpha+1)/2$ cannot be lowered within
  the rate-hypothesis class used in our proof.  The scope of this tightness, and the open
  question of whether algorithmic coupling in a genuine adaptive
  recursion can lower the threshold further, are stated precisely in
  Remark~\ref{rem:sharpness_scope}.
\end{enumerate}
The threshold, therefore, lies in $(3/4,1)$.  Together with the
decomposition, the two results
convert the abstract question ``does this preconditioner preserve the
CLT?'' into a concrete, checkable condition: verifying
$\beta > (\alpha+1)/2$ together with bounded $\|M_t\|_{\mathrm{op}}$
suffices---under the regularity conditions of
Theorem~\ref{thm:clt}---to guarantee that the sandwich covariance
$H^{-1}SH^{-1}$ is the asymptotic covariance of the averaged iterate
and that Wald-type inference retains its nominal properties.

\paragraph*{Contribution~III: stabilization of three SA preconditioner constructions}
We prove a single stabilization result (Theorem~\ref{thm:sa_stab}) for a
stochastic-approximation (SA) recursion of symmetric positive-definite
matrices (or, in the diagonal case, of positive vectors),
$Q_t = (1-\rho_t)Q_{t-1} + \rho_t W_t$ with $\rho_t = c/t$, composed with
either the matrix inverse $A \mapsto A^{-1}$ or the inverse square root
$A\mapsto A^{-1/2}$.  The same proof simultaneously certifies three
concrete online preconditioners---SA full-matrix AdaGrad, diagonal
SA-RMSProp, and SA online Newton step (ONS)---each obtained by specifying
the accumulator $Q_t$, the driving term $W_t$, and the spectral map.  The
result delivers $L^2$ one-step stabilization of order $t^{-1}$, which
upgrades to the pathwise rate $\beta = 1$ under almost-sure bounded inputs.
These constructions are SA \emph{reparametrizations} of the classical
methods---the $c/t$ gain replaces AdaGrad's cumulative-sum accumulator
and RMSProp's constant-EMA weight---so the threshold certifies these SA
variants rather than the originally published AdaGrad, RMSProp, or Adam;
constant-EMA RMSProp is revisited as a threshold-violation case in
Section~\ref{sec:simulation}.

Since all three constructions deliver one-step stabilization of
order $t^{-1}$---pathwise (giving $\beta = 1 > (\alpha+1)/2$ in the
sense of Definition~\ref{def:beta}) under bounded-input conditions,
or in $L^2$ without bounded inputs via
Proposition~\ref{prop:l2_remainder}---the CLT
$\sqrt{n}(\overline{x}_n-x^*) \xrightarrow{d} \mathcal{N}(0,H^{-1}SH^{-1})$
holds for all $\alpha \in (1/2,1)$ (Theorem~\ref{thm:clt}).  It further implies
a quantitative Gaussian approximation bound in Wasserstein distance
(Corollary~\ref{cor:gauss_approx}) that \emph{preserves} the $n^{-1/6}$
rate \citep{kong2026finite} for \emph{unpreconditioned} nonlinear SA at
the optimal step-size exponent $\alpha^* = 2/3$.  Note that the
$n^{-1/6}$ exponent itself is not our contribution---it is inherited
from the unpreconditioned nonlinear baseline; our contribution is
showing that dynamic preconditioning above the threshold incurs no
rate penalty relative to that baseline.

\paragraph*{Contribution~IV: finite-sample operator-factor comparison with a factor-$\kappa(H)$ improvement}
Propagating the one-step stabilization bounds through the decomposition
yields an explicit finite-sample $L^2$ bound on $\sqrt{n}\,R_n$
(Proposition~\ref{prop:tight}), whose constant contains the pathwise
sup $C_P = \sup_t \|M_t\|_{\mathrm{op}}$; passing to the limit isolates
the preconditioner-dependent \emph{limiting operator factor}
$\bar C_P = \|M_\infty\|_{\mathrm{op}}$ with
$M_\infty = \lim_{t\to\infty} M_t$.  Computing this limiting factor for
each of the three SA constructions and for the identity baseline
(Polyak--Ruppert SGD with $P_t \equiv I$) gives a quantitative
comparison across preconditioner choices
(Proposition~\ref{prop:comparison}): under the spectral
normalization $\lambda_{\max}(H) = 1$, the identity baseline carries
$\bar C_P^{\mathrm{Id}} = \kappa(H)$, whereas SA-ONS drives $M_t \to I$
and attains $\bar C_P^{\mathrm{ONS}} = 1$.  SA-ONS, therefore, reduces
the limiting operator factor by a factor of $\kappa(H)$ relative to
Polyak--Ruppert SGD, at no cost to the asymptotic covariance---all three SA constructions share the
same sandwich covariance $H^{-1}SH^{-1}$.  This is, to our knowledge, the
first result that separates the three SA preconditioners
quantitatively in the finite-sample regime while preserving asymptotic
equivalence, and it gives a principled reason, on this operator-factor
measure, to prefer SA-ONS on ill-conditioned problems despite the higher
per-iteration computational cost.

\paragraph*{Roadmap of the core results}
The technical contributions are concentrated in five results:
Lemma~\ref{lem:decomp} (the exact preconditioner-isolating decomposition),
Theorem~\ref{thm:sharp_threshold} (sufficiency of
$\beta > (\alpha+1)/2$), Proposition~\ref{thm:sharpness} (the saturating
construction for the sufficiency hypotheses),
Theorem~\ref{thm:sa_stab} (stabilization of the three SA preconditioner
constructions), and Proposition~\ref{prop:comparison} (the
operator-factor comparison with its factor-$\kappa(H)$ improvement for
SA-ONS).  The remaining downstream results---the CLT
(Theorem~\ref{thm:clt}), the quantitative Gaussian approximation in
Wasserstein distance (Corollary~\ref{cor:gauss_approx}), and the
finite-sample $L^2$ bound (Proposition~\ref{prop:tight})---are
consequences of these five.

\subsection{Related Work}
\label{sec:related}

Our contribution sits at the intersection of four threads: (i)~averaged SA
theory, (ii)~online inference, (iii)~adaptive preconditioning, and
(iv)~stochastic Newton methods.

\textit{Averaged SA.}\;
Polyak and Juditsky~\citep{polyak1992acceleration} established a central
limit theorem for the averaged iterate with asymptotic covariance
$H^{-1}SH^{-1}$; an optimal non-asymptotic upper bound on the
mean-square error with leading term
$\mathrm{Tr}(H^{-1}SH^{-1})/n$ was subsequently established
by Gadat and Panloup~\citep{gadat2023optimal}.
Subsequent refinements and extensions include non-asymptotic
$O(1/n)$ rates for least-squares averaged SGD
\citep{bach2013non}, CLTs for linear SA
\citep{mou2020linear,butyrin2025improved}, Gaussian approximation and
multiplier bootstrap for linear SA \citep{samsonov2024gaussian},
Berry--Esseen bounds for averaged SGD \citep{shao2022berry},
multiplier bootstrap for SGD inference \citep{sheshukova2025bootstrap},
and extensions to weighted \citep{wei2025weighted} and momentum
\citep{tang2023momentum} averaging.  None of these works analyzes
dynamic preconditioning.

\textit{Online inference.}\;
Covariance estimators and online inference tools for unpreconditioned
averaged SGD are developed in
\citep{chen2020statistical,lee2021fast,zhu2023online}, with online
bootstrap confidence intervals for the SGD estimator in
\citep{fang2018online}.  Our
work shows that the same limiting distribution carries over to the
dynamically preconditioned setting when $\beta > (\alpha+1)/2$ and
$\|M_t\|_{\mathrm{op}}$ stays bounded, so that sandwich-based inference
retains its nominal coverage properties.

\textit{Adaptive preconditioning.}\;
AdaGrad \citep{duchi2011adaptive} and Online Newton Step
\citep{hazan2007logarithmic} are the classical references; RMSProp
was introduced in \citep{tieleman2012rmsprop}, with a convergence
analysis in \citep{shi2021rmsprop}.
Further extensions include curvature-based, structured, and general
preconditioner designs
\citep{schraudolph2007stochastic,byrd2016stochastic,li2015preconditioned,scott2025designing,xie2025structured}
and convergence analyses of adaptive methods
\citep{godichon2024adagrad,godichon2025adaptive,surendran2024non}.
The adaptive-preconditioning literature focuses on optimization
convergence rather than inferential consequences.

\textit{Stochastic Newton.}\;
Boyer and Godichon-Baggioni~\citep{boyer2023stochastic} establish CLTs
for weighted-averaged stochastic Newton, and Leluc and
Portier~\citep{leluc2023asymptotic} prove a CLT for conditioned SGD
under $P_t \to P$ a.s.; our decomposition, non-asymptotic bounds, and
stabilization-rate analysis are complementary to both.
See also
\citep{bercu2020efficient,godichon2024hessian,godichon2025streaming,cenac2025gauss}
for related Newton-type analyses.
In a high-dimensional regime, Jagannath
et~al.~\citep{jagannath2025highdim} derive scaling limits---a parallel
line of work.

Table~\ref{tab:positioning} summarizes the positioning.  To the best of our
knowledge, no prior work provides a stabilization-rate threshold for
the CLT under dynamic preconditioning together with a tightness result
for the sufficiency hypotheses of the upper-bound proof.

\begin{table}[t]
\centering
\small
\caption{Positioning relative to the most closely related works.
Columns: stabilization-rate threshold with tightness of the sufficiency
hypotheses (Thresh.\ \& tight.); preconditioner-isolating
decomposition (Decomp.); non-asymptotic remainder bounds (Non-asymp.);
CLT for averaged iterates; SA preconditioner construction with a
stabilization rate matching the threshold $\beta > (\alpha+1)/2$
(SA constr.; ``partial'' = construction provided but the threshold
rate not fully established).}
\label{tab:positioning}
\begin{tabularx}{\textwidth}{@{}Xccccc@{}}
\toprule
\textbf{Reference} & \textbf{Thresh.\ \& tight.} & \textbf{Decomp.} & \textbf{Non-asymp.} & \textbf{CLT} & \textbf{SA constr.} \\
\midrule
Polyak \& Juditsky \citep{polyak1992acceleration}  & \texttimes & \texttimes & \texttimes & \checkmark & \texttimes \\
Gadat \& Panloup \citep{gadat2023optimal}          & \texttimes & \texttimes & \checkmark & \texttimes & \texttimes \\
Mou et~al.\ \citep{mou2020linear}                  & \texttimes & \texttimes & \checkmark & \checkmark & \texttimes \\
Leluc \& Portier \citep{leluc2023asymptotic}        & \texttimes & \texttimes & \texttimes & \checkmark & \texttimes \\
Boyer \& Godichon-B.\ \citep{boyer2023stochastic}  & \texttimes & \texttimes & \texttimes & \checkmark & \texttimes \\
Godichon-B.\ et~al.\ \citep{godichon2024hessian}   & \texttimes & \texttimes & \texttimes & \checkmark & partial \\
Surendran et~al.\ \citep{surendran2024non}          & \texttimes & \texttimes & \checkmark & \texttimes & \texttimes \\
\textbf{This paper}                                & \checkmark & \checkmark & \checkmark & \checkmark & \checkmark \\
\bottomrule
\end{tabularx}
\end{table}

The remainder of the paper is organized as follows.
Section~\ref{sec:problem} introduces the setup and assumptions.
Section~\ref{sec:decomp} develops the preconditioner-isolating decomposition
and establishes the stabilization-rate threshold.
Section~\ref{sec:construction} introduces three SA preconditioner
constructions and proves their stabilization.
Section~\ref{sec:main} presents the finite-sample and asymptotic
analysis, including the CLT, quantitative Gaussian approximation,
and the operator-factor comparison across preconditioner choices.
Section~\ref{sec:simulation} reports numerical experiments.
Section~\ref{sec:discussion} concludes with a discussion.
The supplement collects deferred proofs and additional experiments diagnostics.

\section{Setup and Assumptions}
\label{sec:problem}

This section introduces the probabilistic setup and the model
assumptions, and then specifies the preconditioned recursion that will be
analyzed throughout the paper.

We retain the notation introduced in Section~\ref{sec:intro}: $f$, $F$,
$x^*$, $H$, and $S$ denote the sample loss, the population risk, its unique
minimizer, the population Hessian, and the gradient covariance at $x^*$,
respectively.
We write $A \preceq B$ when $B-A$ is positive semidefinite, and
$\|A\|_{\mathrm{op}}$ for the operator norm of a matrix~$A$.

Let $\{\zeta_t\}_{t \ge 1}$ be a sequence of i.i.d.\ random variables
defined on a probability space $(\Omega, \mathcal{F}, \mathbb{P})$.
Throughout, $\{x_t\}_{t\ge1}$ denotes the iterate sequence generated by the
preconditioned update introduced below, with initialization $x_1$ assumed
deterministic; more generally, the same arguments apply when $x_1$ is
random and independent of $\{\zeta_t\}_{t \ge 1}$ with
$\mathbb{E}\|x_1 - x^*\|^2 < \infty$.  Define the natural filtration
$\mathcal{F}_t := \sigma(x_1,\zeta_1,\ldots,\zeta_t)$, $t \ge 0$.
Let $\xi_t(x) := \nabla f(x, \zeta_t) - \nabla F(x)$ denote the
stochastic gradient noise, and write
$\xi_t := \xi_t(x_t) = \nabla f(x_t,\zeta_t) - \nabla F(x_t)$
for its evaluation at the current iterate.

\subsection{Assumptions}

The assumptions here describe the basic stochastic approximation
model; additional moment or bounded-increment conditions used for
sharper stabilization-rate bounds and quantitative Gaussian
approximation results are introduced where they are invoked.
We adopt assumptions that parallel those of the original Polyak--Ruppert
averaged SGD analysis~\citep{polyak1992acceleration,ruppert1988efficient}
so that our central
claim---preservation of the Polyak--Ruppert CLT under dynamic
preconditioning---is directly comparable to the unpreconditioned
baseline.  Each of Assumptions~\ref{ass:mg}--\ref{ass:quadratic} is standard in
the averaged SGD and adaptive SA literature, as documented by the
per-assumption citations below; Assumption~\ref{ass:iterate} is the
(derived) iterate MSE rate that Proposition~\ref{prop:iterate_bound}
establishes under primitive conditions.

\begin{assumption}[Martingale Difference]
\label{ass:mg}
$\mathbb{E}[\xi_t \mid \mathcal{F}_{t-1}] = 0$ for all $t \geq 1$.
\end{assumption}

\noindent
Assumption~\ref{ass:mg} is the standard unbiased gradient noise condition
\citep{polyak1992acceleration,chen2020statistical,mou2020linear,zhu2023online,%
gadat2023optimal,boyer2023stochastic,leluc2023asymptotic,bach2013non,%
godichon2025adaptive,shao2022berry}.
Since $\zeta_t$ is independent of $\mathcal{F}_{t-1}$ and $x_t$ is
$\mathcal{F}_{t-1}$-measurable, Assumption~\ref{ass:mg} follows whenever
$\nabla f(x,\zeta)$ is an unbiased estimator of $\nabla F(x)$.

\begin{assumption}[Uniform Conditional Covariance Bound]
\label{ass:var}
There exists a deterministic positive semidefinite matrix $\overline{S}$ such
that $\mathbb{E}[\xi_t \xi_t^\top \mid \mathcal{F}_{t-1}] \preceq \overline{S}$
for all $t \geq 1$.
\end{assumption}

\noindent
Assumption~\ref{ass:var} imposes a standard conditional second-moment
control on the gradient noise, with analogous second-moment conditions
appearing in
\citep{polyak1992acceleration,chen2020statistical,mou2020linear,gadat2023optimal,%
boyer2023stochastic,leluc2023asymptotic,samsonov2024gaussian,%
wei2025weighted,godichon2025adaptive}.
Here $S = \mathrm{Cov}(\nabla f(x^*,\zeta))$ is the covariance at $x^*$,
and $\overline{S}$ is a uniform conditional upper bound used for
non-asymptotic control along the trajectory;
in general $S \preceq \overline{S}$, with equality when the conditional
covariance does not depend on the iterate.

\begin{assumption}[Strong Convexity]
\label{ass:convex}
$F$ is $\mu$-strongly convex for some $\mu > 0$.
\end{assumption}

\noindent
Assumption~\ref{ass:convex}, or the closely related local curvature
condition at $x^*$ (positive-definite Hessian at the optimum), is standard
in the averaged SGD and adaptive SA literature
\citep{polyak1992acceleration,chen2020statistical,boyer2023stochastic,%
gadat2023optimal,wei2025weighted,bercu2020efficient,%
godichon2024hessian,zhu2023online,leluc2023asymptotic}.
Assumption~\ref{ass:convex}, combined with the local second-order expansion
in Assumption~\ref{ass:quadratic}, implies $H \succeq \mu I \succ 0$ and
$\|H^{-1}\|_{\mathrm{op}} \leq \mu^{-1}$.
In the linear SA setting \citep{mou2020linear,samsonov2024gaussian}, the
analogous condition is that the driving matrix has eigenvalues with strictly
positive real part.

\begin{assumption}[Local Second-Order Expansion with Trajectory Confinement]
\label{ass:quadratic}
There exist a neighborhood $\mathcal{N}$ of $x^*$ and a constant $L_R > 0$
such that $F$ is twice continuously differentiable on $\mathcal{N}$,
$\nabla^2 F(x^*) = H$, and
\[
  \nabla F(x) = H(x - x^*) + r(x),
  \qquad
  \|r(x)\| \leq L_R \|x - x^*\|^2
\]
for all $x \in \mathcal{N}$; moreover, the iterate sequence satisfies
$x_t \in \mathcal{N}$ almost surely for all $t \ge 1$.
\end{assumption}

\noindent
The local-expansion part of Assumption~\ref{ass:quadratic} is standard
\citep{polyak1992acceleration,chen2020statistical,boyer2023stochastic,%
shao2022berry,bercu2020efficient}
and follows whenever $\nabla^2 F$ is locally Lipschitz near~$x^*$.
The same local quadratic expansion underlies the Berry--Esseen bounds
of Shao and Zhang~\citep{shao2022berry} for nonlinear $M$-estimators
and the stochastic Newton analysis of Bercu, Godichon-Baggioni,
and Portier~\citep{bercu2020efficient}; unlike implicit linearization
approaches, we retain the remainder $r(x)$ explicitly, which is what
makes the preconditioner-isolating decomposition of
Lemma~\ref{lem:decomp} exact rather than approximate.  The containment condition $x_t \in \mathcal{N}$ is
vacuous when $F$ is globally smooth ($\mathcal{N} = \mathbb{R}^d$).
For locally smooth objectives it can be justified a~posteriori: once
almost-sure iterate convergence is established by other means,
$x_t \in \mathcal{N}$ holds for all~$t$ beyond some (random)
time~$\tau_0$, and the CLT applies with~$\tau_0$ in place
of~$t=1$ via a shift-of-origin argument
\citep{polyak1992acceleration}; the quantitative Gaussian
approximation (Corollary~\ref{cor:gauss_approx}), being a
finite-sample bound, requires the containment to hold from
$t=1$.
Corollaries~\ref{cor:adaptive_iterate} and~\ref{cor:ons_iterate}
enforce the containment by positing a deterministic convex set
(inside $\mathcal{N}$) that contains all iterates almost surely
(hypothesis~(b) therein).

Under Assumption~\ref{ass:quadratic}, define
the Taylor remainder $u_t := r(x_t)$, so that
\begin{equation} \label{eq:expansion}
  \nabla f(x_t, \zeta_t) = H(x_t - x^*) + u_t + \xi_t,
  \qquad
  \|u_t\| \leq L_R \|x_t - x^*\|^2,
\end{equation}
where $\xi_t = \nabla f(x_t,\zeta_t) - \nabla F(x_t)$.

\begin{assumption}[Iterate Bound]
\label{ass:iterate}
There exist constants $C_\Delta > 0$ and $\alpha \in (1/2, 1)$ such that
$\mathbb{E}\|\Delta_t\|^2 \leq C_\Delta\, t^{-\alpha}$ for all $t \geq 1$,
where $\Delta_t = x_t - x^*$.
\end{assumption}

\noindent
This is a standard MSE rate achieved by step-size
$\eta_t = \eta_0 t^{-\alpha}$ under strong convexity
\citep{gadat2023optimal,godichon2025adaptive,chen2020statistical,%
boyer2023stochastic,godichon2024adagrad}.
Crucially for our setting, Godichon-Baggioni, Lu, and
Portier~\citep{godichon2024adagrad} establish convergence rates for an
AdaGrad-type adaptive method, confirming that adaptively preconditioned
schemes admit the kind of MSE decay we require here.
Proposition~\ref{prop:iterate_bound} gives a sufficient condition, and
Corollaries~\ref{cor:adaptive_iterate}--\ref{cor:ons_iterate} establish it for
all three SA preconditioners under their respective bounded-input hypotheses.

\subsection{Preconditioned Update Recursion}
\label{sec:update}

We consider the following preconditioned SGD recursion.  Let
$\eta_t = \eta_0 t^{-\alpha}$ with $\eta_0 > 0$ and
$\alpha \in (1/2, 1)$, and let
$P_t \in \mathbb{R}^{d \times d}$ be a time-varying symmetric positive
definite preconditioner that is $\mathcal{F}_{t-1}$-measurable
(equivalently, predictable with respect to the filtration
$\{\mathcal{F}_t\}_{t \ge 0}$):
\[
    x_{t+1} = x_t - \eta_t P_t \nabla f(x_t, \zeta_t).
\]
Subtracting $x^*$ from both sides of the update and substituting the
expansion~\eqref{eq:expansion} yields the error recursion
\begin{equation} \label{eq:recursion}
    \Delta_{t+1} = (I - \eta_t P_t H)\Delta_t - \eta_t P_t \xi_t
    - \eta_t P_t u_t.
\end{equation}

\begin{proposition}[A Sufficient Condition for Assumption~\ref{ass:iterate}]
\label{prop:iterate_bound}
Consider the preconditioned SGD update
\[
    x_{t+1} = x_t - \eta_t P_t \nabla f(x_t,\zeta_t),
    \qquad
    \eta_t = \eta_0 t^{-\alpha},
    \qquad
    \alpha \in (1/2,1),
\]
where $P_t$ is $\mathcal{F}_{t-1}$-measurable (predictable) and the
initialization $x_1$ satisfies $\mathbb{E}\|\Delta_1\|^2 < \infty$
(in particular, $x_1$ deterministic, or more generally $x_1$ random
and independent of $\{\zeta_t\}_{t\ge 1}$ with finite second moment).
Assume, in addition to Assumptions~\ref{ass:mg}--\ref{ass:convex}, that:
\begin{enumerate}
  \item[(i)] $F$ is $L$-smooth on a convex set
    $\mathcal{K} \subseteq \mathbb{R}^d$ that contains $x^*$ and satisfies
    $x_t \in \mathcal{K}$ almost surely for all $t \geq 1$;
  \item[(ii)] there exist deterministic constants
    $0 < p_- \leq p_+ < \infty$ such that
    \[
        p_- I \preceq P_t \preceq p_+ I
        \qquad \text{almost surely for all } t \geq 1.
    \]
\end{enumerate}
Then there exists a constant $C_\Delta > 0$ such that
\[
    \mathbb{E}\|\Delta_t\|^2 \leq C_\Delta\, t^{-\alpha}
    \qquad \text{for all } t \geq 1.
\]
\end{proposition}

\begin{proof}
Deferred to the supplementary material.
\end{proof}

Proposition~\ref{prop:iterate_bound} follows from a standard descent
argument under $L$-smoothness and $\mu$-strong convexity, specialized
to the predictable-preconditioner setting.  Note that
Proposition~\ref{prop:iterate_bound} does not invoke
Assumption~\ref{ass:quadratic}: $L$-smoothness of $F$ on $\mathcal{K}$
together with $\mu$-strong convexity already suffice for the iterate
MSE rate, so the local second-order expansion is not needed here.
Condition~(i) is vacuous when $F$ is globally
smooth; for locally smooth objectives it is a pathwise containment
hypothesis that must be verified separately (e.g., via problem-specific
Lyapunov arguments).  This condition also implies the containment part
of Assumption~\ref{ass:quadratic} whenever $\mathcal{K} \subseteq \mathcal{N}$.
After the three SA preconditioners are introduced in
Section~\ref{sec:construction}, Corollaries~\ref{cor:adaptive_iterate}
and~\ref{cor:ons_iterate} apply
Proposition~\ref{prop:iterate_bound}: their hypothesis~(b) supplies
condition~(i), while condition~(ii) is established from the primitive
bounded-gradient or bounded-Hessian inputs, thereby securing both the
containment part of Assumption~\ref{ass:quadratic} and
Assumption~\ref{ass:iterate}.  With the assumption set and the
preconditioned recursion in place, Section~\ref{sec:decomp} develops
the exact telescoping identity that isolates the preconditioner
sequence in a single dynamic remainder and drives the analysis that
follows.

\section{Preconditioner-Isolating Decomposition}
\label{sec:decomp}

The analysis that follows is organized around an exact pathwise
decomposition of the averaged error
$\overline{\Delta}_n := n^{-1}\sum_{t=1}^n \Delta_t$
(Lemma~\ref{lem:decomp}) into three pieces: a martingale term $\Xi_n$,
a Taylor remainder $T_n$, and a dynamic remainder $R_n$ that carries
all explicit dependence on the preconditioner sequence.  Earlier CLT
proofs for preconditioned SGD invoke an almost-sure limit $P_t \to P$
to identify the Gaussian covariance; the decomposition~\eqref{eq:decomposition}
replaces this indirect step with an explicit pathwise formula.  Because $\Xi_n$ and
$T_n$ are controlled by standard martingale and iterate-rate estimates
(Propositions~\ref{prop:variance} and~\ref{prop:taylor}), the
preconditioner-specific obstacle to the CLT reduces to showing
$\sqrt{n}\,R_n \to 0$, and the stabilization threshold is obtained
from a direct analysis of $R_n$ alone
(Theorem~\ref{thm:sharp_threshold}).

\begin{lemma}[Exact Preconditioner-Isolating Decomposition]
\label{lem:decomp}
Under the recursion \eqref{eq:recursion} and
Assumptions~\ref{ass:convex}--\ref{ass:quadratic}, the Polyak--Ruppert
averaged error $\overline{\Delta}_n$
admits the exact pathwise decomposition
\begin{equation} \label{eq:decomposition}
    \overline{\Delta}_n
    = \underbrace{-\frac{1}{n}H^{-1}\sum_{t=1}^n \xi_t}_{\displaystyle \Xi_n}
    \;\underbrace{-\frac{1}{n}H^{-1}\sum_{t=1}^n u_t}_{\displaystyle T_n}
    \;+\; R_n(\{P_t\}),
\end{equation}
where $\Xi_n$ is the martingale term, $T_n$ is the Taylor
remainder, and the dynamic remainder is defined by
\begin{equation} \label{eq:remainder}
    R_n(\{P_t\})
    :=
    \frac{1}{n} \left[
        A_1\Delta_1 - A_n\Delta_{n+1}
        + \sum_{t=2}^n (A_t-A_{t-1})\Delta_t
    \right],
    \qquad
    A_t := \eta_t^{-1}(P_tH)^{-1}.
\end{equation}
In particular, no preconditioner matrix appears explicitly in $\Xi_n$ or
$T_n$; all explicit dependence on the preconditioner sequence is carried by
$R_n(\{P_t\})$.  The terms $\Xi_n$ and $T_n$ may still depend implicitly on the
preconditioned trajectory through the iterate sequence $\{x_t\}$.
\end{lemma}

\begin{proof}
Deferred to the supplementary material.
\end{proof}

\begin{remark}[Role of the decomposition]
\label{rem:decomp_novelty}
The proof combines left-multiplication by $\eta_t^{-1}(P_tH)^{-1}$
(using the cancellation $(P_tH)^{-1}P_t = H^{-1}$) with Abel summation.
In prior CLT proofs, the preconditioner remains inside the leading
term ($C_t\xi_t$ in~\citep{leluc2023asymptotic}, $H_t^{-1}\xi_t$
in~\citep{boyer2023stochastic,godichon2024hessian}), so an explicit
threshold on the stabilization rate is not visible in their analysis.
Lemma~\ref{lem:decomp} yields the exact leading term
$-n^{-1}H^{-1}\sum_{t=1}^n \xi_t$ with no $P_t$, enabling the explicit
threshold analysis of \S\ref{sec:remainder}.
Lemma~\ref{lem:decomp} extends the fixed-matrix identity of
Mou et~al.~\citep{mou2020linear} to time-varying $P_t$.
\end{remark}

\begin{proposition}[Second-Moment Bound for the Statistical-Noise Term]
\label{prop:variance}
Under Assumptions~\ref{ass:mg}--\ref{ass:convex},
\begin{equation}
    \mathbb{E} \left\|\frac{1}{n}H^{-1}\sum_{t=1}^n \xi_t \right\|^2
    \;\leq\; \frac{1}{n}\,\mathrm{Tr}(H^{-1}\overline{S}H^{-1}).
\end{equation}
In particular, $\mathbb{E}\|\Xi_n\|^2 = \mathcal{O}(n^{-1})$.  When
$\overline{S} = S$, the right-hand side equals
$\frac{1}{n}\mathrm{Tr}(H^{-1}SH^{-1})$, the $n^{-1}$-scaled
trace of the classical Polyak--Ruppert asymptotic covariance
$H^{-1}SH^{-1}$~\citep{polyak1992acceleration}.
\end{proposition}

\begin{proof}
Deferred to the supplementary material.
\end{proof}

\begin{proposition}[Negligibility of the Taylor Remainder]
\label{prop:taylor}
Under Assumptions~\ref{ass:convex}--\ref{ass:iterate},
\begin{equation}
    \mathbb{E}\|T_n\| = \mathcal{O}(n^{-\alpha}).
\end{equation}
Consequently, $\sqrt{n}\,T_n \to 0$ in $L^1$.
If, in addition, there exists a constant $C_{\Delta,4}>0$ such that
\[
    \mathbb{E}\|\Delta_t\|^4 \le C_{\Delta,4}\,t^{-2\alpha}
    \qquad \text{for all } t \ge 1,
\]
then
\begin{equation}
    \mathbb{E}\|T_n\|^2 = \mathcal{O}(n^{-2\alpha}),
\end{equation}
and, hence, $\sqrt{n}\,T_n \to 0$ in $L^2$.
\end{proposition}

\begin{proof}
Deferred to the supplementary material.
\end{proof}

\subsection{Stabilization Threshold}
\label{sec:remainder}

With $\Xi_n$ and $T_n$ controlled by Propositions~\ref{prop:variance}
and~\ref{prop:taylor}, the dynamic remainder $R_n(\{P_t\})$ is analyzed
below.  Define
\[
    M_t := (P_t H)^{-1}.
\]
The matrix $A_t$ from Lemma~\ref{lem:decomp} can be written as
\[
    A_t = \eta_t^{-1} M_t,
\]
so all explicit dependence of the dynamic remainder on the preconditioner
sequence enters through~$M_t$ (the iterates $\Delta_t$ may still depend
implicitly on the preconditioned trajectory).  The increment $A_t-A_{t-1}$
appearing in $R_n(\{P_t\})$ decomposes as
\begin{equation} \label{eq:coupled_diff}
    A_t - A_{t-1}
    = (\eta_t^{-1} - \eta_{t-1}^{-1})M_t + \eta_{t-1}^{-1}(M_t - M_{t-1}).
\end{equation}
The step-size increment
$\eta_t^{-1} - \eta_{t-1}^{-1} = \eta_0^{-1}(t^\alpha - (t-1)^\alpha)
= O(t^{\alpha-1})$ is determined entirely by the schedule.  For a fixed
preconditioner, $M_t - M_{t-1} = 0$ and only the step-size term
in~\eqref{eq:coupled_diff} survives.  Under dynamic preconditioning, the
term $\eta_{t-1}^{-1}(M_t - M_{t-1})$ is also present; its sole new
ingredient is the matrix increment $M_t - M_{t-1}$, whose decay we
quantify through the following stabilization-rate condition.

\begin{definition}[Stabilization Rate]
\label{def:beta}
We say that a preconditioner sequence $\{P_t\}$ \emph{admits stabilization
rate} $\beta > 0$ if there exists a deterministic constant $C_M > 0$ such
that $\|M_t - M_{t-1}\|_{\mathrm{op}} \leq C_M\, t^{-\beta}$ almost surely
for all $t \geq 2$, where $M_t = (P_tH)^{-1}$.
\end{definition}

Theorem~\ref{thm:sharp_threshold} below establishes the sufficiency
direction of the threshold analysis; the matching tightness statement
appears as Proposition~\ref{thm:sharpness}.

\begin{theorem}[Stabilization Threshold (Sufficiency)]
\label{thm:sharp_threshold}
Under Assumptions~\ref{ass:mg}--\ref{ass:iterate}, if $\{P_t\}$ admits
stabilization rate $\beta$ and there exists a deterministic constant
$C_P < \infty$ such that
\[
    \|M_t\|_{\mathrm{op}} \le C_P
    \qquad \text{almost surely, for all } t \ge 1,
\]
then $\sqrt{n}\,R_n(\{P_t\}) \to 0$ in $L^2$ whenever
\begin{equation} \label{eq:beta_threshold}
    \beta \;>\; \frac{\alpha+1}{2}.
\end{equation}
Since $\alpha \in (1/2, 1)$, this threshold lies in $(3/4, 1)$.
\end{theorem}

\begin{proof}
Deferred to the supplementary material.
\end{proof}

\subsubsection*{A saturating construction for the sufficiency hypotheses}
The following proposition exhibits a deterministic construction that
saturates the hypotheses of Theorem~\ref{thm:sharp_threshold} while
$\sqrt{n}\,R_n(\{P_t\}) \not\to 0$ in $L^2$;
Remark~\ref{rem:sharpness_scope} below delineates the precise scope.

\begin{proposition}[Saturating Construction for the Sufficiency Hypotheses]
\label{thm:sharpness}
Fix $\alpha \in (1/2,1)$ and let $0<\beta \le (\alpha+1)/2$.  Then there exist
deterministic scalar sequences
$\{M_t\}_{t\ge1}\subset(0,\infty)$ and $\{\Delta_t\}_{t\ge1}\subset\mathbb R$
with $d=1$, $H=1$, and $P_t := M_t^{-1}$, such that
\[
\|M_t-M_{t-1}\|_{\mathrm{op}} = \Theta(t^{-\beta}),\qquad
\|M_t\|_{\mathrm{op}} = \Theta(1),\qquad
\|\Delta_t\|_{L^2} = \Theta(t^{-\alpha/2}),
\]
while
\[
\sqrt{n}\,\|R_n(\{P_t\})\|_{L^2}
=
\begin{cases}
\Theta\!\left(n^{(\alpha+1)/2-\beta}\right),
& 0<\beta<(\alpha+1)/2,\\[1mm]
\Theta(1),
& \beta=(\alpha+1)/2.
\end{cases}
\]
\end{proposition}

\begin{proof}
Deferred to the supplementary material.
\end{proof}

\begin{remark}[Scope of the tightness result]
\label{rem:sharpness_scope}
Proposition~\ref{thm:sharpness} is a tightness statement about the
\emph{hypotheses} of Theorem~\ref{thm:sharp_threshold}, not a lower
bound for every adaptive algorithm.  The construction produces
deterministic scalar sequences $\{M_t\}$ and $\{\Delta_t\}$ that attain
the boundedness and polynomial-decay hypotheses of
Theorem~\ref{thm:sharp_threshold} with equality, so no argument using
only these hypotheses can lower the threshold $(\alpha+1)/2$.
Crucially, the construction is \emph{decoupled}: $\{M_t\}$ is prescribed
independently of $\{\Delta_t\}$, whereas an adaptive recursion couples
the two through the dynamics.  Whether this coupling can be exploited
to lower the threshold below $(\alpha+1)/2$ for some subclass of
adaptive algorithms is an open question; any such improvement would
require structural assumptions beyond the rate hypotheses considered
here.  Elsewhere in the paper, ``the threshold $(\alpha+1)/2$ is
tight'' refers to the statement given here.
\end{remark}

\section{SA Preconditioner Constructions and Stabilization}
\label{sec:construction}

We now introduce three concrete preconditioner constructions that will be
analyzed via the stabilization theorem (Theorem~\ref{thm:sa_stab}) of
Section~\ref{sec:stab}.  In each construction, we replace
the cumulative-sum or constant-EMA updates of classical AdaGrad, RMSProp,
and ONS by SA running averages with gain $\rho_t = c/t$, $0 < c \le 1$.
Under the assumptions stated below, this yields
$\|M_t - M_{t-1}\|_{L^2(\mathrm{op})}
:= (\mathbb{E}\|M_t - M_{t-1}\|_{\mathrm{op}}^2)^{1/2}
= \mathcal{O}(t^{-1})$, which
upgrades to the pathwise rate
$\|M_t - M_{t-1}\|_{\mathrm{op}} = \mathcal{O}(t^{-1})$---i.e.,
stabilization rate $\beta = 1$---when the SA inputs are almost surely bounded.

\subsection{Stochastic Approximation Full-Matrix AdaGrad}
\label{sec:sa_adagrad}

We maintain a running average $C_t$ of the gradient outer products:
\begin{equation} \label{eq:sa_adagrad}
    C_t = (1 - \rho_t) C_{t-1} + \rho_t \left(g_t g_t^\top + \epsilon I\right),
    \quad P_t = C_{t-1}^{-1/2},
\end{equation}
where $g_t := \nabla f(x_t,\zeta_t)$, $\rho_t = c/t$ with $0 < c \le 1$,
$C_0 = \epsilon I$, and $\epsilon > 0$.  Induction gives
$C_t \succeq \epsilon I$, so $P_t$ is well defined and
$\mathcal{F}_{t-1}$-measurable.
When $c=1$, $C_t = t^{-1}\sum_{s=1}^t(g_sg_s^\top + \epsilon I)$, the
time-normalized classical AdaGrad accumulator of Duchi
et~al.~\citep{duchi2011adaptive} with a fixed $\epsilon I$ regularization.

\subsection{Stochastic Approximation RMSProp}
\label{sec:sa_rmsprop}

To obtain a diagonal adaptive rule, we track only the coordinatewise gradient
second moments:
\begin{equation} \label{eq:sa_rmsprop}
    v_t = (1 - \rho_t) v_{t-1} + \rho_t \left(g_t \odot g_t + \epsilon \mathbf{1}\right),
    \qquad
    P_t = \mathrm{Diag}(v_{t-1})^{-1/2},
\end{equation}
where $\odot$ denotes the Hadamard product, $\rho_t = c/t$ with
$0 < c \le 1$, and $v_0 = \epsilon \mathbf{1}$.
Write $s_{\mathrm{diag}} := \mathrm{diag}(S) \in \mathbb{R}^d$ for the
diagonal of the noise covariance at~$x^*$.  Heuristically, once the iterates
concentrate near~$x^*$, the recursion drives $v_t$ toward
$s_{\mathrm{diag}} + \epsilon \mathbf{1}$; this limit is made rigorous in
Proposition~\ref{prop:comparison}, whose hypotheses include continuity of
the state-dependent noise covariance
$\Sigma(x) := \mathbb{E}[\xi(x,\zeta)\xi(x,\zeta)^\top]$ at~$x^*$ and are
not in force at this stage.
As in SA-AdaGrad, $v_t \ge \epsilon\mathbf{1}$ componentwise, so $P_t$ is
well defined and $\mathcal{F}_{t-1}$-measurable.
The classical RMSProp \citep{tieleman2012rmsprop} uses a fixed decay
$\gamma$; our SA variant replaces this with the SA gain $\rho_t = c/t$,
which is the schedule required by Theorem~\ref{thm:sa_stab} to deliver
the $L^2$ one-step stabilization bound
$\|M_t - M_{t-1}\|_{L^2(\mathrm{op})} = \mathcal{O}(t^{-1})$, upgrading
to the pathwise rate $\beta = 1$ under a.s.\ bounded gradients.

\subsection{Stochastic Approximation Online Newton Step}
\label{sec:sa_ons}

We maintain a recursive estimate $B_t$ of the target Hessian $H$
using symmetric noisy Hessian estimates $\hat{H}_t$
\citep{bercu2020efficient,godichon2024hessian}:
\begin{equation} \label{eq:sa_ons}
    B_t = (1 - \rho_t) B_{t-1} + \rho_t \hat{H}_t, \quad P_t = B_{t-1}^{-1},
\end{equation}
where $\rho_t = c/t$ with $0 < c \le 1$ and $B_0 = \epsilon I$; here
$\epsilon \in [h_-,h_+]$ is required for the uniform-ellipticity bound
of Corollary~\ref{cor:ons_iterate}.
Under $h_- I \preceq \hat{H}_t \preceq h_+ I$ a.s., $B_t$ remains
positive definite, so $P_t = B_{t-1}^{-1}$ is well defined and
$\mathcal{F}_{t-1}$-measurable.
Theorem~\ref{thm:sa_stab} delivers the $L^2$ one-step stabilization bound
$\|M_t - M_{t-1}\|_{L^2(\mathrm{op})} = \mathcal{O}(t^{-1})$, upgrading to
the pathwise rate $\beta = 1$ via Theorem~\ref{thm:sa_stab}(iv), since
the a.s.\ bound $\|\hat{H}_t\|_{\mathrm{op}} \le h_+$ supplied above
applies.
If $\mathbb{E}[\hat{H}_t \mid \mathcal{F}_{t-1}] = H$ and
$\sup_t \mathbb{E}\|\hat{H}_t\|_{\mathrm{op}}^2 < \infty$, then
$B_t \xrightarrow{a.s.} H$ and
$P_t H \xrightarrow{a.s.} I$.

\subsection{Verification of the Iterate Bound under Primitive Conditions}
\label{sec:iterate_verification}

Proposition~\ref{prop:iterate_bound} reduces the iterate bound
(Assumption~\ref{ass:iterate}) to two conditions: (i)~$L$-smoothness
of~$F$ on a deterministic convex set $\mathcal{K}$ containing all
iterates almost surely, and (ii)~uniform ellipticity of~$P_t$.  The
following corollaries verify~(ii) constructively for each SA
construction under bounded gradients (for SA-AdaGrad and SA-RMSProp)
or a two-sided spectral bound on the Hessian estimates (for SA-ONS);
condition~(i)---both the $L$-smoothness of~$F$ and the
iterate-containment property---is assumed via hypothesis~(b) in each
corollary.

\begin{corollary}[Primitive conditions for SA-AdaGrad and SA-RMSProp]
\label{cor:adaptive_iterate}
Under Assumptions~\ref{ass:mg}--\ref{ass:convex}, consider either the SA-AdaGrad update \eqref{eq:sa_adagrad} or the
SA-RMSProp update \eqref{eq:sa_rmsprop} with $\rho_t = c/t$ for some
$0 < c \le 1$, and initialization $C_0 = \epsilon I$ in the AdaGrad case and
$v_0 = \epsilon \mathbf{1}$ in the RMSProp case.  Suppose:
\begin{enumerate}
  \item[(a)] the stochastic gradients satisfy $\|g_t\| \le G$ almost surely for
    all $t \ge 1$;
  \item[(b)] $F$ is $L$-smooth on a deterministic
    convex set $\mathcal{K}$ containing $x^*$ and all iterates almost surely.
\end{enumerate}
Then, for both preconditioners:
\begin{enumerate}
  \item[\textnormal{(i)}] the preconditioner is uniformly elliptic with
    \[
        \frac{1}{\sqrt{G^2+\epsilon}}\,I
        \;\preceq\;
        P_t
        \;\preceq\;
        \frac{1}{\sqrt{\epsilon}}\,I
        \qquad \text{almost surely for all } t \ge 1;
    \]
  \item[\textnormal{(ii)}] Assumption~\ref{ass:iterate} holds with
    $\mathbb{E}\|\Delta_t\|^2 \le C_\Delta\, t^{-\alpha}$.
\end{enumerate}
\end{corollary}

\begin{proof}
Deferred to the supplementary material.
\end{proof}

The analogous result for SA-ONS replaces the bounded-gradient condition with
a two-sided spectral bound $h_- I \preceq \hat{H}_t \preceq h_+ I$ on the
Hessian estimates and uses the matrix inverse rather than the inverse square
root.

\begin{corollary}[Primitive conditions for SA-ONS]
\label{cor:ons_iterate}
Under Assumptions~\ref{ass:mg}--\ref{ass:convex}, consider the SA-ONS update
\eqref{eq:sa_ons} with $\rho_t = c/t$ for some $0 < c \le 1$ and
$B_0 = \epsilon I$.  Suppose:
\begin{enumerate}
  \item[(a)] the stochastic Hessian estimates are symmetric and satisfy
    $h_- I \preceq \hat{H}_t \preceq h_+ I$ almost surely for all $t \geq 1$
    and some deterministic constants $0 < h_- \leq h_+$, and the initialization
    is spectrally compatible: $\epsilon \in [h_-,h_+]$;
  \item[(b)] $F$ is $L$-smooth on a deterministic
    convex set $\mathcal{K}$ containing $x^*$ and all iterates almost surely.
\end{enumerate}
Then:
\begin{enumerate}
  \item[\textnormal{(i)}] the Hessian-estimate recursion satisfies
    $h_- I \preceq B_t \preceq h_+ I$ almost surely for every $t \geq 0$, so
    $P_t = B_{t-1}^{-1}$ is uniformly elliptic with
    \[
        h_+^{-1} I \;\preceq\; P_t \;\preceq\; h_-^{-1} I;
    \]
  \item[\textnormal{(ii)}] Assumption~\ref{ass:iterate} holds with
    $\mathbb{E}\|\Delta_t\|^2 \leq C_\Delta\, t^{-\alpha}$.
\end{enumerate}
\end{corollary}

\begin{proof}
Deferred to the supplementary material.
\end{proof}

\noindent
Condition~(a) can always be enforced by spectrally clipping any symmetric
Hessian estimator to the interval $[h_-,h_+]$ via eigenvalue thresholding.
It also holds automatically for Tikhonov-regularized minibatch Hessians in
generalized linear models whenever the covariates are bounded and the
link-function curvature is bounded on the iterate domain, with~$h_-$
determined by the regularization strength and~$h_+$ by the uniform
curvature bound.

With the iterate bound established for all three
constructions (under bounded inputs and hypothesis~(b)),
the remaining question is whether the preconditioners
stabilize fast enough---that is, whether the rate $\beta$ in
Definition~\ref{def:beta} exceeds the threshold $(\alpha+1)/2$ identified
in Section~\ref{sec:remainder}.  The next subsection answers this
affirmatively.

\subsection{Stabilization Analysis of the SA Preconditioners}
\label{sec:stab}

Having introduced the SA preconditioners, we next study whether they stabilize
rapidly enough for the dynamic remainder in Lemma~\ref{lem:decomp} to be
negligible.  For the SA matrix recursions considered below, the gain schedule
$\rho_t = c/t$ yields an $L^2$ stabilization bound of order $t^{-1}$ for the
one-step increments of the underlying recursion, namely
$\|Q_t - Q_{t-1}\|_{L^2(\mathrm{op})} = \mathcal{O}(t^{-1})$, where $Q_t$
denotes the generic SA running average introduced in
Theorem~\ref{thm:sa_stab} below.
Under the uniform ellipticity conditions established in
Section~\ref{sec:construction}, a Lipschitz spectral map relating $Q_t$ to
$M_t = (P_tH)^{-1}$ transfers this rate to the effective inverse drift
matrix~$M_t$.  The following assumption and theorem formalize this
mechanism for a generic SA matrix recursion.

\begin{assumption}[Uniform Fourth Moment for the Gradient-Driven SA Inputs]
\label{ass:g4}
For the SA-AdaGrad and SA-RMSProp constructions, there exists a constant
$C_{g,4} > 0$ such that
\[
    \mathbb{E}\|g_t\|^4 \le C_{g,4}
    \qquad \text{for all } t \ge 1.
\]
\end{assumption}

\noindent
A uniform fourth-moment (or bounded-gradient) hypothesis on the
gradient driving the SA recursion is a standard moment condition that
appears in various forms throughout the adaptive/preconditioned SA
literature
\citep{duchi2011adaptive,bercu2020efficient,boyer2023stochastic,%
godichon2024adagrad,godichon2024hessian}; analogous bounds are used
there to control the fluctuations of the running averages that arise in
related matrix recursions, and the condition holds automatically
whenever $\|g_t\| \le G$ a.s.

\begin{theorem}[Stabilization of SA Matrix Recursions under Spectral Maps]
\label{thm:sa_stab}
Let $\{Q_t\}_{t \ge 0}$ be an SA recursion of the form
\begin{equation}\label{eq:sa_general_stab}
    Q_t = (1-\rho_t)\,Q_{t-1} + \rho_t\,W_t, \qquad \rho_t = c/t,
\end{equation}
taking values in $\mathbb{S}_{++}^d$ (or, in the diagonal-vector case,
in $\mathbb{R}_{++}^d$), where $0 < c \le 1$, $Q_0 \succeq \epsilon I$ is
deterministic, and the driving term $W_t$ is $\mathcal{F}_t$-measurable
and satisfies $W_t \succeq \epsilon I$ almost surely and
$\sup_{t}\|W_t\|_{L^2(\mathrm{op})} \le C_W < \infty$.
Suppose the preconditioner is $P_t = \varphi(Q_{t-1})$, where
$\varphi \colon \mathbb{S}_{++}^d \to \mathbb{S}_{++}^d$ is either
(a)~the matrix inverse $\varphi(A) = A^{-1}$, or
(b)~the inverse square root $\varphi(A) = A^{-1/2}$.
In the diagonal-vector case, all statements are interpreted coordinatewise
after identifying a positive vector with the corresponding diagonal matrix.
Define $M_t = (P_t H)^{-1}$ (well defined since $H \succ 0$ by Assumption~\ref{ass:convex}).

Then:
\begin{enumerate}
  \item[\textnormal{(i)}]
    $Q_t \succeq \epsilon I$ almost surely for all $t$,
    $\sup_t \|Q_t\|_{L^2(\mathrm{op})} < \infty$.
  \item[\textnormal{(ii)}]
    The one-step increment bound satisfies
    $\|Q_t - Q_{t-1}\|_{L^2(\mathrm{op})} = \mathcal{O}(t^{-1})$.
  \item[\textnormal{(iii)}]
    For either $\varphi = A^{-1}$ (used by SA-ONS) or
    $\varphi = A^{-1/2}$ (used by SA-AdaGrad and SA-RMSProp),
    $\|M_t - M_{t-1}\|_{L^2(\mathrm{op})} = \mathcal{O}(t^{-1})$ and
    $\|M_t\|_{L^2(\mathrm{op})} = \mathcal{O}(1)$.
  \item[\textnormal{(iv)}]
    If $\|W_t\|_{\mathrm{op}} \le C_{\max}$ almost surely, then
    all bounds in (ii)--(iii) hold with $\|\cdot\|_{L^2(\mathrm{op})}$ replaced by
    $\|\cdot\|_{\mathrm{op}}$, so the pathwise stabilization rate
    (Definition~\ref{def:beta}) is $\beta = 1$.
\end{enumerate}
If, in addition, there exists a deterministic matrix
$\bar W \in \mathbb{S}_{++}^d$ such that
\[
    \mathbb{E}[W_t \mid \mathcal{F}_{t-1}] = \bar W
    \qquad \text{almost surely for all } t \ge 1,
\]
then $Q_t \xrightarrow{a.s.} \bar W$.  Writing
$P_\infty := \varphi(\bar W)$ and $M_\infty := (P_\infty H)^{-1}$, we also
have $M_t \xrightarrow{a.s.} M_\infty$.
\end{theorem}

\begin{proof}
Deferred to the supplementary material.
\end{proof}

Theorem~\ref{thm:sa_stab} is stated for a generic SA recursion.  To apply
it, one simply identifies the recursion variable $Q_t$, the driving term $W_t$,
and the map $\varphi$ for each of the three constructions introduced in
Sections~\ref{sec:sa_adagrad}--\ref{sec:sa_ons}.  The following corollary
records the result.

\begin{corollary}[Stabilization of All Three SA Preconditioners]
\label{cor:sa_stab}
Each of the three SA preconditioners satisfies
$\|M_t - M_{t-1}\|_{L^2(\mathrm{op})} = \mathcal{O}(t^{-1})$ and
$\|M_t\|_{L^2(\mathrm{op})} = \mathcal{O}(1)$, under the following
conditions:
\begin{itemize}
  \item \textbf{SA-AdaGrad} \eqref{eq:sa_adagrad}: Assumptions~\ref{ass:convex}
    and~\ref{ass:g4}, with $Q_t = C_t$, $W_t = g_tg_t^\top + \epsilon I$,
    $\varphi = A^{-1/2}$.
  \item \textbf{SA-RMSProp} \eqref{eq:sa_rmsprop}: same assumptions, applied coordinatewise
    with $Q_t = \mathrm{Diag}(v_t)$, $W_t = \mathrm{Diag}(g_t \odot g_t + \epsilon \mathbf{1})$, $\varphi = A^{-1/2}$.
  \item \textbf{SA-ONS} \eqref{eq:sa_ons}: Assumption~\ref{ass:convex} with
    $\hat{H}_t \succeq \epsilon I$ a.s.,
    $\mathbb{E}[\hat{H}_t \mid \mathcal{F}_{t-1}] = H$,
    $\sup_{t}\mathbb{E}\|\hat{H}_t\|_{\mathrm{op}}^2 \le C_H$;
    here $Q_t = B_t$, $W_t = \hat{H}_t$, $\varphi = A^{-1}$, and
    $M_t \xrightarrow{a.s.} I$.
\end{itemize}
If, in addition, the SA inputs are almost surely bounded
($\|g_t\| \le G$ for AdaGrad/RMSProp, or
$\|\hat{H}_t\|_{\mathrm{op}} \le h_+$ for ONS), then
Theorem~\ref{thm:sa_stab}(iv) yields the pathwise stabilization rate
$\beta = 1$.
\end{corollary}

\begin{proof}
Deferred to the supplementary material.
\end{proof}

Corollary~\ref{cor:sa_stab} shows that each SA preconditioner achieves
$L^2$ one-step stabilization of order $t^{-1}$, which upgrades to the
pathwise rate $\beta = 1$ under bounded inputs---exceeding the threshold
$(\alpha+1)/2$ of Theorem~\ref{thm:sharp_threshold} for every
$\alpha \in (1/2,1)$.  The next corollary combines the one-step bounds
with the step-size schedule to control the coupled quantity
$\eta_t^{-1}M_t - \eta_{t-1}^{-1}M_{t-1}$ needed for the finite-sample
remainder bound (Proposition~\ref{prop:tight}).

\begin{remark}[Generic features of the stabilization argument]
\label{rem:generic_scope}
The proof of Theorem~\ref{thm:sa_stab} uses only the recursion form
$Q_t = (1-\rho_t)Q_{t-1} + \rho_t W_t$ with $\rho_t = c/t$, the
spectral maps $\varphi \in \{A\mapsto A^{-1},\; A\mapsto A^{-1/2}\}$,
and the ellipticity and moment hypotheses on $W_t$ stated in
Theorem~\ref{thm:sa_stab}; it makes no further structural use of the
particular adaptive method from which $Q_t$ arises.
\end{remark}

\begin{corollary}[Coupled One-Step Variation of $\eta_t^{-1}M_t$]
\label{cor:threshold}
Under the respective hypotheses of Corollary~\ref{cor:sa_stab}, each of the three SA preconditioners satisfies
\begin{align*}
    \bigl\|\eta_t^{-1}M_t - \eta_{t-1}^{-1}M_{t-1}\bigr\|_{L^2(\mathrm{op})}
    &\le
    \bigl|\eta_t^{-1} - \eta_{t-1}^{-1}\bigr|\,\|M_t\|_{L^2(\mathrm{op})}\\
    &\quad + \eta_{t-1}^{-1}\|M_t - M_{t-1}\|_{L^2(\mathrm{op})},
\end{align*}
and the right-hand side is $\mathcal{O}(t^{\alpha-1})$.
Under the additional almost-sure boundedness assumptions in
Corollary~\ref{cor:sa_stab}, the bound
$\|M_t - M_{t-1}\|_{\mathrm{op}} = \mathcal{O}(t^{-1})$ holds pathwise, so
$M_t$ attains the stabilization rate $\beta=1$ in the sense of
Definition~\ref{def:beta}.
\end{corollary}

\begin{proof}
Deferred to the supplementary material.
\end{proof}

The one-step stabilization of Corollary~\ref{cor:sa_stab} and the
coupled variation $\mathcal{O}(t^{\alpha-1})$ of
Corollary~\ref{cor:threshold} feed into the analysis of
Section~\ref{sec:main} through two pathways:
\begin{itemize}
  \item \emph{$L^2$ pathway (no bounded-input assumption).}
    The $L^2(\mathrm{op})$ bounds from Corollaries~\ref{cor:sa_stab}
    and~\ref{cor:threshold} suffice to prove
    $\sqrt{n}\,R_n \to 0$ in $L^1$ via a Cauchy--Schwarz argument
    (Proposition~\ref{prop:l2_remainder} below), which is enough for the CLT
    (Theorem~\ref{thm:clt}).
  \item \emph{Pathwise pathway (bounded inputs).}
    Under the additional bounded-input conditions in
    Corollary~\ref{cor:sa_stab} ($\|g_t\| \le G$ or
    $\|\hat{H}_t\|_{\mathrm{op}} \le h_+$), the $L^2$ bounds upgrade
    to pathwise (a.s.)\ bounds, which feed into
    Proposition~\ref{prop:tight} and yield the explicit finite-sample
    constant~$C_R$.  The operator-factor comparison
    (Proposition~\ref{prop:comparison}) then builds on these pathwise
    bounds together with additional construction-specific hypotheses,
    including continuity of the noise covariance $\Sigma(x)$ at~$x^*$.
\end{itemize}

\section{Finite-Sample and Asymptotic Analysis}
\label{sec:main}

\subsection{Finite-Sample Control of the Dynamic Remainder}
\label{sec:tight}

This section propagates the one-step stabilization bounds of
Section~\ref{sec:stab} through the exact decomposition of
Lemma~\ref{lem:decomp} to control $\sqrt{n}\,R_n$.  The analysis
splits into two tiers with distinct downstream purposes, which we
state upfront.

\begin{itemize}
  \item \textbf{Default tier ($L^2$ pathway).}
    Proposition~\ref{prop:l2_remainder} establishes
    $\mathbb{E}\|\sqrt{n}\,R_n\| = \mathcal{O}(n^{(\alpha-1)/2}) \to 0$
    using only the $L^2(\mathrm{op})$ stabilization bounds of
    Corollaries~\ref{cor:sa_stab}--\ref{cor:threshold}.  \emph{No
    bounded-input hypothesis is required.}  This is the route used by
    Theorem~\ref{thm:clt} (CLT) and it covers all three SA
    preconditioners, including in settings with Gaussian covariates or
    other unbounded inputs, provided the SA-input conditions of each
    construction in Corollary~\ref{cor:sa_stab} hold
    (Assumption~\ref{ass:g4} for SA-AdaGrad and SA-RMSProp; the
    ellipticity floor $\hat{H}_t \succeq \epsilon I$, the
    conditional-mean constancy
    $\mathbb{E}[\hat{H}_t \mid \mathcal{F}_{t-1}] = H$, and the
    second-moment bound
    $\sup_t \mathbb{E}\|\hat{H}_t\|_{\mathrm{op}}^2 \le C_H$ for
    SA-ONS).
  \item \textbf{Refinement tier (pathwise hypotheses).}
    Proposition~\ref{prop:tight} is a \emph{separate, conditional}
    strengthening of Proposition~\ref{prop:l2_remainder}: under the
    additional \emph{pathwise} (a.s.)\ bounded-input hypotheses
    ($\|g_t\| \le G$ for SA-AdaGrad and SA-RMSProp; the two-sided
    spectral bound $h_- I \preceq \hat{H}_t \preceq h_+ I$ for SA-ONS,
    which deterministically implies both the ellipticity floor
    $\hat{H}_t \succeq h_- I$ and the operator-norm bound
    $\|\hat{H}_t\|_{\mathrm{op}} \le h_+$, and, thus, supersedes the
    second-moment bound of the default tier), which upgrade the
    $L^2(\mathrm{op})$ stabilization of
    Corollary~\ref{cor:sa_stab} to pathwise operator bounds and supply
    the pathwise ellipticity of $P_t$ through
    Corollaries~\ref{cor:adaptive_iterate}
    and~\ref{cor:ons_iterate}, it upgrades the first-moment
    $\mathcal{O}(n^{(\alpha-1)/2})$ bound to
    the $L^2$ bound
    $\mathbb{E}\|\sqrt{n}\,R_n\|^2 = \mathcal{O}(n^{\alpha-1})$
    with an \emph{explicit finite-sample constant}~$C_R$, an
    $L^1\!\to\! L^2$ moment upgrade at the same effective rate.
    The constant $C_R$ drives the operator-factor comparison
    across preconditioners (Proposition~\ref{prop:comparison}) and the
    quantitative Wasserstein bound (Corollary~\ref{cor:gauss_approx}).
    Neither of these downstream consequences is needed for the CLT.
\end{itemize}
In particular, when the bounded-input hypotheses fail (e.g., Gaussian
covariates in Section~\ref{sec:simulation}), the default tier still
delivers the CLT; only the refinement tier's explicit constant $C_R$
and its two downstream consequences (operator-factor comparison and
the Wasserstein bound) are unavailable.

\begin{proposition}[$L^2$ Pathway: First-Moment Control of $\sqrt{n}\,R_n$]
\label{prop:l2_remainder}
Under Assumptions~\ref{ass:mg}--\ref{ass:iterate}, suppose the effective
inverse drift matrix satisfies the $L^2$ stabilization bounds
\[
    \|M_t\|_{L^2(\mathrm{op})} = \mathcal{O}(1),
    \qquad
    \bigl\|\eta_t^{-1}M_t - \eta_{t-1}^{-1}M_{t-1}\bigr\|_{L^2(\mathrm{op})}
    = \mathcal{O}(t^{\alpha-1}).
\]
Then
\[
    \mathbb{E}\bigl\|\sqrt{n}\,R_n(\{P_t\})\bigr\|
    \;=\;
    \mathcal{O}\!\left(n^{(\alpha-1)/2}\right)
    \;\longrightarrow\; 0
    \qquad \text{as } n \to \infty,
\]
so $\sqrt{n}\,R_n \to 0$ in $L^1$ and, hence, in probability.  The
hypotheses hold for all three SA preconditioners of
Section~\ref{sec:construction} under the $L^2$ conditions of
Corollaries~\ref{cor:sa_stab}--\ref{cor:threshold}, with no
bounded-input assumption.
\end{proposition}

\begin{proof}
Deferred to the supplementary material.
\end{proof}

\begin{proposition}[Finite-Sample $L^2$ Bound under Pathwise Hypotheses]
\label{prop:tight}
Under Assumptions~\ref{ass:mg}--\ref{ass:iterate}, suppose an SA construction
satisfies
\[
    \|M_t\|_{\mathrm{op}} \leq C_P
    \qquad \text{almost surely for all } t \ge 1,
\]
and there exists an index $t_0 \ge 2$ such that
\[
    \bigl\|\eta_t^{-1}M_t - \eta_{t-1}^{-1}M_{t-1}\bigr\|_{\mathrm{op}}
    \leq C_{\mathrm{coup}}\, t^{\alpha-1}
\]
almost surely for all $t \ge t_0$.  Then there exists a constant
$C_R > 0$, depending on
$\eta_0$, $\alpha$, $C_\Delta$, $C_P$, $C_{\mathrm{coup}}$, $t_0$,
such that
\begin{equation}
    \mathbb{E}\|R_n(\{P_t\})\|^2 \;\leq\; \frac{C_R}{n^{2-\alpha}}.
\end{equation}
Equivalently, $\|R_n\|_{L^2} = \mathcal{O}(n^{\alpha/2 - 1})$.
Consequently,
\[
    \sqrt{n}\,\|R_n\|_{L^2}
    =
    \mathcal{O}(n^{(\alpha-1)/2})
    \to 0
    \qquad \text{for all } \alpha < 1.
\]
\end{proposition}

\begin{proof}
Deferred to the supplementary material.
\end{proof}

The pathwise hypotheses of Proposition~\ref{prop:tight} are verified
for each of the three SA constructions under the bounded-input
conditions ($\|g_t\|\le G$ for SA-AdaGrad and SA-RMSProp;
$h_- I \preceq \hat{H}_t \preceq h_+ I$ for SA-ONS): the uniform
pathwise bound $\|M_t\|_{\mathrm{op}}\le C_P$ follows from the pathwise
ellipticity of $P_t$ supplied by
Corollaries~\ref{cor:adaptive_iterate} and~\ref{cor:ons_iterate}, and
the pathwise coupled-increment rate
$\|\eta_t^{-1}M_t-\eta_{t-1}^{-1}M_{t-1}\|_{\mathrm{op}}
=\mathcal{O}(t^{\alpha-1})$ a.s.\ follows by combining that uniform
bound with the pathwise stabilization rate $\beta=1$ of
Corollary~\ref{cor:sa_stab} through the triangle inequality of
Corollary~\ref{cor:threshold}.  In both tiers the dynamic remainder is
of lower order than the $n^{-1/2}$ martingale term, so the dynamic
preconditioning contributes only an asymptotically negligible
correction under the $\sqrt{n}$ normalization.  The next result makes
the preconditioner dependence of~$C_R$ explicit.

\begin{proposition}[Operator-Factor Comparison across Preconditioner Constructions]
\label{prop:comparison}
Assume Assumptions~\ref{ass:mg}--\ref{ass:iterate}, the pathwise
containment hypothesis $x_t \in \mathcal{K}$ for all~$t$, and continuity
of the state-dependent noise covariance
$\Sigma(x) := \mathbb{E}[\xi(x,\zeta)\xi(x,\zeta)^\top]$ at~$x^*$.  For
each SA construction, assume the full hypotheses~(a)--(b) of the
respective primitive corollary (which subsume both the $L$-smoothness
of~$F$ on~$\mathcal{K}$ and the iterate containment
$x_t\in\mathcal{K}$), and additionally for SA-ONS the conditional-mean
constancy $\mathbb{E}[\hat{H}_t \mid \mathcal{F}_{t-1}] = H$ a.s.  Then
Theorem~\ref{thm:sa_stab}(iv) yields the pathwise stabilization rate
$\beta=1$ in each case.  Moreover, $M_t \to M_\infty$ a.s.\@ for each of
the three SA preconditioner constructions: for SA-ONS this follows from
the final clause of Theorem~\ref{thm:sa_stab} (with $\bar W = H$), while
for SA-AdaGrad and SA-RMSProp it is established by a direct pathwise
argument given in the supplement.  The uniform operator bound on
$\|M_t\|_{\mathrm{op}}$ follows from the pathwise uniform ellipticity
of $P_t$ supplied by the respective primitive corollaries
(Corollaries~\ref{cor:adaptive_iterate} and~\ref{cor:ons_iterate})
combined with $H\succ 0$, via $M_t = (P_tH)^{-1}$.
Proposition~\ref{prop:tight}, therefore, applies, yielding the
finite-sample constant $C_R$ with some uniform pathwise bound $C_P$.
Since $M_t \to M_\infty$ almost surely in each construction, we
compare the constructions at the level of the \emph{asymptotic
operator factor}
$\bar C_P := \|M_\infty\|_{\mathrm{op}}$---the $t\to\infty$ limit of
$\|M_t\|_{\mathrm{op}}$, whose pathwise uniform bound supplies the
constant $C_P$ in Proposition~\ref{prop:tight}---which takes the
following explicit form in each case:
\begin{itemize}
  \item \textbf{Identity (PR-SGD)}: $P_t \equiv I$, so $M_t = H^{-1}$ and
    \[
      \bar C_P^{\mathrm{Id}}
      :=
      \|H^{-1}\|_{\mathrm{op}}
      =
      \lambda_{\min}(H)^{-1}
      =
      \frac{\kappa(H)}{\lambda_{\max}(H)}.
    \]
  \item \textbf{SA-ONS}: $M_t = H^{-1}B_{t-1} \to I$, so
    \[
      \bar C_P^{\mathrm{ONS}} := 1.
    \]
  \item \textbf{SA-AdaGrad}: $M_t = H^{-1}C_{t-1}^{1/2} \to H^{-1}(S+\epsilon I)^{1/2}$, so
    \[
      \bar C_P^{\mathrm{AdaGrad}}
      :=
      \|H^{-1}(S+\epsilon I)^{1/2}\|_{\mathrm{op}},
    \]
    whose magnitude depends on the joint spectrum and alignment of $H$ and $S$.
  \item \textbf{SA-RMSProp}: writing $s_{\mathrm{diag}} = \mathrm{diag}(S)$,
    $M_t = H^{-1}\mathrm{Diag}(v_{t-1})^{1/2}
    \to H^{-1}\mathrm{Diag}(s_{\mathrm{diag}}+\epsilon \mathbf{1})^{1/2}$, so
    \[
      \bar C_P^{\mathrm{RMSProp}}
      =
      \left\|H^{-1}\mathrm{Diag}(s_{\mathrm{diag}}+\epsilon \mathbf{1})^{1/2}\right\|_{\mathrm{op}}.
    \]
\end{itemize}
In particular,
\[
  \frac{\bar C_P^{\mathrm{ONS}}}{\bar C_P^{\mathrm{Id}}}
  =
  \lambda_{\min}(H)
  =
  \frac{\lambda_{\max}(H)}{\kappa(H)}.
\]
Under the natural spectral normalization $\lambda_{\max}(H)=1$, one has
$\bar C_P^{\mathrm{Id}} = \kappa(H)$ and $\bar C_P^{\mathrm{ONS}} = 1$, so the
SA-ONS operator factor is smaller than the identity-baseline factor by the
multiplicative ratio $1/\kappa(H)$.
\end{proposition}

\begin{proof}
Deferred to the supplementary material.
\end{proof}

\noindent
The comparison isolates the dominant geometric effect: SA-ONS drives
$M_t \to I$, eliminating the anisotropy of~$H^{-1}$ and removing the
condition-number penalty carried by the identity baseline at the
level of the asymptotic operator factor $\bar C_P$.  Among the
full-matrix rules, SA-ONS achieves the Newton-metric limit
$P_tH \to I$ unconditionally, whereas SA-AdaGrad achieves it only in
the non-generic case $H^2 = S + \epsilon I$; SA-RMSProp replaces full
covariance adaptation by a diagonal surrogate.  No universal ordering
between $\bar C_P^{\mathrm{RMSProp}}$ and the full-matrix factors
follows without additional structure on $H$ and $S$, and a full
comparison of the complete constants $C_R$ would also require matching
the coupled-difference terms.

\subsection{Asymptotic Normality and Quantitative Gaussian Approximation}
\label{sec:clt}

We now establish the inferential consequences.  The decomposition of
Lemma~\ref{lem:decomp} confines all preconditioner dependence
to~$R_n$, and Proposition~\ref{prop:l2_remainder} guarantees
$\sqrt{n}\,R_n \to 0$ in $L^1$ (hence, in probability) under only the
$L^2(\mathrm{op})$ stabilization bounds of
Corollaries~\ref{cor:sa_stab}--\ref{cor:threshold}, with no
bounded-input hypothesis required.  The CLT, therefore, follows from a
standard martingale central limit theorem.  Corollary~\ref{cor:gauss_approx} then gives
a quantitative Gaussian approximation bound in Wasserstein distance;
its framing relative to the unpreconditioned nonlinear baseline of
Kong and Srikant~\citep{kong2026finite} is discussed after the theorem
statement.

\begin{theorem}[Asymptotic Normality under Dynamically Preconditioned ASGD]
\label{thm:clt}
Let Assumptions~\ref{ass:mg}--\ref{ass:iterate} hold.  Assume additionally that
the state-dependent noise covariance map
$\Sigma(x) := \mathbb{E}[\xi(x,\zeta)\xi(x,\zeta)^\top]$
is continuous at $x^*$ and that
$\sup_{t \ge 1}\mathbb{E}\|\xi_t\|^{2+\delta} < \infty$ for some $\delta > 0$.
Assume in addition that $\sqrt{n}\,R_n(\{P_t\}) \to 0$ in probability;
this is supplied by Proposition~\ref{prop:l2_remainder} whenever the
$L^2(\mathrm{op})$ stabilization bounds of
Corollaries~\ref{cor:sa_stab}--\ref{cor:threshold} hold, with no
bounded-input hypothesis required.
Then
\begin{equation}
    \sqrt{n}\,(\overline{x}_n - x^*) \;\xrightarrow{d}\;
    \mathcal{N}\!\left(0,\; H^{-1}S H^{-1}\right).
\end{equation}
\end{theorem}

\begin{proof}
Deferred to the supplementary material.
\end{proof}

The limiting distribution is identical to that of unpreconditioned
Polyak--Ruppert averaging
\citep{polyak1992acceleration,ruppert1988efficient}: the dynamic
preconditioner $P_t$ affects only the higher-order remainder $R_n$, not the
asymptotic covariance $H^{-1}SH^{-1}$.  Consequently, first-order
Wald-type inference based on a consistent estimator of $H^{-1}SH^{-1}$
applies without modification relative to the unpreconditioned case.

\begin{corollary}[Asymptotic Efficiency under Information Equality]
\label{cor:efficiency}
When $S = H$---the information-equality case associated with correctly
specified likelihood models, including well-specified canonical GLMs---the
asymptotic covariance $H^{-1}SH^{-1} = H^{-1}$ equals the inverse Fisher
information matrix, attaining the asymptotic Fisher-information lower bound
(under the usual likelihood regularity conditions).  Hence, all three
SA-preconditioned Polyak--Ruppert averages are asymptotically efficient.
\end{corollary}

\begin{proof}
Immediate from Theorem~\ref{thm:clt} with $S = H$.
\end{proof}

\noindent
Without information equality, the asymptotic covariance retains the
general Huber--White sandwich form $H^{-1}SH^{-1}$.

\subsubsection{Rate Preservation in Wasserstein Distance}

Writing
$\widetilde{\Xi}_n := -n^{-1/2}H^{-1}\sum_{t=1}^n \xi_t$,
the Wasserstein Gaussian approximation problem for
$\sqrt{n}(\overline{x}_n-x^*)$ reduces---via the decomposition of
Lemma~\ref{lem:decomp}---to a Gaussian approximation for
$\widetilde{\Xi}_n$ plus explicit first-moment control of the Taylor
and dynamic remainders.  The Wasserstein metric is convenient for this
reduction because additive remainders enter through first moments, a
structural feature exploited in Wasserstein CLT rate bounds
\citep{rio2009upper,bonis2020stein}.

Throughout this subsection we write $V = H^{-1}SH^{-1}$ for the limiting
covariance and $\Phi_V$ for the $\mathcal{N}(0,V)$ distribution.  For
probability measures $\mu,\nu$ on $\mathbb{R}^d$ with finite first moments, we
use the \emph{$1$-Wasserstein distance}
\begin{equation}\label{eq:wasserstein_dist}
  d_W(\mu,\nu)
  \;=\;
  \sup_{\|h\|_{\mathrm{Lip}} \le 1}
  \left|
    \int h \, d\mu - \int h \, d\nu
  \right|,
\end{equation}
where $\|h\|_{\mathrm{Lip}}$ denotes the Euclidean Lipschitz seminorm of~$h$.

\begin{corollary}[Wasserstein Rate Preservation under Dynamic Preconditioning]
\label{cor:gauss_approx}
Let Assumptions~\ref{ass:mg}--\ref{ass:iterate} hold with step size
$\eta_t = \eta_0 t^{-\alpha}$ for $\alpha \in (1/2,1)$, and suppose the
dynamic remainder satisfies
\[
    \mathbb{E}\|R_n(\{P_t\})\|^2 \le C_R n^{\alpha-2}.
\]
Let
$\widetilde{\Xi}_n := -\frac{1}{\sqrt{n}}H^{-1}\sum_{t=1}^n \xi_t$, and let
$\Delta_{\mathrm{mart},W}(n)$ denote the Wasserstein approximation error for
the rescaled martingale term:
\[
    \Delta_{\mathrm{mart},W}(n)
  \;=\;
  d_W\!\left(
    \mathcal{L}(\widetilde{\Xi}_n),
    \Phi_V
  \right).
\]
Then there exist constants $C_1, C_2 > 0$, depending only on
$(H, C_R, L_R, C_\Delta, \alpha, \eta_0)$, such that
\begin{equation}\label{eq:gauss_approx}
  d_W\!\left(
    \mathcal{L}\bigl(\sqrt{n}\,\overline{\Delta}_n\bigr),
    \Phi_V
  \right)
  \;\leq\;
  \Delta_{\mathrm{mart},W}(n)
  \;+\; C_1\, n^{(\alpha-1)/2}
  \;+\; C_2\, n^{1/2-\alpha}.
\end{equation}
In particular, for each of the three SA preconditioners (SA-AdaGrad,
SA-RMSProp, SA-ONS), the pathwise uniform ellipticity of $P_t$
supplied by Corollaries~\ref{cor:adaptive_iterate} (for SA-AdaGrad and
SA-RMSProp) and~\ref{cor:ons_iterate} (for SA-ONS), together with the
stabilization results of
Corollaries~\ref{cor:sa_stab} and~\ref{cor:threshold} under the
bounded-input conditions stated there ($\|g_t\| \le G$ for SA-AdaGrad
and SA-RMSProp; the two-sided spectral bound
$h_- I \preceq \hat{H}_t \preceq h_+ I$ together with the
conditional-mean constancy
$\mathbb{E}[\hat{H}_t \mid \mathcal{F}_{t-1}] = H$ for SA-ONS),
establish the pathwise hypotheses of Proposition~\ref{prop:tight},
which in turn imply the dynamic-remainder bound assumed above.
\end{corollary}

\begin{proof}
Deferred to the supplementary material.
\end{proof}

\paragraph*{Framing: rate preservation, not a new rate}
Corollary~\ref{cor:gauss_approx} is a preservation statement, not a
new rate.  Kong and Srikant~\citep{kong2026finite} established
the $n^{-1/6}$ Wasserstein rate for \emph{unpreconditioned} nonlinear
SA (under their non-asymptotic CLT condition on the driving noise);
our contribution is to show that the stabilization threshold carries that baseline rate
to the preconditioned setting, so that imposing a data-driven preconditioner on top of averaged SGD incurs no additional
penalty in the quantitative Gaussian approximation.  We state the
result as a corollary to reflect this inheritance: the exponent is
inherited from the unpreconditioned analysis, while the threshold
identifies the class of preconditioners for which the inheritance is
legitimate.

The
inequality~\eqref{eq:gauss_approx} involves two competing remainder
contributions: the Taylor term at rate $n^{1/2-\alpha}$---a consequence
of nonlinearity that is present even without preconditioning---and the
dynamic remainder at rate $n^{(\alpha-1)/2}$, which captures the cost
of time-varying preconditioning.  Since $\alpha \in (1/2,1)$, both
exponents are negative, so both contributions vanish as $n \to \infty$.
The optimal choice of~$\alpha$ balances them:
$(\alpha-1)/2 = 1/2 - \alpha$, giving $\alpha^* = 2/3$.
At this value, both remainders contribute at rate $n^{-1/6}$, yielding
\[
  d_W\!\left(
    \mathcal{L}\bigl(\sqrt{n}\,\overline{\Delta}_n\bigr),
    \Phi_V
  \right)
  \le
  \Delta_{\mathrm{mart},W}(n) + \mathcal{O}(n^{-1/6}).
\]
The decisive observation is that the dynamic remainder $n^{(\alpha-1)/2}$
matches the Taylor remainder $n^{1/2-\alpha}$ at $\alpha^*$---it does
not dominate it.  Consequently, the overall exponent is governed by the
Taylor term, which is \emph{not} attributable to preconditioning, and the
preconditioner-specific contribution disappears into the same order as
the nonlinearity cost.  When the martingale term itself satisfies
$\Delta_{\mathrm{mart},W}(n) = \mathcal{O}(n^{-1/2})$---as holds for
i.i.d.\ summands with finite third moments
\citep{rio2009upper,bonis2020stein}---the right-hand side reduces to
$\mathcal{O}(n^{-1/6})$, reproducing the unpreconditioned nonlinear
rate of Kong and Srikant~\citep{kong2026finite}.

The Wasserstein metric is natural for this preservation statement
because
\[
  d_W(\mathcal{L}(X+R),\mathcal{L}(Y)) \le
  d_W(\mathcal{L}(X),\mathcal{L}(Y)) + \mathbb{E}\|R\|,
\]
so additive remainders enter through first moments---exactly the form
produced by the preconditioner-isolating decomposition in
Lemma~\ref{lem:decomp}.  The preservation argument, therefore, goes
through: our decomposition produces an additive remainder whose
first-moment rate is controlled by the stabilization threshold, and
the Wasserstein metric converts this first-moment control directly into
Gaussian-approximation error.

For context, in linear SA ($u_t = 0$) the Taylor remainder vanishes
identically, and correspondingly sharper rates are available:
$n^{-1/4}$ \citep{samsonov2024gaussian} and $n^{-1/3}$
\citep{butyrin2025improved} in convex distance (a different metric, so
the comparison is qualitative); Sheshukova et~al.\
\citep{sheshukova2025bootstrap} further develop a multiplier bootstrap
that avoids explicit covariance estimation.  For nonlinear averaged
SGD, Shao and Zhang~\citep{shao2022berry} establish multivariate
Berry--Esseen bounds via Stein's method.  The gap between the linear
and nonlinear rates is an artifact of the Taylor remainder, not of
preconditioning, and our bound inherits this gap without enlarging it: within the
Wasserstein-$1$ framework of Corollary~\ref{cor:gauss_approx}, dynamic
preconditioning above the threshold tracks the rate admitted by the
underlying unpreconditioned nonlinear problem.

\begin{remark}[Prospects for sharper rates]
\label{rem:rate_gap_open}
The proof bounds $\{u_t\}$ adversarially.  Because $u_t$ is approximately
quadratic in $\Delta_t$, the average $n^{-1}\sum u_t$ may admit partial
cancellation.  Exploiting this, or applying Stein's method directly to
$\sqrt{n}\,\overline{\Delta}_n$, could sharpen the rate beyond
$n^{-1/6}$.
\end{remark}

\section{Numerical Study}
\label{sec:simulation}

We test three predictions of the theory.  Identity ($P_t = I$,
classical PR-SGD) is included throughout as the established baseline;
the new claims concern the three SA-preconditioned variants
SA-AdaGrad, SA-RMSProp, and SA-ONS.  (i)~The CLT of
Theorem~\ref{thm:clt} holds for all three SA-preconditioned
Polyak--Ruppert averages, so that sandwich-based Wald intervals are
asymptotically calibrated.  (ii)~The dynamic remainder $\sqrt{n}\,R_n$
vanishes in $L^1$ via the $L^2$ pathway of
Proposition~\ref{prop:l2_remainder}, with the coupled-difference
constant $C_{\mathrm{coup}}$ in the bound of
Proposition~\ref{prop:tight} providing a heuristic for the transient
finite-sample ordering across SA preconditioners.
(iii)~The stabilization threshold $\beta > (\alpha+1)/2$ of
Theorem~\ref{thm:sharp_threshold} is sharp in the sense of
Proposition~\ref{thm:sharpness}, so that constant-EMA preconditioners
(for which no positive stabilization rate is expected) fail to drive
$\sqrt{n}\,R_n$ to zero over the observed sample-size range.
Our diagnostics are \emph{coverage}, the fraction of coordinates for
which the sandwich-based 95\% marginal confidence interval contains
the truth, and \emph{NMSE},
$n\|\overline{x}_n - x^*\|^2 / \mathrm{Tr}(H^{-1}SH^{-1})$, whose
limit in distribution has mean~$1$ under the CLT.  Throughout
Section~\ref{sec:simulation}, both diagnostics are computed against
the oracle sandwich $H^{-1}SH^{-1}$ formed from the population (or
full-dataset) Hessian and gradient covariance, isolating the
preconditioner effects from plug-in estimation error; the resulting
intervals, therefore, probe the asymptotic calibration claim of
Theorem~\ref{thm:clt} rather than feasible inference with an online
sandwich estimator.  Three experiments are reported:
a synthetic linear-regression benchmark (Figure~\ref{fig:overview}),
a real-data logistic regression (Figure~\ref{fig:logistic}), and a
threshold-violation experiment (Figure~\ref{fig:threshold}).
Additional diagnostics are in the supplementary material.

\paragraph*{Experimental design}
We consider streaming linear regression in dimensions $d \in \{5, 20, 50\}$ with
Gaussian covariates $a_t \sim \mathcal{N}(0, H)$, Toeplitz
$H_{jk} = 0.4^{|j-k|}$ ($\kappa(H) \le 5.5$ uniformly in $d$), and two noise
regimes: \textbf{general sandwich} ($S \neq H$, heteroskedastic with
$\|S-H\|_F/\|H\|_F \in [1.1,\,2.0]$) and
\textbf{information-equality} ($S = H$, homoskedastic Gaussian).
Four methods---Identity ($P_t = I$), SA-AdaGrad, SA-RMSProp,
SA-ONS---share $\eta_t = 0.2\,t^{-0.7}$ ($\alpha = 0.7$) and
$\rho_t = (t+1)^{-1}$ (the $c/t$ schedule of
Theorem~\ref{thm:sa_stab} with $c=1$, shifted by one step).  All three SA runs are Hessian-input variants
of the Section~\ref{sec:construction} constructions, with driving
input $a_ta_t^\top$ for the full-matrix accumulators (SA-AdaGrad and
SA-ONS) and $\mathrm{diag}(a_t\odot a_t)$ for the diagonal accumulator (SA-RMSProp).
The experiment adds a constant ridge $0.5\,I$ to each accumulator
before inversion; absorbing it into the driving term yields
$\widetilde{W}_t = a_ta_t^\top + 0.5\,I$ or
$\mathrm{diag}(a_t\odot a_t) + 0.5\,I$, each satisfying
$\widetilde{W}_t \succeq 0.5\,I$ a.s., with $\|\widetilde{W}_t\|_{\mathrm{op}}$
admitting moments of all orders since $a_t$ is Gaussian.  The three
methods differ by accumulator structure (full
matrix vs.\ diagonal) and spectral map ($A\mapsto A^{-1/2}$ for
SA-AdaGrad/SA-RMSProp, $A\mapsto A^{-1}$ for SA-ONS), probing the
dependence of the finite-sample remainder on the preconditioner
construction in the spirit of
Proposition~\ref{prop:comparison}.  The recursions fit inside the
generic framework of Theorem~\ref{thm:sa_stab} (cf.\
Remark~\ref{rem:generic_scope}), which delivers
$\|M_t\|_{L^2(\mathrm{op})} = \mathcal{O}(1)$ and
$\|M_t - M_{t-1}\|_{L^2(\mathrm{op})} = \mathcal{O}(t^{-1})$; combined
with the step-size schedule, these yield
$\|\eta_t^{-1}M_t - \eta_{t-1}^{-1}M_{t-1}\|_{L^2(\mathrm{op})}
= \mathcal{O}(t^{\alpha-1})$, verifying the hypotheses of
Proposition~\ref{prop:l2_remainder}, the $L^2$ pathway to the CLT of
Theorem~\ref{thm:clt} (rather than Proposition~\ref{prop:tight}, whose
almost-sure pathwise operator-norm hypotheses on $M_t$ are unavailable
under the Gaussian design).  We use 50 replications on a log-spaced
grid $n \in \{5{,}000, \ldots, 5{,}000{,}000\}$.

\paragraph*{Synthetic results}

\begin{figure}[tbp]
\centering
\includegraphics[width=\textwidth]{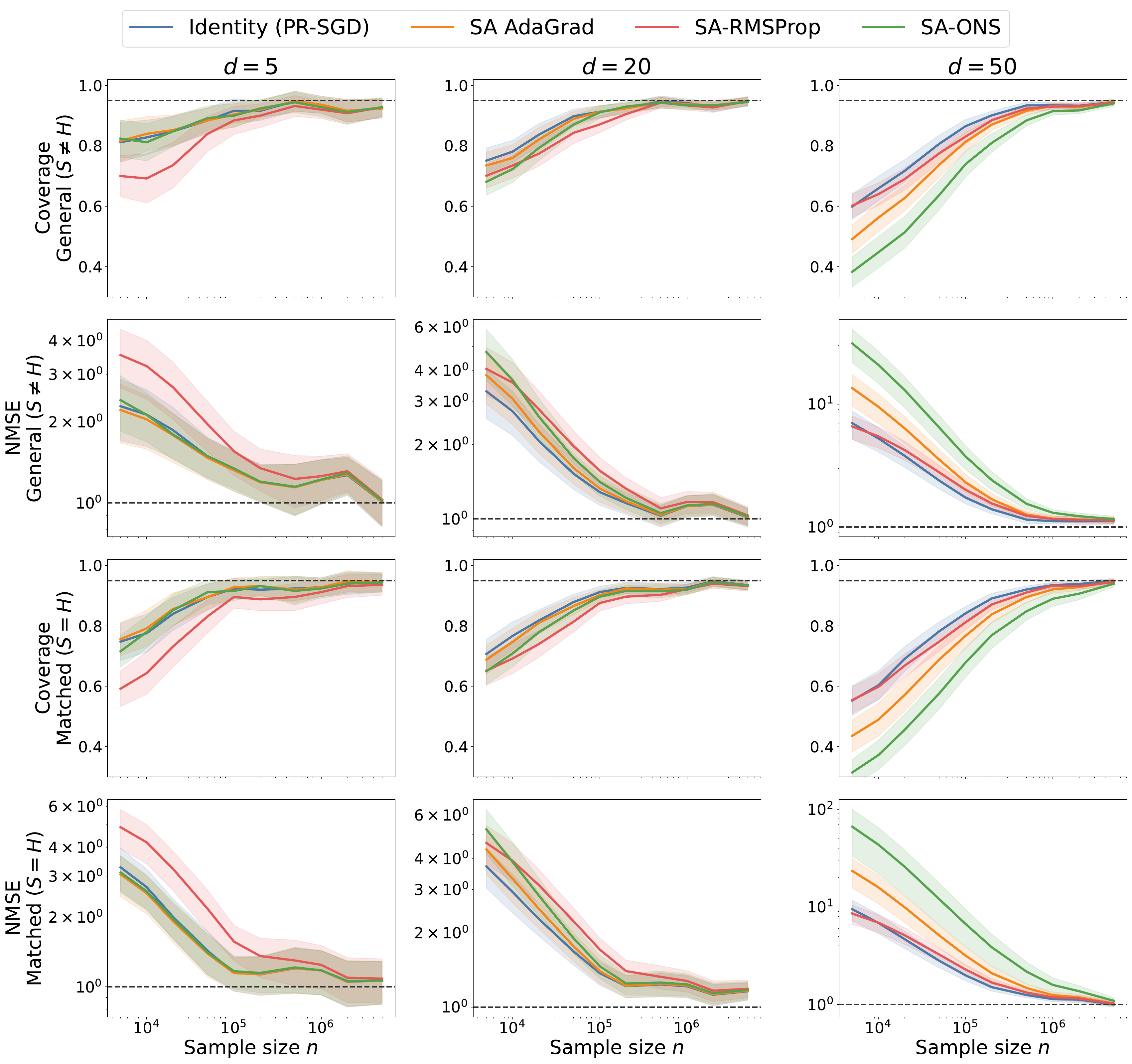}
\caption{Coordinatewise marginal coverage (rows 1, 3) and normalized MSE
(rows 2, 4) for the synthetic linear regression experiment.  Columns
correspond to $d = 5, 20, 50$; the upper two rows show the general
regime ($S \neq H$) and the lower two rows show the information-equality
regime ($S = H$).  Sample sizes are log-spaced from $n = 5 \times 10^3$
to $n = 5 \times 10^6$.  Each curve is the mean over 50 replications;
shaded bands are $\pm 1.96 \times \mathrm{SE}$ Monte Carlo uncertainty
bands.  Dashed lines mark the nominal coverage level~$0.95$ and the
mean~$1.0$ of the NMSE limit distribution under the CLT.
The four methods are Identity ($P_t = I$), SA-AdaGrad, SA-RMSProp, and SA-ONS.}
\label{fig:overview}
\end{figure}

Figure~\ref{fig:overview} displays the principal findings.  At the
terminal sample size $n = 5 \times 10^6$, all four methods attain
coverage $\in [0.92,\,0.95]$ and $\mathrm{NMSE} \in [1.00,\,1.19]$
across every $(d,\text{scenario})$ cell, with the four curves
visually indistinguishable at that horizon.  This is consistent with
the CLT universality of Theorem~\ref{thm:clt}: the asymptotic covariance
$H^{-1}SH^{-1}$ is preconditioner-independent, so the four estimators
share the same limit law and, once $\sqrt{n}\,R_n$ is negligible,
become statistically indistinguishable up to Monte Carlo error at this
horizon.  Convergence slows with dimension; at $d = 50$ and
$n = 2 \times 10^5$ the methods are still pre-asymptotic
(coverage $\in [0.81,\,0.90]$, NMSE $\in [1.39,\,2.40]$), then
collapse onto the benchmarks as~$n$ grows, consistent with the
vanishing of $\sqrt{n}\,R_n$ in $L^1$ (and, hence, in probability)
asserted by Proposition~\ref{prop:l2_remainder}.  In the $d = 50$ general-regime panel, the transient
NMSE ordering at $n = 2 \times 10^5$ is
Identity~$1.39 <$~SA-RMSProp~$1.55 <$~SA-AdaGrad~$1.68 <$~SA-ONS~$2.40$,
qualitatively in line with the role of the coupled-difference
constant $C_{\mathrm{coup}}$ in the bound of
Proposition~\ref{prop:tight}: full-matrix preconditioners incur a
larger estimation burden than diagonal ones, while Identity incurs
none.  Proposition~\ref{prop:tight} is a generic bound and does not
formally rank methods; the observed ordering is reported as empirical
evidence consistent with that mechanism.  The
$d \in \{5, 20\}$ panels serve as the universality check and do not
reproduce this ordering at small~$n$; we tentatively attribute this
to the coupled-difference constants being within Monte Carlo
resolution at these smaller dimensions.  Both findings persist across noise regimes.  Additional
CLT-convergence diagnostics---the scaled dynamic remainder
$\sqrt{n}\,\|R_n\|$ and its empirical log--log slopes---are reported
in the supplement.

\paragraph*{Real-data experiment}
We test the four methods on $\ell_2$-regularized logistic regression
using the UCI Breast Cancer Wisconsin
dataset~\citep{uci_breast_cancer} (10 mean features plus intercept,
$d = 11$; $\kappa(H) \approx 4.3$; ridge $\lambda = 0.1$; 100
replications; $n$ up to $200{,}000$).  The SA second-moment input is
the stochastic Hessian
$\hat{H}_t = p_t(1 - p_t)\,a_ta_t^\top + \lambda I$ (or its diagonal
for SA-RMSProp), where $p_t = \sigma(a_t^\top x_t)$; this is a
Hessian-input variant analogous to the synthetic setup, accommodated
by the generic SA recursion of Theorem~\ref{thm:sa_stab}.  Per the
Section~\ref{sec:simulation} convention, coverage is computed against
the oracle sandwich $H^{-1}SH^{-1}$ formed from the full dataset.
Figure~\ref{fig:logistic} shows that
all four methods reach $\mathrm{NMSE} \approx 1$ by
$n \approx 100{,}000$ and nominal $0.95$ coverage by
$n \approx 150{,}000$, with the four curves becoming visually
indistinguishable by $n \approx 50{,}000$ (small differences are
visible at the smallest sample sizes).  This is consistent with the
CLT universality of Theorem~\ref{thm:clt} on a nonlinear real-data
model: whichever SA preconditioner is chosen, Wald-type inference
calibrated by the sandwich covariance is asymptotically valid.  Any
preconditioner-specific finite-sample ordering would live in the
sub-leading remainder and is below Monte Carlo resolution at this
mild conditioning; it is in the $d = 50$ synthetic panel,
qualitatively in line with the role of $C_{\mathrm{coup}}$ in
Proposition~\ref{prop:tight}, that such an ordering becomes
transiently visible.

\begin{figure}[tbp]
\centering
\includegraphics[width=\textwidth]{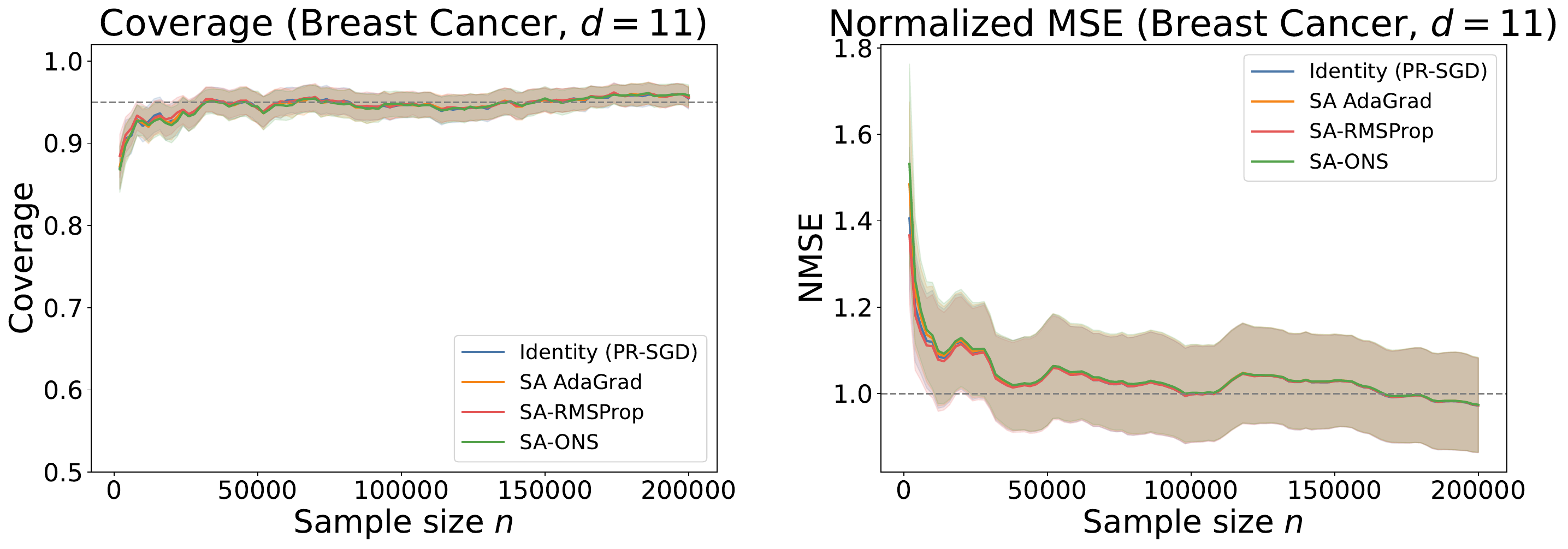}
\caption{Coordinatewise marginal coverage and normalized MSE for
  $\ell_2$-regularized logistic regression on the Breast Cancer Wisconsin
  dataset ($d = 11$).  Dashed lines mark the nominal coverage $0.95$ and
  the mean~$1.0$ of the NMSE limit distribution under the CLT.
  Each curve is the mean over 100 replications; shaded bands are
  $\pm 1.96 \times \mathrm{SE}$ Monte Carlo uncertainty bands.}
\label{fig:logistic}
\end{figure}

\paragraph*{Threshold violation: constant-EMA preconditioners}
For $\alpha = 0.7$ the stabilization threshold of
Theorem~\ref{thm:sharp_threshold} is $(\alpha+1)/2 = 0.85$.
Classical RMSProp \citep{tieleman2012rmsprop} uses a constant
exponential moving average; heuristically, in this nondegenerate
stochastic setting a constant gain $\rho$ keeps the fluctuations of
$M_t$ from decaying, so no positive stabilization rate is expected
in the sense of Definition~\ref{def:beta} (we do not formally
establish a lower bound on these fluctuations for the exact
constant-EMA recursion).  The hypotheses of
Theorem~\ref{thm:sharp_threshold} are, therefore, inapplicable, and
Proposition~\ref{thm:sharpness} shows that the threshold cannot be
lowered using only the boundedness and polynomial-decay hypotheses,
so non-stabilization is the predicted outcome.  We test this by
replacing the SA gain $\rho_t = (t+1)^{-1}$ in RMSProp with constants
$\rho \in \{0.5, 0.999\}$ in the same $d = 5$ Toeplitz($0.4$) design
as above, with step-size $\eta_t = 2.0\,t^{-0.7}$ and ridge
regularization $+I$ (parameters tuned for this stress test; see
the supplement), 100 replications, and $n$ up to
$5 \times 10^6$.

\begin{figure}[tbp]
\centering
\includegraphics[width=0.65\textwidth]{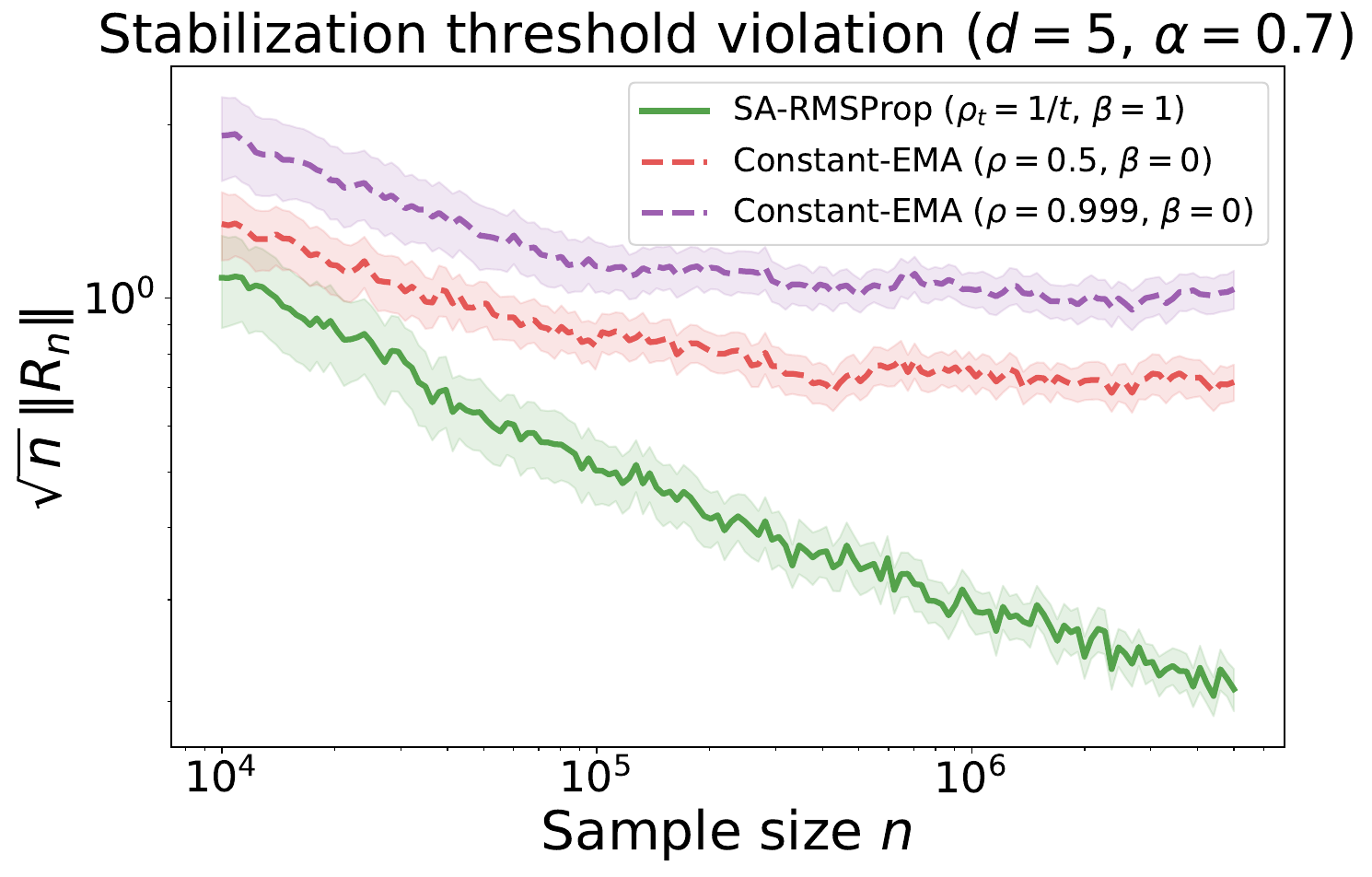}
\caption{Stabilization threshold violation ($d = 5$, $\alpha = 0.7$).
  The scaled remainder $\sqrt{n}\,\|R_n\|$ is shown on log--log axes.
  SA-RMSProp (decaying gain $\rho_t = (t+1)^{-1}$, heuristically
  $\beta = 1$ under unbounded Gaussian inputs; solid green) declines
  steadily (slope $-0.21$), consistent with
  $\beta > (\alpha+1)/2$.  Constant-EMA variants with $\rho = 0.5$
  and~$0.999$ (no positive stabilization rate; dashed) flatten
  (slopes $-0.02$ and $-0.03$), consistent with a non-vanishing
  remainder over the observed sample-size range.
  Shaded bands are $\pm 1.96 \times \mathrm{SE}$ Monte Carlo
  uncertainty bands over 100 replications.
  Coverage and NMSE panels are in the supplement.}
\label{fig:threshold}
\end{figure}

Figure~\ref{fig:threshold} is consistent with the theoretical
prediction.
SA-RMSProp's scaled remainder $\sqrt{n}\,\|R_n\|$ declines with
log--log slope~$-0.21$, while the constant-EMA curves flatten
(slopes~$-0.02$ and~$-0.03$), producing a 3--5$\times$ larger
remainder at $n = 5 \times 10^6$.  This is consistent with the
mechanism illustrated by Proposition~\ref{thm:sharpness}:
non-decaying preconditioner fluctuations accumulate in the
Polyak--Ruppert average.  Coverage
effects are moderate at these sample sizes (all methods reach
roughly $0.95$ at $n = 5 \times 10^6$); the deterministic
construction of Proposition~\ref{thm:sharpness} would force a
sharper divergence.  Additional coverage and NMSE panels and the
log--log slope table are reported in the supplement.

\section{Discussion}
\label{sec:discussion}

We have established a stabilization-rate threshold for the CLT under
dynamically preconditioned Polyak--Ruppert averaging and shown it
cannot be weakened within its hypothesis class.  Prior work
\citep{leluc2023asymptotic,boyer2023stochastic,godichon2024hessian}
retains $P_t$ inside the leading term; to our knowledge, no prior
result identifies a phase transition for the CLT under online
preconditioning at this level of generality.

The threshold classifies the SA preconditioner designs studied here.
Theorem~\ref{thm:sa_stab} shows that SA preconditioners with gain
$\rho_t = c/t$ deliver $L^2$ one-step stabilization of order $t^{-1}$
(pathwise $\beta = 1$ under bounded inputs), covering SA-AdaGrad,
SA-RMSProp, and SA-ONS in a single argument.  By contrast, classical
constant-EMA methods (RMSProp \citep{tieleman2012rmsprop}, Adam
\citep{kingma2015adam}) heuristically do not admit polynomial
stabilization (Section~\ref{sec:simulation}); extending the framework
to constant-EMA and momentum-based methods---via ergodic or mixing-rate
conditions on~$M_t$, with Lemma~\ref{lem:decomp} as the starting
point---is an important open problem.

Proposition~\ref{prop:comparison} places the three SA constructions on a
common operator-factor scale: SA-ONS eliminates the $\kappa(H)$ penalty
carried by the identity baseline, while SA-RMSProp offers $\mathcal{O}(d)$
cost.  The experiments of Section~\ref{sec:simulation} suggest that at
moderate-to-high $d$ diagonal SA-RMSProp can be preferable in finite
samples, even though the operator-factor ordering between diagonal and
full-matrix rules is problem-dependent.

\medskip
\noindent\textbf{Limitations and open problems.}
\begin{enumerate}
  \item \textbf{Coupled lower bounds.}
    Proposition~\ref{thm:sharpness} establishes tightness via a
    deterministic construction in which $\{M_t\}$ and $\{\Delta_t\}$
    are decoupled (Remark~\ref{rem:sharpness_scope}).  Whether the
    threshold $\beta > (\alpha+1)/2$ can be improved by exploiting
    the genuine algorithmic coupling---where $M_t$ depends on
    $\Delta_t$ through the dynamics---remains open.

  \item \textbf{Optimality of $n^{-1/6}$.}
    No lower bound for the normal-approximation rate in nonlinear SA is
    known.  Exploiting cancellation in $n^{-1}\sum_t u_t$ or applying
    Stein's method directly could potentially improve the rate
    (Remark~\ref{rem:rate_gap_open}).

  \item \textbf{Non-strongly-convex objectives.}
    When $\alpha = 1$ or $F$ is merely convex, the present analysis does
    not apply; heuristically the threshold may tighten to $\beta \ge 1$,
    but a rigorous treatment would require new techniques.

  \item \textbf{Sub-Gaussian tails and dimension-explicit constants.}
    Relaxing bounded-gradient assumptions to sub-Gaussian tails,
    extending the non-asymptotic Gaussian approximation bounds of
    \citep{samsonov2024gaussian,butyrin2025improved} to the
    preconditioned setting, and deriving sharp dimension-explicit
    remainder constants are natural next steps.
\end{enumerate}

\begin{supplement}
\stitle{Supplement to ``When Does Dynamic Preconditioning Preserve the Polyak--Ruppert CLT? A Stabilization Threshold''}
\sdescription{This supplement contains deferred proofs of all results stated in the main text
and additional experiments diagnostics.}
\end{supplement}

\bibliographystyle{imsart-number}
\bibliography{references}

\clearpage

\begin{center}
{\Large\bfseries Supplement to ``When Does Dynamic Preconditioning Preserve the Polyak--Ruppert CLT? A Stabilization Threshold''\par}
\end{center}

\bigskip

\renewcommand{\appendixname}{}
\appendix
\setcounter{theorem}{0}
\renewcommand{\thetheorem}{S.\arabic{theorem}}

\noindent
This supplement contains deferred proofs of all results stated in the main
text and additional experiments diagnostics.  The supplementary material
is organized as follows.  Section~\ref{sec:proofs} collects all deferred
proofs from the main text.  Section~\ref{sec:experiments} provides
additional experiments diagnostics.

\section{Deferred Proofs}
\label{sec:proofs}

\subsection{Proof of Proposition~\ref{prop:iterate_bound}}
\label{app:iterate_bound}
\begin{proof}[Proof of Proposition~\ref{prop:iterate_bound}]
Let $\mathbb{E}_t[\cdot] := \mathbb{E}[\cdot \mid \mathcal{F}_{t-1}]$ and define
the excess risk
\[
    a_t := \mathbb{E}[F(x_t) - F(x^*)].
\]
Since $F$ is $L$-smooth on $\mathcal{K}$ and the iterates remain in
$\mathcal{K}$ almost surely, the descent inequality yields
\[
    F(x_{t+1})
    \leq
    F(x_t)
    +
    \nabla F(x_t)^\top (x_{t+1}-x_t)
    +
    \frac{L}{2}\|x_{t+1}-x_t\|^2.
\]
Using $x_{t+1}-x_t = -\eta_t P_t(\nabla F(x_t)+\xi_t)$, we obtain
\[
    F(x_{t+1})
    \leq
    F(x_t)
    -
    \eta_t \nabla F(x_t)^\top P_t(\nabla F(x_t)+\xi_t)
    +
    \frac{L\eta_t^2}{2}\|P_t(\nabla F(x_t)+\xi_t)\|^2.
\]
Taking conditional expectation and using
$\mathbb{E}_t[\xi_t] = 0$ from Assumption~\ref{ass:mg},
\[
    \mathbb{E}_t[F(x_{t+1})]
    \leq
    F(x_t)
    -
    \eta_t \nabla F(x_t)^\top P_t \nabla F(x_t)
    +
    \frac{L\eta_t^2}{2}\,
    \mathbb{E}_t\|P_t(\nabla F(x_t)+\xi_t)\|^2.
\]
Condition~(ii) implies
\[
    \nabla F(x_t)^\top P_t \nabla F(x_t)
    \geq
    p_- \|\nabla F(x_t)\|^2
\]
and
\[
    \|P_t(\nabla F(x_t)+\xi_t)\|^2
    \leq
    p_+^2 \|\nabla F(x_t)+\xi_t\|^2.
\]
Using $\|a+b\|^2 \leq 2\|a\|^2 + 2\|b\|^2$, we, therefore, get
\[
    \mathbb{E}_t[F(x_{t+1})]
    \leq
    F(x_t)
    -
    \eta_t \bigl(p_- - L p_+^2 \eta_t\bigr) \|\nabla F(x_t)\|^2
    +
    L\eta_t^2 p_+^2 \mathbb{E}_t\|\xi_t\|^2.
\]
By Assumption~\ref{ass:var},
\[
    \mathbb{E}_t\|\xi_t\|^2
    =
    \mathrm{Tr}\!\left(\mathbb{E}_t[\xi_t\xi_t^\top]\right)
    \leq
    \mathrm{Tr}(\overline{S})
    =: \sigma^2.
\]
Since $\eta_t \to 0$, setting
$T_0 := \lceil (2L p_+^2 \eta_0 / p_-)^{1/\alpha} \rceil$ gives
$L p_+^2 \eta_t \leq p_-/2$ for all $t \geq T_0$.  For such $t$,
$p_- - L p_+^2 \eta_t \geq p_-/2$, and, therefore,
\[
    \mathbb{E}_t[F(x_{t+1})]
    \leq
    F(x_t)
    -
    \frac{\eta_t p_-}{2} \|\nabla F(x_t)\|^2
    +
    Lp_+^2\sigma^2 \eta_t^2.
\]
Subtracting $F(x^*)$ and applying the Polyak--{\L}ojasiewicz inequality
$\|\nabla F(x)\|^2 \ge 2\mu(F(x)\!-\!F(x^*))$ from $\mu$-strong
convexity (Assumption~\ref{ass:convex}), we obtain
\[
    \mathbb{E}_t[F(x_{t+1}) - F(x^*)]
    \leq
    \bigl(1 - \mu p_- \eta_t\bigr)
    \bigl(F(x_t)-F(x^*)\bigr)
    +
    Lp_+^2\sigma^2 \eta_t^2.
\]
Taking total expectation yields, for $t \geq T_0$,
\[
    a_{t+1}
    \leq
    \bigl(1 - \mu p_- \eta_t\bigr)a_t
    +
    Lp_+^2\sigma^2 \eta_t^2.
\]
Write
\[
    c := \mu p_- \eta_0,
    \qquad
    b := Lp_+^2\sigma^2 \eta_0^2,
\]
so that
\[
    a_{t+1} \leq (1 - c t^{-\alpha})a_t + b t^{-2\alpha}
    \qquad \text{for } t \geq T_0.
\]
Choose $K > b/c$.  By the mean value theorem there exists
$\theta_t \in (t,t+1)$ with
\[
    t^{-\alpha} - (t+1)^{-\alpha}
    =
    \alpha\, \theta_t^{-\alpha-1}
    \leq
    \alpha\, t^{-\alpha-1}.
\]
Because $\alpha < 1$, we have $2\alpha < \alpha+1$, so
$t^{-2\alpha}$ decays more slowly than $t^{-\alpha-1}$.  Therefore, there
exists $T_1 \geq T_0$ such that
\[
    (cK-b)t^{-2\alpha}
    \geq
    K\bigl(t^{-\alpha} - (t+1)^{-\alpha}\bigr)
    \qquad \text{for all } t \geq T_1.
\]
Note that $a_t < \infty$ for each finite $t$.  Indeed, $F$ is
$L$-smooth on $\mathcal{K}$ and $\nabla F(x^*) = 0$, so
$\|\nabla F(x_t)\| = \|\nabla F(x_t) - \nabla F(x^*)\| \leq L\|\Delta_t\|$,
and, therefore,
\[
    \mathbb{E}\|x_{t+1}-x_t\|^2
    \leq
    2\eta_t^2 p_+^2\bigl(L^2\,\mathbb{E}\|\Delta_t\|^2 + \sigma^2\bigr)
\]
by condition~(ii) and Assumption~\ref{ass:var}.  Starting from
$\mathbb{E}\|\Delta_1\|^2 < \infty$ (which holds by hypothesis on the
initialization $x_1$),
$\|\Delta_{t+1}\|^2 \leq 2\|\Delta_t\|^2 + 2\|x_{t+1}-x_t\|^2$ then
gives $\mathbb{E}\|\Delta_t\|^2 < \infty$ for each finite $t$ by
induction on $t$, and $L$-smoothness together with
$\nabla F(x^*) = 0$ yields
$a_t \leq (L/2)\,\mathbb{E}\|\Delta_t\|^2 < \infty$.  In particular,
$\sup_{1 \leq t \leq T_1} a_t\, t^{\alpha}$ is finite, so we may
increase $K$ further if necessary so that $a_t \leq K t^{-\alpha}$ for
all $1 \leq t \leq T_1$.  Then for any
$t \geq T_1$, if $a_t \leq K t^{-\alpha}$,
\begin{align*}
    a_{t+1}
    &\leq
    (1 - c t^{-\alpha})K t^{-\alpha} + b t^{-2\alpha} \\
    &= K t^{-\alpha} - (cK-b)t^{-2\alpha} \\
    &\leq K t^{-\alpha} - K\bigl(t^{-\alpha} - (t+1)^{-\alpha}\bigr) \\
    &= K(t+1)^{-\alpha}.
\end{align*}
Combining the base-case bound for $1 \leq t \leq T_1$ with the step
just displayed, induction on $t$ yields
\[
    a_t \leq K t^{-\alpha}
    \qquad \text{for all } t \geq 1.
\]
Finally, strong convexity gives
\[
    F(x_t)-F(x^*)
    \geq
    \frac{\mu}{2}\|x_t-x^*\|^2
    =
    \frac{\mu}{2}\|\Delta_t\|^2.
\]
Taking expectations,
\[
    \mathbb{E}\|\Delta_t\|^2
    \leq
    \frac{2}{\mu}\,a_t
    \leq
    \frac{2K}{\mu}\,t^{-\alpha}.
\]
Thus, Assumption~\ref{ass:iterate} holds with $C_\Delta = 2K/\mu$.
\end{proof}

\subsection{Proof of Lemma~\ref{lem:decomp}}
\label{app:decomp}
\begin{proof}[Proof of Lemma~\ref{lem:decomp}]
By Assumptions~\ref{ass:convex}--\ref{ass:quadratic}, $H = \nabla^2 F(x^*)$
satisfies $H \succeq \mu I$ and is, therefore, invertible; each $P_t$ is
positive definite, so the matrices
\[
    A_t := \eta_t^{-1}(P_tH)^{-1},
    \qquad t \ge 1,
\]
are well defined.

Expanding the error recursion~\eqref{eq:recursion},
\[
    \Delta_{t+1}
    =
    \Delta_t - \eta_t P_t H \Delta_t - \eta_t P_t \xi_t - \eta_t P_t u_t,
\]
rearranging gives
\[
    \Delta_t - \Delta_{t+1}
    =
    \eta_t P_t\bigl(H\Delta_t + \xi_t + u_t\bigr).
\]
Left-multiplying both sides by $A_t$ and using
the cancellation $(P_tH)^{-1}P_t = H^{-1}$, we obtain
\[
    A_t(\Delta_t-\Delta_{t+1})
    =
    \Delta_t + H^{-1}\xi_t + H^{-1}u_t.
\]
Equivalently,
\begin{equation}
    \Delta_t
    =
    A_t(\Delta_t-\Delta_{t+1})
    - H^{-1}\xi_t
    - H^{-1}u_t.
    \label{eq:lemma_one_step}
\end{equation}

Summing \eqref{eq:lemma_one_step} over $t=1,\ldots,n$ yields
\begin{equation}\label{eq:sum_delta}
    \sum_{t=1}^n \Delta_t
    =
    \sum_{t=1}^n A_t(\Delta_t-\Delta_{t+1})
    -
    H^{-1}\sum_{t=1}^n \xi_t
    -
    H^{-1}\sum_{t=1}^n u_t.
\end{equation}
The first term is expanded by Abel summation:
\begin{align*}
    \sum_{t=1}^n A_t(\Delta_t-\Delta_{t+1})
    &=
    \sum_{t=1}^n A_t\Delta_t
    -
    \sum_{t=1}^n A_t\Delta_{t+1} \\
    &=
    A_1\Delta_1 - A_n\Delta_{n+1}
    + \sum_{t=2}^n (A_t-A_{t-1})\Delta_t.
\end{align*}
Substituting this identity into \eqref{eq:sum_delta} gives
\[
    \sum_{t=1}^n \Delta_t
    =
    A_1\Delta_1 - A_n\Delta_{n+1}
    + \sum_{t=2}^n (A_t-A_{t-1})\Delta_t
    - H^{-1}\sum_{t=1}^n \xi_t
    - H^{-1}\sum_{t=1}^n u_t.
\]
Dividing by $n$ and using
$\overline{\Delta}_n = n^{-1}\sum_{t=1}^n \Delta_t$
proves~\eqref{eq:decomposition} with remainder~\eqref{eq:remainder}.
\end{proof}

\subsection{Proof of Proposition~\ref{prop:variance}}
\label{app:proof_prop_variance}
\begin{proof}[Proof of Proposition~\ref{prop:variance}]
Because $\{\xi_t\}$ is a martingale difference sequence
(Assumption~\ref{ass:mg}), the cross-terms vanish: for $s < t$, the tower
property gives
$\mathbb{E}[\xi_s \xi_t^\top]
= \mathbb{E}\bigl[\xi_s\,\mathbb{E}[\xi_t^\top \mid \mathcal{F}_{t-1}]\bigr]
= 0$,
and transposing covers $s > t$.  Therefore,
\[
  \mathbb{E}\!\left\|\frac{1}{n}H^{-1}\sum_{t=1}^n \xi_t\right\|^2
  = \frac{1}{n^2}\sum_{t=1}^n
    \mathbb{E}\|H^{-1}\xi_t\|^2
  = \frac{1}{n^2}\sum_{t=1}^n
    \mathrm{Tr}\!\bigl(H^{-1}\mathbb{E}[\xi_t\xi_t^\top]H^{-1}\bigr),
\]
where the second equality uses the identity
$\mathbb{E}\|Ax\|^2 = \mathrm{Tr}(A\,\mathbb{E}[xx^\top]A^\top)$.
By Assumption~\ref{ass:var} and the tower property,
$\mathbb{E}[\xi_t\xi_t^\top]
= \mathbb{E}\bigl[\mathbb{E}[\xi_t\xi_t^\top \mid \mathcal{F}_{t-1}]\bigr]
\preceq \overline{S}$.
Applying the congruence map $A \mapsto H^{-1}AH^{-1}$, therefore, yields
\[
  H^{-1}\mathbb{E}[\xi_t\xi_t^\top]H^{-1}
  \preceq
  H^{-1}\overline{S}H^{-1}.
\]
Taking traces gives
\[
  \mathrm{Tr}\!\bigl(H^{-1}\mathbb{E}[\xi_t\xi_t^\top]H^{-1}\bigr)
  \le
  \mathrm{Tr}(H^{-1}\overline{S}\,H^{-1})
\]
for every $t$.  Summing over $t$ in the display above yields the claim.
This is the standard Polyak--Ruppert variance calculation
\citep{polyak1992acceleration,chen2020statistical}.
\end{proof}

\subsection{Proof of Proposition~\ref{prop:taylor}}
\label{app:proof_prop_taylor}
\begin{proof}[Proof of Proposition~\ref{prop:taylor}]
Recall $T_n = -n^{-1}H^{-1}\sum_{t=1}^n u_t$.
Assumption~\ref{ass:quadratic} gives $\|u_t\| \leq L_R \|\Delta_t\|^2$, so
\[
  \mathbb{E}\|T_n\|
  = \mathbb{E}\!\left\|\frac{1}{n}H^{-1}\sum_{t=1}^n u_t\right\|
  \leq \frac{\|H^{-1}\|_{\mathrm{op}}}{n}\sum_{t=1}^n \mathbb{E}\|u_t\|
  \leq \frac{\|H^{-1}\|_{\mathrm{op}}L_R}{n}\sum_{t=1}^n
    \mathbb{E}\|\Delta_t\|^2.
\]
Using Assumption~\ref{ass:iterate},
\[
  \frac{\|H^{-1}\|_{\mathrm{op}}L_R}{n}\sum_{t=1}^n
    \mathbb{E}\|\Delta_t\|^2
  \leq \frac{\|H^{-1}\|_{\mathrm{op}}L_R C_\Delta}{n}\sum_{t=1}^n t^{-\alpha}
  = \mathcal{O}(n^{-\alpha}),
\]
where the last step uses $\sum_{t=1}^n t^{-\alpha} = \mathcal{O}(n^{1-\alpha})$
for $\alpha < 1$.
Multiplying by $\sqrt{n}$ yields
$\sqrt{n}\,\mathbb{E}\|T_n\|
= \mathcal{O}(n^{1/2-\alpha}) \to 0$.

For the $L^2$ bound, Minkowski's inequality together with
the fourth-moment iterate bound gives
\[
  \bigl(\mathbb{E}\|T_n\|^2\bigr)^{1/2}
  \le
  \frac{\|H^{-1}\|_{\mathrm{op}}}{n}\sum_{t=1}^n
    \bigl(\mathbb{E}\|u_t\|^2\bigr)^{1/2}
  \le
  \frac{\|H^{-1}\|_{\mathrm{op}}L_R}{n}\sum_{t=1}^n
    \bigl(\mathbb{E}\|\Delta_t\|^4\bigr)^{1/2}.
\]
Since $\bigl(\mathbb{E}\|\Delta_t\|^4\bigr)^{1/2} \le C_{\Delta,4}^{1/2} t^{-\alpha}$,
we obtain
\[
  \bigl(\mathbb{E}\|T_n\|^2\bigr)^{1/2}
  \le
  \frac{\|H^{-1}\|_{\mathrm{op}}L_R C_{\Delta,4}^{1/2}}{n}
  \sum_{t=1}^n t^{-\alpha}
  = \mathcal{O}(n^{-\alpha}).
\]
Squaring gives $\mathbb{E}\|T_n\|^2 = \mathcal{O}(n^{-2\alpha})$,
and, therefore,
$\mathbb{E}\|\sqrt{n}\,T_n\|^2 = \mathcal{O}(n^{1-2\alpha}) \to 0$
since $\alpha > 1/2$.
\end{proof}

\subsection{Proof of Theorem~\ref{thm:sharp_threshold}}
\label{app:proof_thm_sharp_threshold}
\begin{proof}[Proof of Theorem~\ref{thm:sharp_threshold}]
By the hypotheses, the uniform a.s.\ bound $\|M_t\|_{\mathrm{op}} \leq
C_P$ holds for all $t \geq 1$, and the stabilization-rate bound
$\|M_t - M_{t-1}\|_{\mathrm{op}} \leq C_M\, t^{-\beta}$ holds almost
surely for all $t \geq 2$.  Since $C_P$ and $C_M$ are deterministic,
these a.s.\ bounds pass inside $L^2$-expectations: for every $t \geq 1$,
\[
    \|M_t\,\Delta_t\|_{L^2} \leq C_P\,\|\Delta_t\|_{L^2},
\]
and for every $t \geq 2$,
\[
    \|(M_t - M_{t-1})\Delta_t\|_{L^2}
    \leq C_M\, t^{-\beta}\,\|\Delta_t\|_{L^2}.
\]
In particular, $\|M_1\,\Delta_1\|_{L^2} \leq C_P\,\|\Delta_1\|_{L^2}$
is a finite deterministic constant by Assumption~\ref{ass:iterate}.

Applying Minkowski's inequality to~\eqref{eq:remainder}, we obtain
\begin{align*}
  \|R_n(\{P_t\})\|_{L^2}
  &\leq \frac{1}{n}\Bigl[
      \eta_1^{-1}\|M_1 \Delta_1\|_{L^2}
      + C_P\,\eta_n^{-1}\|\Delta_{n+1}\|_{L^2}\\
  &\qquad\quad
      + \sum_{t=2}^n
        \bigl\|(\eta_t^{-1}M_t - \eta_{t-1}^{-1}M_{t-1})\Delta_t\bigr\|_{L^2}
    \Bigr].
\end{align*}

The first boundary term is $O(n^{-1})$, since
$\eta_1^{-1}\|M_1 \Delta_1\|_{L^2} \leq \eta_1^{-1}\,C_P\,\|\Delta_1\|_{L^2}$
is a finite deterministic constant independent of~$n$.  For the second
boundary term, using
$\eta_n^{-1} = \eta_0^{-1}n^\alpha$ and
$\|\Delta_{n+1}\|_{L^2} \leq C_\Delta^{1/2}(n+1)^{-\alpha/2}$, we obtain
\[
  \frac{C_P}{n}\,\eta_n^{-1}\|\Delta_{n+1}\|_{L^2}
  =
  O\!\left(n^{\alpha/2-1}\right).
\]

For the summation term, using~\eqref{eq:coupled_diff} together with
the deterministic a.s.\ bounds, we split
\begin{align*}
  &\frac{1}{n}\sum_{t=2}^n
    \bigl\|(\eta_t^{-1}M_t - \eta_{t-1}^{-1}M_{t-1})\Delta_t\bigr\|_{L^2} \\
  &\qquad\le
    \frac{C_P}{n}\sum_{t=2}^n
    |\eta_t^{-1}-\eta_{t-1}^{-1}|\,\|\Delta_t\|_{L^2} \\
  &\qquad\quad+
    \frac{1}{n}\sum_{t=2}^n
    \eta_{t-1}^{-1}\,\|(M_t - M_{t-1})\Delta_t\|_{L^2}.
\end{align*}
By the mean value theorem,
\[
  |\eta_t^{-1}-\eta_{t-1}^{-1}|
  =
  \eta_0^{-1}\bigl(t^\alpha-(t-1)^\alpha\bigr)
  =
  O(t^{\alpha-1}),
\]
and $\|\Delta_t\|_{L^2} = O(t^{-\alpha/2})$, so the step-size
contribution satisfies
\[
  \frac{C_P}{n}\sum_{t=2}^n
  O(t^{\alpha-1})\,O(t^{-\alpha/2})
  =
  O\!\left(n^{\alpha/2-1}\right).
\]
For the stabilization contribution, the stabilization-rate bound gives
$\|(M_t - M_{t-1})\Delta_t\|_{L^2} \leq C_M\, t^{-\beta}\,\|\Delta_t\|_{L^2}$
for all $t \geq 2$, and since $\eta_{t-1}^{-1}\, t^{-\beta}
= O(t^{\alpha-\beta})$, we obtain
\[
  \frac{C_M}{n}\sum_{t=2}^n
  O(t^{\alpha-\beta})\,O(t^{-\alpha/2})
  =
  O\!\left(n^{\alpha/2-\beta}\right)
\]
when $\beta \le 1$.  The step-size contribution is already
$O(n^{\alpha/2-1})$, which matches the stabilization contribution at
$\beta=1$; stabilization rates $\beta > 1$, therefore, provide no
additional improvement, and we capture this by setting
$\tilde\beta := \min\{1,\beta\}$, so that the combined sum satisfies
\[
  \frac{1}{n}\sum_{t=2}^n
  O\!\left(t^{\alpha-\tilde\beta}\right)\,O(t^{-\alpha/2})
  =
  O\!\left(n^{\alpha/2-\tilde\beta}\right).
\]
Hence,
\[
  \|R_n(\{P_t\})\|_{L^2}
  =
  O(n^{-1}) + O(n^{\alpha/2-1}) + O\!\left(n^{\alpha/2-\tilde\beta}\right)
  =
  O\!\left(n^{\alpha/2-\tilde\beta}\right),
\]
and so
\[
  \sqrt{n}\,\|R_n(\{P_t\})\|_{L^2}
  =
  O\!\left(n^{(\alpha+1)/2-\tilde\beta}\right).
\]
This converges to zero whenever
\[
  \tilde\beta > \frac{\alpha+1}{2}.
\]
Because $(\alpha+1)/2 < 1$, this condition is equivalent to
$\beta > (\alpha+1)/2 = \alpha/2 + 1/2$, which proves the claim.
\end{proof}

\subsection{Proof of Proposition~\ref{thm:sharpness}}
\label{app:proof_thm_sharpness}
\begin{proof}[Proof of Proposition~\ref{thm:sharpness}]
We work in one dimension and construct deterministic scalar sequences, so both
operator norms and $L^2$-norms coincide with absolute value.

Let
\[
C_\beta
:=
\sup_{t\ge2}
\left|
\sum_{s=2}^t (-1)^s s^{-\beta}
\right|
< \infty,
\]
which is finite by the Leibniz alternating-series test, since
$\sum_{s\ge2}(-1)^s s^{-\beta}$ converges for every $\beta>0$.  Choose
constants $c_\Delta,c_M>0$ and $m_0>c_M C_\beta$, and define
\[
\Delta_t := c_\Delta (-1)^t t^{-\alpha/2},
\qquad
M_t := m_0 + c_M\sum_{s=2}^t (-1)^s s^{-\beta},
\qquad
P_t := M_t^{-1},
\]
where the empty-sum convention gives $M_1 = m_0$.
Then
\[
m_0-c_M C_\beta \le M_t \le m_0+c_M C_\beta \qquad (t\ge2),
\]
so $\{M_t\}$ is bounded and bounded away from zero, hence, $\{P_t\}$ is a
bounded positive scalar preconditioner sequence. Moreover,
\[
M_t-M_{t-1}=c_M(-1)^t t^{-\beta},
\]
so
\[
|M_t-M_{t-1}| = c_M t^{-\beta} = \Theta(t^{-\beta}),
\qquad
|M_t|=\Theta(1),
\qquad
|\Delta_t|=c_\Delta t^{-\alpha/2}=\Theta(t^{-\alpha/2}).
\]

Now decompose (4) as
\[
R_n(\{P_t\}) = B_n + S_n + D_n,
\]
where
\[
B_n
:=
\frac{1}{n}
\Bigl[
\eta_1^{-1}M_1\Delta_1
-
\eta_n^{-1}M_n\Delta_{n+1}
\Bigr],
\]
\[
S_n
:=
\frac{1}{n}
\sum_{t=2}^n
\bigl(\eta_t^{-1}-\eta_{t-1}^{-1}\bigr)M_t\Delta_t,
\]
and
\[
D_n
:=
\frac{1}{n}
\sum_{t=2}^n
\eta_{t-1}^{-1}(M_t-M_{t-1})\Delta_t.
\]
Since $\eta_t=\eta_0 t^{-\alpha}$, we have
\[
\eta_{t-1}^{-1}=\eta_0^{-1}(t-1)^\alpha.
\]

We first estimate the dynamic term $D_n$. By construction, the alternating
signs cancel, so every summand is positive:
\[
D_n
=
\frac{c_\Delta c_M}{\eta_0 n}
\sum_{t=2}^n
(t-1)^\alpha t^{-\beta-\alpha/2}.
\]
For $t\ge2$,
\[
(t-1)^\alpha
=
t^\alpha\left(1-\frac{1}{t}\right)^\alpha,
\]
and, therefore,
\[
2^{-\alpha} t^{\alpha/2-\beta}
\le
(t-1)^\alpha t^{-\beta-\alpha/2}
\le
t^{\alpha/2-\beta}.
\]
Hence,
\[
D_n
=
\Theta\!\left(
\frac{1}{n}\sum_{t=2}^n t^{\alpha/2-\beta}
\right).
\]
Because $\beta\le(\alpha+1)/2$, we have
\[
\frac{\alpha}{2}-\beta \ge -\frac12 > -1,
\]
so the standard power-sum estimate gives
\[
\frac{1}{n}\sum_{t=2}^n t^{\alpha/2-\beta}
=
\Theta\!\left(n^{\alpha/2-\beta}\right).
\]
Therefore,
\[
D_n=\Theta\!\left(n^{\alpha/2-\beta}\right).
\]

Next, the boundary term satisfies
\[
|B_n|
\le
\frac{1}{n}
\Bigl[
\eta_1^{-1}|M_1||\Delta_1|
+
\eta_n^{-1}|M_n||\Delta_{n+1}|
\Bigr]
=
O(n^{-1}) + O\!\left(n^{\alpha/2-1}\right)
=
O\!\left(n^{\alpha/2-1}\right).
\]

For the step-size term, boundedness of $M_t$ and the mean-value estimate
\[
\eta_t^{-1}-\eta_{t-1}^{-1}
=
\eta_0^{-1}\bigl(t^\alpha-(t-1)^\alpha\bigr)
=
O(t^{\alpha-1})
\]
yield, for a constant $C$ depending on $\sup_t|M_t|$, $c_\Delta$, and
$\eta_0$,
\[
|S_n|
\le
\frac{C}{n}\sum_{t=2}^n t^{\alpha-1} t^{-\alpha/2}
=
\frac{C}{n}\sum_{t=2}^n t^{\alpha/2-1}
=
O\!\left(n^{\alpha/2-1}\right).
\]

Since $\alpha<1$, we have
\[
\frac{\alpha}{2}-1 < -\frac12 \le \frac{\alpha}{2}-\beta,
\]
and, thus,
\[
|B_n|+|S_n|
=
o\!\left(n^{\alpha/2-\beta}\right)
=
o(D_n).
\]
Therefore,
\[
R_n(\{P_t\}) = D_n + o(D_n),
\]
which implies
\[
|R_n(\{P_t\})|
=
\Theta\!\left(n^{\alpha/2-\beta}\right).
\]
Since the construction is deterministic, this is exactly
\[
\|R_n(\{P_t\})\|_{L^2}
=
\Theta\!\left(n^{\alpha/2-\beta}\right).
\]

Multiplying by $\sqrt{n}$ yields
\[
\sqrt{n}\,\|R_n(\{P_t\})\|_{L^2}
=
\begin{cases}
\Theta\!\left(n^{(\alpha+1)/2-\beta}\right),
& 0<\beta<(\alpha+1)/2,\\[1mm]
\Theta(1),
& \beta=(\alpha+1)/2.
\end{cases}
\]
Hence, $\sqrt{n}\,\|R_n(\{P_t\})\|_{L^2}\not\to0$ in general when
$\beta\le(\alpha+1)/2$, which establishes sharpness relative to the
hypotheses of Theorem~\ref{thm:sharp_threshold}.
\end{proof}

\subsection{Proof of Corollary~\ref{cor:adaptive_iterate}}
\label{app:proof_cor_adaptive}
\begin{proof}[Proof of Corollary~\ref{cor:adaptive_iterate}]
\textit{Part~(i): SA-AdaGrad.}
The recursion
\[
    C_t = (1-\rho_t)C_{t-1} + \rho_t(g_t g_t^\top + \epsilon I)
\]
starts from $C_0 = \epsilon I$.  Since $\|g_t\| \le G$ almost surely,
\[
    0 \preceq g_t g_t^\top \preceq \|g_t\|^2 I \preceq G^2 I,
\]
and, therefore,
\[
    \epsilon I \preceq g_t g_t^\top + \epsilon I \preceq (G^2+\epsilon)I.
\]
Because $0 < \rho_t \le 1$, the update is a convex combination of matrices in
the order interval $[\epsilon I,\,(G^2+\epsilon)I]$.  By induction,
\[
    \epsilon I \preceq C_t \preceq (G^2+\epsilon)I
    \qquad \text{almost surely for all } t \ge 0.
\]
Since $P_t = C_{t-1}^{-1/2}$ and the inverse square-root reverses Loewner
order on positive definite matrices,
\[
    \frac{1}{\sqrt{G^2+\epsilon}}\,I
    \preceq
    P_t
    \preceq
    \frac{1}{\sqrt{\epsilon}}\,I.
\]

\textit{Part~(i) continued: SA-RMSProp.}
The recursion
\[
    v_t = (1-\rho_t)v_{t-1} + \rho_t(g_t \odot g_t + \epsilon \mathbf{1})
\]
starts from $v_0 = \epsilon \mathbf{1}$.  Since $\|g_t\| \le G$ almost surely,
each coordinate satisfies
\[
    0 \le g_{t,j}^2 \le G^2,
    \qquad
    \epsilon \le g_{t,j}^2 + \epsilon \le G^2+\epsilon.
\]
Again using $0 < \rho_t \le 1$, induction gives
\[
    \epsilon \mathbf{1} \le v_t \le (G^2+\epsilon)\mathbf{1}
    \qquad \text{coordinatewise almost surely for all } t \ge 0.
\]
Hence,
\[
    \frac{1}{\sqrt{G^2+\epsilon}}\,I
    \preceq
    P_t = \mathrm{Diag}(v_{t-1})^{-1/2}
    \preceq
    \frac{1}{\sqrt{\epsilon}}\,I.
\]

\textit{Part~(ii).}
Under Assumption~\ref{ass:mg} and hypothesis~(a),
\[
    \nabla F(x_t)
    = \mathbb{E}[g_t \mid \mathcal{F}_{t-1}],
    \qquad
    \|\nabla F(x_t)\|
    \le \mathbb{E}[\|g_t\| \mid \mathcal{F}_{t-1}]
    \le G.
\]
Therefore,
\[
    \|\xi_t\|
    = \|g_t - \nabla F(x_t)\|
    \le 2G,
\]
so that
$\xi_t\xi_t^\top \preceq \|\xi_t\|^2 I \preceq 4G^2 I$ almost surely, and,
therefore,
\[
    \mathbb{E}[\xi_t\xi_t^\top \mid \mathcal{F}_{t-1}]
    \preceq 4G^2 I.
\]
Hypothesis~(a), thus, gives the explicit constant
$\overline S = 4G^2 I$ in Assumption~\ref{ass:var}.
For both SA-AdaGrad and SA-RMSProp, part~(i) gives the uniform ellipticity
bounds required in condition~(ii) of Proposition~\ref{prop:iterate_bound},
with
\[
    p_- = \frac{1}{\sqrt{G^2+\epsilon}},
    \qquad
    p_+ = \frac{1}{\sqrt{\epsilon}}.
\]
Hypothesis~(b) verifies condition~(i) of
Proposition~\ref{prop:iterate_bound}, while Assumptions~\ref{ass:mg}, \ref{ass:var}, and~\ref{ass:convex} are now in force.  Therefore,
Proposition~\ref{prop:iterate_bound} yields a constant $C_\Delta > 0$ such
that
\[
    \mathbb{E}\|\Delta_t\|^2 \le C_\Delta\, t^{-\alpha}
    \qquad \text{for all } t \ge 1,
\]
which is precisely Assumption~\ref{ass:iterate}.
\end{proof}

\subsection{Proof of Corollary~\ref{cor:ons_iterate}}
\label{app:proof_cor_ons}
\begin{proof}[Proof of Corollary~\ref{cor:ons_iterate}]
\textit{Part~(i).}
Because $\epsilon \in [h_-,h_+]$, the initialization satisfies
\[
    h_- I \preceq B_0 \preceq h_+ I.
\]
Since $0 < c \le 1$, we have $0 < \rho_t = c/t \le 1$ for every $t \ge 1$.
Hence,
\[
    B_t = (1-\rho_t)B_{t-1} + \rho_t \hat{H}_t
\]
is a convex combination of $B_{t-1}$ and $\hat{H}_t$.  If
\[
    h_- I \preceq B_{t-1} \preceq h_+ I
\]
and
\[
    h_- I \preceq \hat{H}_t \preceq h_+ I,
\]
then
\[
    h_- I
    \preceq
    (1-\rho_t)B_{t-1} + \rho_t \hat{H}_t
    \preceq
    h_+ I.
\]
By induction,
\[
    h_- I \preceq B_t \preceq h_+ I
    \qquad \text{almost surely for all } t \ge 0.
\]
Since inversion reverses the Loewner order on $\mathbb{S}_{++}^d$, taking
inverses yields
\[
    h_+^{-1}I \preceq P_t = B_{t-1}^{-1} \preceq h_-^{-1}I
    \qquad \text{almost surely for all } t \ge 1.
\]

\textit{Part~(ii).}
Part~(i) verifies condition~(ii) of Proposition~\ref{prop:iterate_bound} with
$p_- = h_+^{-1}$ and $p_+ = h_-^{-1}$, while hypothesis~(b) verifies
condition~(i).  Together with Assumptions~\ref{ass:mg}--\ref{ass:convex},
Proposition~\ref{prop:iterate_bound}, therefore, yields a constant $C_\Delta > 0$
such that
\[
    \mathbb{E}\|\Delta_t\|^2 \le C_\Delta\, t^{-\alpha}
    \qquad \text{for all } t \ge 1,
\]
which is precisely Assumption~\ref{ass:iterate} for the SA-ONS recursion.
\end{proof}

\subsection{Proof of Theorem~\ref{thm:sa_stab}}
\label{app:proof_thm_stab}
\begin{proof}[Proof of Theorem~\ref{thm:sa_stab}]
\textit{Part~(i).}
Since $\rho_t \in (0,1]$ and both $Q_{t-1} \succeq \epsilon I$ and
$W_t \succeq \epsilon I$, the convex combination preserves
$Q_t \succeq \epsilon I$.  The $L^2$ boundedness follows from the contraction
\[
    \|Q_t\|_{L^2(\mathrm{op})}
    \le (1-\rho_t)\|Q_{t-1}\|_{L^2(\mathrm{op})} + \rho_t C_W,
\]
and, therefore, $\sup_t \|Q_t\|_{L^2(\mathrm{op})} < \infty$ by induction.

\textit{Part~(ii).}
From~\eqref{eq:sa_general_stab},
$Q_t - Q_{t-1} = \rho_t(W_t - Q_{t-1})$, so
\[
    \|Q_t - Q_{t-1}\|_{L^2(\mathrm{op})}
    \le \rho_t\bigl(C_W + \sup_s\|Q_s\|_{L^2(\mathrm{op})}\bigr)
    = \mathcal{O}(t^{-1}).
\]

\textit{Part~(iii).}
We treat the two spectral maps in turn.

\textit{Case $\varphi = A^{-1}$.}
Here $M_t = H^{-1}Q_{t-1}$, so
$M_t - M_{t-1} = H^{-1}(Q_{t-1}-Q_{t-2})$. Hence,
\[
    \|M_t - M_{t-1}\|_{L^2(\mathrm{op})}
    \le \|H^{-1}\|_{\mathrm{op}}
    \|Q_{t-1}-Q_{t-2}\|_{L^2(\mathrm{op})}
    = \mathcal{O}(t^{-1}).
\]
Moreover,
\[
    \|M_t\|_{L^2(\mathrm{op})}
    \le \|H^{-1}\|_{\mathrm{op}} \|Q_{t-1}\|_{L^2(\mathrm{op})}
    = \mathcal{O}(1).
\]

\textit{Case $\varphi = A^{-1/2}$.}
Here $M_t = H^{-1}Q_{t-1}^{1/2}$.  By the Fr\'{e}chet
derivative bound for the matrix square root,
\[
    \left\|Q_{t-1}^{1/2} - Q_{t-2}^{1/2}\right\|_{\mathrm{op}}
    \le \frac{1}{2\sqrt{\epsilon}}\,
    \left\|Q_{t-1}-Q_{t-2}\right\|_{\mathrm{op}}
    \qquad \text{almost surely},
\]
using $Q_t \succeq \epsilon I$ from~(i) and the perturbation bound for
operator monotone functions
\citep[Theorem~V.3.3]{bhatia1997matrix}.  Taking $L^2(\mathrm{op})$ norms gives
\[
    \left\|Q_{t-1}^{1/2} - Q_{t-2}^{1/2}\right\|_{L^2(\mathrm{op})}
    \le \frac{1}{2\sqrt{\epsilon}}\,
    \left\|Q_{t-1}-Q_{t-2}\right\|_{L^2(\mathrm{op})}
    = \mathcal{O}(t^{-1}).
\]
Therefore,
\[
    \|M_t - M_{t-1}\|_{L^2(\mathrm{op})}
    \le \|H^{-1}\|_{\mathrm{op}}
    \left\|Q_{t-1}^{1/2} - Q_{t-2}^{1/2}\right\|_{L^2(\mathrm{op})}
    = \mathcal{O}(t^{-1}).
\]
Also,
\begin{align*}
    \|M_t\|_{L^2(\mathrm{op})}
    &= \|H^{-1}Q_{t-1}^{1/2}\|_{L^2(\mathrm{op})}
    \le \|H^{-1}\|_{\mathrm{op}}\,
    \|Q_{t-1}^{1/2}\|_{L^2(\mathrm{op})} \\
    &\le \|H^{-1}\|_{\mathrm{op}}\,
    \|Q_{t-1}\|_{L^2(\mathrm{op})}^{1/2}
    = \mathcal{O}(1).
\end{align*}

\textit{Part~(iv).}
If $\|W_t\|_{\mathrm{op}} \le C_{\max}$ almost surely, then
$\sup_t \|Q_t\|_{\mathrm{op}} \le \max(C_{\max}, \|Q_0\|_{\mathrm{op}})$ by
induction, and $\|Q_t - Q_{t-1}\|_{\mathrm{op}} = \mathcal{O}(t^{-1})$ follows
from the same calculation as~(ii).  The rest of the proof carries over with
$L^2(\mathrm{op})$ norms replaced by operator norms.

\textit{Convergence under a deterministic target.}
Assume now that $\mathbb{E}[W_t \mid \mathcal{F}_{t-1}] = \bar W$ almost surely
for all $t \ge 1$, where $\bar W \in \mathbb{S}_{++}^d$ is deterministic.
Writing $E_t := Q_t - \bar W$ and $Z_t := W_t - \bar W$, we obtain
\[
    E_t = (1-\rho_t)E_{t-1} + \rho_t Z_t,
    \qquad
    \mathbb{E}[Z_t \mid \mathcal{F}_{t-1}] = 0.
\]
Because $\sup_t \|W_t\|_{L^2(\mathrm{op})} < \infty$, the martingale increments
$Z_t$ have uniformly bounded second moments.  To reduce to a nonnegative
process, fix an entry $(i,j)$ and set $e_t := (E_t)_{ij}$,
$z_t := (Z_t)_{ij}$, so that $e_t = (1-\rho_t)e_{t-1} + \rho_t z_t$
with $\mathbb{E}[z_t \mid \mathcal{F}_{t-1}] = 0$.  Squaring and taking
conditional expectations gives
\[
    \mathbb{E}[e_t^2 \mid \mathcal{F}_{t-1}]
    \;=\; (1-\rho_t)^2\,e_{t-1}^2
    \;+\; \rho_t^2\,\mathbb{E}[z_t^2 \mid \mathcal{F}_{t-1}]
    \;=\; (1+\rho_t^2)\,e_{t-1}^2
    \;-\; 2\rho_t\,e_{t-1}^2
    \;+\; \gamma_t,
\]
where $\gamma_t := \rho_t^2\,\mathbb{E}[z_t^2 \mid \mathcal{F}_{t-1}] \ge 0$.
Since $\sum_t \rho_t^2 < \infty$, $\sum_t \rho_t = \infty$, and
$\sum_t \gamma_t < \infty$ a.s.\ (because
$\sum_t \mathbb{E}[\gamma_t]
 = \sum_t \rho_t^2\,\mathbb{E}[z_t^2]
 \le \sup_t \mathbb{E}\|Z_t\|_{\mathrm{op}}^2 \cdot \sum_t \rho_t^2 < \infty$),
the Robbins--Siegmund almost-supermartingale convergence
theorem~\citep[Theorem~1]{robbins1971convergence} applied to the
nonnegative process $e_t^2$ yields $e_t^2 \to V_\infty$ a.s.\ and
$\sum_t \rho_t\,e_{t-1}^2 < \infty$ a.s.; the divergence of
$\sum_t \rho_t$ forces $V_\infty = 0$.  Since $d$ is finite,
$E_t \to 0$ entrywise and, hence,
$Q_t \xrightarrow{a.s.} \bar W$.  Since
$Q_t \succeq \epsilon I$ almost surely for all $t$, the maps
$A \mapsto A^{-1}$ and $A \mapsto A^{-1/2}$ are continuous on the relevant
uniformly elliptic set, so $P_t \xrightarrow{a.s.} P_\infty := \varphi(\bar W)$.
Consequently,
\[
    M_t = (P_tH)^{-1} \xrightarrow{a.s.} (P_\infty H)^{-1} = M_\infty.
\]
\end{proof}

\subsection{Proof of Corollary~\ref{cor:sa_stab}}
\label{app:proof_cor_stab}
\begin{proof}[Proof of Corollary~\ref{cor:sa_stab}]
We verify the hypotheses of Theorem~\ref{thm:sa_stab} for each construction.

\textit{SA-AdaGrad.}
Set $Q_t = C_t$, $W_t = g_tg_t^\top + \epsilon I$, and $\varphi = A^{-1/2}$.
Assumption~\ref{ass:g4} gives
\[
    \|W_t\|_{L^2(\mathrm{op})}
    =
    \left(\mathbb{E}\|g_t g_t^\top + \epsilon I\|_{\mathrm{op}}^2\right)^{1/2}
    \le
    \left(\mathbb{E}\|g_t\|^4\right)^{1/2} + \epsilon
    \le C_{g,4}^{1/2} + \epsilon.
\]
Since $W_t \succeq \epsilon I$, the hypotheses
of Theorem~\ref{thm:sa_stab}(iii) are satisfied.  If $\|g_t\| \le G$,
then $\|W_t\|_{\mathrm{op}} \le G^2 + \epsilon$, so part~(iv) applies.

\textit{SA-RMSProp.}
Set $Q_t = \mathrm{Diag}(v_t)$,
$W_t = \mathrm{Diag}(g_t \odot g_t + \epsilon\mathbf{1})$, and
$\varphi = A^{-1/2}$.
Each diagonal entry satisfies $[W_t]_{jj} = g_{t,j}^2 + \epsilon \ge \epsilon$, and
\[
    \|W_t\|_{L^2(\mathrm{op})}
    =
    \left(\mathbb{E}\|g_t \odot g_t + \epsilon\mathbf{1}\|_\infty^2\right)^{1/2}
    \le
    \left(\mathbb{E}\|g_t\|^4\right)^{1/2} + \epsilon
    \le C_{g,4}^{1/2} + \epsilon.
\]
Thus, the hypotheses of Theorem~\ref{thm:sa_stab}(iii) are satisfied.
If $\|g_t\| \le G$, then
$\|W_t\|_{\mathrm{op}} = \|g_t \odot g_t + \epsilon\mathbf{1}\|_\infty
\le G^2 + \epsilon$, so part~(iv) applies.

\textit{SA-ONS.}
Set $Q_t = B_t$, $W_t = \hat{H}_t$, and $\varphi = A^{-1}$, with
$C_W = C_H^{1/2}$.  The stated assumptions give
$W_t \succeq \epsilon I$ almost surely,
$\sup_t \|W_t\|_{L^2(\mathrm{op})} \le C_H^{1/2}$, and
$\mathbb{E}[W_t \mid \mathcal{F}_{t-1}] = H$.  Hence, the hypotheses of
Theorem~\ref{thm:sa_stab}(iii) and of its deterministic-target clause are
satisfied with $\bar W = H$, so $P_\infty = H^{-1}$ and
$M_\infty = (H^{-1}H)^{-1} = I$.  If $\|\hat{H}_t\|_{\mathrm{op}} \le h_+$
almost surely, then part~(iv) applies as well.
\end{proof}

\subsection{Proof of Corollary~\ref{cor:threshold}}
\label{app:proof_cor_threshold}
\begin{proof}[Proof of Corollary~\ref{cor:threshold}]
Since $\eta_t^{-1} = \eta_0^{-1} t^\alpha$, the mean value theorem gives
$t^\alpha - (t-1)^\alpha = \alpha\,\xi^{\alpha-1}$ for some
$\xi \in (t-1,t)$, and since $\alpha < 1$ and $\xi \ge t-1$, we obtain
\[
    \left|\eta_t^{-1} - \eta_{t-1}^{-1}\right|
    = \eta_0^{-1}\left|t^\alpha - (t-1)^\alpha\right|
    = \mathcal{O}(t^{\alpha-1}).
\]
Also, $\eta_{t-1}^{-1} = \mathcal{O}(t^\alpha)$.  Theorem~\ref{thm:sa_stab}
and Corollary~\ref{cor:sa_stab} give
\[
    \|M_t\|_{L^2(\mathrm{op})} = \mathcal{O}(1),
    \qquad
    \|M_t - M_{t-1}\|_{L^2(\mathrm{op})} = \mathcal{O}(t^{-1}),
\]
where both bounds are supplied directly by Theorem~\ref{thm:sa_stab}(iii)
and recorded for each SA construction in Corollary~\ref{cor:sa_stab}.  Applying the triangle inequality to the
decomposition~\eqref{eq:coupled_diff} and substituting these bounds yields
$\|\eta_t^{-1}M_t - \eta_{t-1}^{-1}M_{t-1}\|_{L^2(\mathrm{op})}
= \mathcal{O}(t^{\alpha-1})$.  Under the additional almost-sure boundedness
assumptions in Corollary~\ref{cor:sa_stab}, the Lipschitz-spectral-map
argument of Theorem~\ref{thm:sa_stab} runs pathwise with
$\|\cdot\|_{\mathrm{op}}$ replacing $\|\cdot\|_{L^2(\mathrm{op})}$, giving
$\|M_t - M_{t-1}\|_{\mathrm{op}} = \mathcal{O}(t^{-1})$ and
$\|M_t\|_{\mathrm{op}} = \mathcal{O}(1)$ almost surely, so
Definition~\ref{def:beta} yields the pathwise stabilization rate $\beta = 1$.
In particular, since $\alpha \in (1/2,1)$, we have $1 > (\alpha+1)/2$, so
this pathwise rate lies strictly above the threshold in
Theorem~\ref{thm:sharp_threshold}.
\end{proof}

\subsection{Proof of Proposition~\ref{prop:l2_remainder}}
\label{app:proof_prop_l2_remainder}
\begin{proof}[Proof of Proposition~\ref{prop:l2_remainder}]
With $A_t := \eta_t^{-1}(P_tH)^{-1} = \eta_t^{-1}M_t$, the
decomposition~\eqref{eq:remainder} reads
$n\,R_n(\{P_t\}) = B_1 - B_2 + B_3$, where
\[
    B_1 := A_1\,\Delta_1,
    \qquad
    B_2 := A_n\,\Delta_{n+1},
    \qquad
    B_3 := \sum_{t=2}^n (A_t - A_{t-1})\,\Delta_t.
\]
Taking norms, applying the triangle inequality pointwise, taking
expectations, and multiplying by $\sqrt{n}$ yields
\begin{equation}\label{eq:l2_split}
    \mathbb{E}\bigl\|\sqrt{n}\,R_n\bigr\|
    \;\le\;
    \frac{1}{\sqrt{n}}\Bigl(
        \mathbb{E}\|B_1\|
        \;+\; \mathbb{E}\|B_2\|
        \;+\; \mathbb{E}\|B_3\|
    \Bigr).
\end{equation}
Throughout we use Cauchy--Schwarz in the form
$\mathbb{E}[\|X\|_{\mathrm{op}}\|Y\|]
\le \|X\|_{L^2(\mathrm{op})}\,\|Y\|_{L^2}$,
together with the iterate bound
$\|\Delta_t\|_{L^2} \le C_\Delta^{1/2}\,t^{-\alpha/2}$ supplied
directly by Assumption~\ref{ass:iterate}.

\textit{Boundary term $B_1$.}
Since $\eta_1^{-1} = \eta_0^{-1}$ is deterministic and
$\|M_1\|_{L^2(\mathrm{op})} = \mathcal{O}(1)$ by hypothesis,
\[
    \mathbb{E}\|B_1\|
    \;\le\; \eta_0^{-1}\,\|M_1\|_{L^2(\mathrm{op})}\,\|\Delta_1\|_{L^2}
    \;=\; \mathcal{O}(1),
\]
so $(1/\sqrt{n})\,\mathbb{E}\|B_1\| = \mathcal{O}(n^{-1/2})$.

\textit{Boundary term $B_2$.}
With $\eta_n^{-1} = \eta_0^{-1} n^{\alpha}$ and
$\|\Delta_{n+1}\|_{L^2} \le C_\Delta^{1/2}(n+1)^{-\alpha/2}
\le C_\Delta^{1/2} n^{-\alpha/2}$,
\[
    \mathbb{E}\|B_2\|
    \;\le\; \eta_0^{-1} n^{\alpha}\,
            \|M_n\|_{L^2(\mathrm{op})}\,\|\Delta_{n+1}\|_{L^2}
    \;=\; \mathcal{O}\!\left(n^{\alpha/2}\right),
\]
so $(1/\sqrt{n})\,\mathbb{E}\|B_2\| = \mathcal{O}(n^{(\alpha-1)/2})$.

\textit{Summation term $B_3$.}
By the triangle inequality and Cauchy--Schwarz applied term-by-term,
\[
    \mathbb{E}\|B_3\|
    \;\le\;
    \sum_{t=2}^n
      \bigl\|\eta_t^{-1}M_t - \eta_{t-1}^{-1}M_{t-1}\bigr\|_{L^2(\mathrm{op})}
      \,\|\Delta_t\|_{L^2}.
\]
Substituting the $L^2$ stabilization hypothesis
$\|\eta_t^{-1}M_t - \eta_{t-1}^{-1}M_{t-1}\|_{L^2(\mathrm{op})}
= \mathcal{O}(t^{\alpha-1})$ and the iterate bound
$\|\Delta_t\|_{L^2} = \mathcal{O}(t^{-\alpha/2})$,
\[
    \mathbb{E}\|B_3\|
    \;=\; \sum_{t=2}^n \mathcal{O}\!\left(t^{\alpha/2 - 1}\right).
\]
Since $\alpha \in (1/2,1)$, the integral comparison
$\sum_{t=1}^n t^{\alpha/2 - 1} \le (2/\alpha)\,n^{\alpha/2}$ gives
$\mathbb{E}\|B_3\| = \mathcal{O}(n^{\alpha/2})$, hence,
$(1/\sqrt{n})\,\mathbb{E}\|B_3\| = \mathcal{O}(n^{(\alpha-1)/2})$.

Combining the three contributions in~\eqref{eq:l2_split},
\[
    \mathbb{E}\bigl\|\sqrt{n}\,R_n\bigr\|
    \;=\; \mathcal{O}(n^{-1/2})
        + \mathcal{O}(n^{(\alpha-1)/2})
    \;=\; \mathcal{O}\!\left(n^{(\alpha-1)/2}\right),
\]
which vanishes as $n \to \infty$ since $\alpha < 1$.  Convergence in
$L^1$ implies convergence in probability.
\end{proof}

\subsection{Proof of Proposition~\ref{prop:tight}}
\label{app:proof_prop_tight}
\begin{proof}[Proof of Proposition~\ref{prop:tight}]
With $A_t := \eta_t^{-1}(P_tH)^{-1} = \eta_t^{-1}M_t$, the
decomposition~\eqref{eq:remainder} reads
\[
    n\,R_n(\{P_t\})
    \;=\; A_1\Delta_1 \;-\; A_n\Delta_{n+1}
      \;+\; \sum_{t=2}^n(A_t - A_{t-1})\Delta_t.
\]
Minkowski's inequality gives
\begin{equation}\label{eq:tight_split}
    n\,\|R_n\|_{L^2}
    \;\le\; B_1 + B_2 + B_3,
\end{equation}
where
\[
    B_1 := \|A_1\Delta_1\|_{L^2},
    \qquad
    B_2 := \|A_n\Delta_{n+1}\|_{L^2},
    \qquad
    B_3 := \Bigl\|\sum_{t=2}^n(A_t - A_{t-1})\Delta_t\Bigr\|_{L^2}.
\]
Throughout we use the iterate bound $\|\Delta_t\|_{L^2} \le
C_\Delta^{1/2}\,t^{-\alpha/2}$ from Assumption~\ref{ass:iterate},
together with the almost-sure operator bound
$\|M_t\|_{\mathrm{op}} \le C_P$, which factors out of any $L^2$ norm:
for any random vector~$Y$,
$\bigl\|\|M_t\|_{\mathrm{op}}\|Y\|\bigr\|_{L^2}
\le C_P\,\|Y\|_{L^2}$
by monotonicity of the expectation applied to the a.s.\ inequality
$\|M_t\|_{\mathrm{op}}^2\|Y\|^2 \le C_P^2\|Y\|^2$.

\textit{Boundary term $B_1$.}
Since $\eta_1^{-1}=\eta_0^{-1}$ is deterministic,
\[
    B_1
    \;=\; \eta_0^{-1}\bigl\|\|M_1\|_{\mathrm{op}}\|\Delta_1\|\bigr\|_{L^2}
    \;\le\; \eta_0^{-1}C_P\,\|\Delta_1\|_{L^2}
    \;\le\; \eta_0^{-1}C_P\,C_\Delta^{1/2}
    \;=:\; K_1.
\]

\textit{Boundary term $B_2$.}
With $\eta_n^{-1} = \eta_0^{-1}n^\alpha$ and $\|\Delta_{n+1}\|_{L^2}\le
C_\Delta^{1/2}(n+1)^{-\alpha/2}$,
\[
    B_2
    \;\le\; \eta_0^{-1}n^\alpha\,C_P\,C_\Delta^{1/2}\,(n+1)^{-\alpha/2}
    \;=\; \eta_0^{-1}C_P C_\Delta^{1/2}\cdot n^{\alpha/2}
          \Bigl(\tfrac{n}{n+1}\Bigr)^{\alpha/2}
    \;\le\; K_2\,n^{\alpha/2},
\]
where $K_2 := \eta_0^{-1}C_P C_\Delta^{1/2}$, using the identity
$n^\alpha(n+1)^{-\alpha/2} = n^{\alpha/2}\bigl(n/(n+1)\bigr)^{\alpha/2}
\le n^{\alpha/2}$.

\textit{Summation term $B_3$.}
A second application of Minkowski's inequality commutes the sum out of
the $L^2$ norm:
\[
    B_3
    \;\le\; \sum_{t=2}^n \bigl\|(A_t - A_{t-1})\Delta_t\bigr\|_{L^2}.
\]
We split the sum at the index~$t_0$.  For $2\le t < t_0$, the triangle
inequality and the uniform bound $\|M_t\|_{\mathrm{op}}\le C_P$ give the
a.s.\ bound
\[
    \|A_t - A_{t-1}\|_{\mathrm{op}}
    \;\le\; \eta_t^{-1}C_P + \eta_{t-1}^{-1}C_P
    \;\le\; 2\eta_0^{-1}C_P\,t_0^\alpha,
\]
since $\eta_t^{-1}=\eta_0^{-1}t^\alpha\le\eta_0^{-1}t_0^\alpha$ on this
range.  Combined with $\|\Delta_t\|_{L^2}\le C_\Delta^{1/2}$ for
$t\ge 1$, summing over the at most $t_0-2$ indices gives
\[
    \sum_{t=2}^{t_0-1}\bigl\|(A_t - A_{t-1})\Delta_t\bigr\|_{L^2}
    \;\le\; 2(t_0-2)\,\eta_0^{-1}C_P C_\Delta^{1/2}\,t_0^\alpha
    \;=:\; K_0.
\]
For $t\ge t_0$, the pathwise coupled-difference hypothesis
$\|A_t - A_{t-1}\|_{\mathrm{op}}\le C_{\mathrm{coup}}\,t^{\alpha-1}$
a.s.\ and the iterate bound give
\[
    \sum_{t=t_0}^n\bigl\|(A_t - A_{t-1})\Delta_t\bigr\|_{L^2}
    \;\le\; C_{\mathrm{coup}}\,C_\Delta^{1/2}
            \sum_{t=t_0}^n t^{\alpha/2-1}.
\]
Since $\alpha\in(1/2,1)$, the map $x\mapsto x^{\alpha/2-1}$ is positive
and decreasing on $(0,\infty)$, so an upper Riemann sum comparison
yields
\[
    \sum_{t=1}^n t^{\alpha/2-1}
    \;\le\; \int_0^n x^{\alpha/2-1}\,dx
    \;=\; \tfrac{2}{\alpha}\,n^{\alpha/2}
\]
(the integral converges at~$0$ since $\alpha>0$), and, hence,
\[
    \sum_{t=t_0}^n\bigl\|(A_t - A_{t-1})\Delta_t\bigr\|_{L^2}
    \;\le\; K_3\,n^{\alpha/2},
    \qquad
    K_3 := \tfrac{2}{\alpha}\,C_{\mathrm{coup}}\,C_\Delta^{1/2}.
\]
Combining the two ranges, $B_3 \le K_0 + K_3\,n^{\alpha/2}$.

\textit{Combining contributions.}
Substituting the three bounds into~\eqref{eq:tight_split} and using
$n^{\alpha/2}\ge 1$ to absorb the $n$-independent constants into the
leading term,
\[
    n\,\|R_n\|_{L^2}
    \;\le\; (K_0 + K_1) + (K_2 + K_3)\,n^{\alpha/2}
    \;\le\; (K_0+K_1+K_2+K_3)\,n^{\alpha/2},
\]
so $\|R_n\|_{L^2}\le(K_0+K_1+K_2+K_3)\,n^{\alpha/2-1}$.  Squaring
yields the finite-sample bound
\[
    \mathbb{E}\|R_n\|^2
    \;\le\; C_R\,n^{\alpha-2},
    \qquad
    C_R \;:=\; (K_0+K_1+K_2+K_3)^2,
\]
which depends explicitly on $\eta_0,\alpha,C_\Delta,C_P,
C_{\mathrm{coup}},t_0$ through the definitions of $K_0,\dots,K_3$.
\end{proof}

\subsection{Proof of Proposition~\ref{prop:comparison}}
\label{app:proof_prop_comparison}

The proof invokes the following pathwise version of the classical
Toeplitz lemma; we record it explicitly since we apply it on individual
sample paths rather than in mean.

\begin{lemma}[Pathwise Toeplitz lemma]
\label{lem:pathwise_toeplitz}
Let $\{a_t\}_{t\ge 1}$ be a (possibly random) scalar sequence with
$a_t \to a$ almost surely, and let
$\{w_{t,n} : 1 \le t \le n,\, n \ge 1\}$ be a deterministic triangular
array of nonnegative weights satisfying
\[
    \sum_{t=1}^n w_{t,n} \;\to\; 1
    \qquad\text{and}\qquad
    \max_{1\le t\le n} w_{t,n} \;\to\; 0
    \qquad \text{as } n \to \infty.
\]
Then $\sum_{t=1}^n w_{t,n}\, a_t \to a$ almost surely.
\end{lemma}

\begin{proof}
The classical deterministic Toeplitz lemma states that for any
deterministic sequence $b_t \to b$, the weighted averages
$\sum_{t=1}^n w_{t,n} b_t$ converge to $b$ under the two displayed
conditions on the weights (short $\varepsilon$-argument: split the sum
at a tail index beyond which $|b_t - b| < \varepsilon$, then use
$\sum w_{t,n}\to 1$ on the tail and $\max w_{t,n}\to 0$ on the finite
head).  Fix $\omega$ in the probability-one event on which $a_t \to a$;
applying the deterministic statement sample-path-wise yields the
claim.  The exceptional null set is independent of the (deterministic)
weights, so no measurability subtlety arises.
\end{proof}

\begin{proof}[Proof of Proposition~\ref{prop:comparison}]
We establish $M_t \to M_\infty$ a.s.\ for each construction, then derive
the operator-factor expressions.

\textit{SA-ONS.}
We identify the SA recursion in Theorem~\ref{thm:sa_stab} by setting
$Q_t = B_t$, $W_t = \hat{H}_t$, and $\varphi(A) = A^{-1}$.  The
driving term satisfies $W_t \succeq h_- I$ a.s.\ by the two-sided
spectral bound (Theorem~\ref{thm:sa_stab}'s generic ellipticity floor
$\epsilon$ is identified here with $h_-$),
$\sup_t\|W_t\|_{\mathrm{op}}\le h_+$ a.s., and the conditional-mean
identity $\mathbb{E}[\hat{H}_t\mid\mathcal{F}_{t-1}] = H$ holds by
hypothesis.  The final clause of Theorem~\ref{thm:sa_stab} applied with
$\bar W = H$, therefore, gives $B_t \to H$ and
$M_t = H^{-1}B_{t-1} \to I$ a.s.

\textit{Preliminary ($\Delta_t \to 0$ for SA-AdaGrad and SA-RMSProp).}
For both SA-AdaGrad and SA-RMSProp, the conditional-mean hypothesis of
Theorem~\ref{thm:sa_stab} fails---e.g., for SA-AdaGrad
$\mathbb{E}[W_t \mid \mathcal{F}_{t-1}]
 = \nabla F(x_t)\nabla F(x_t)^\top + \Sigma(x_t) + \epsilon I$ with
$\Sigma(x_t) := \mathbb{E}[\xi_t\xi_t^\top\mid\mathcal{F}_{t-1}]$,
which is not a deterministic constant---so we give a direct pathwise
argument.  We first establish $\Delta_t \to 0$ almost surely by a
Robbins--Siegmund analysis on the objective gap
$V_t := F(x_t) - F(x^*) \ge 0$.  Both constructions satisfy the same
hypotheses~(a)--(b) of Corollary~\ref{cor:adaptive_iterate} and the
same pathwise uniform ellipticity
$p_- I \preceq P_t \preceq p_+ I$ a.s.\ with
$p_- := 1/\sqrt{G^2+\epsilon}$ and $p_+ := 1/\sqrt{\epsilon}$
(Corollary~\ref{cor:adaptive_iterate}(i)), so the argument below
applies verbatim to both.

The preconditioned SGD update $x_{t+1} = x_t - \eta_t P_t g_t$ combined
with $L$-smoothness of $F$ on $\mathcal{K}$
(Corollary~\ref{cor:adaptive_iterate}(b)) yields the descent inequality
\[
    F(x_{t+1}) - F(x_t)
    \;\le\;
    -\eta_t\bigl\langle\nabla F(x_t),P_t g_t\bigr\rangle
    + \tfrac{L\,\eta_t^2}{2}\,\|P_t g_t\|^2.
\]
Take $\mathcal{F}_{t-1}$-conditional expectation.
Using the unbiasedness relation
$\mathbb{E}[g_t\mid\mathcal{F}_{t-1}] = \nabla F(x_t)$,
the $\mathcal{F}_{t-1}$-measurability of $P_t$, and $\|g_t\|\le G$,
\[
    \mathbb{E}[V_{t+1}\mid\mathcal{F}_{t-1}]
    \;\le\;
    V_t - \eta_t\bigl\langle\nabla F(x_t), P_t\nabla F(x_t)\bigr\rangle
    + \tfrac{1}{2}L\,p_+^2\,G^2\,\eta_t^2.
\]
Since $P_t\succeq p_-I$ gives
$\langle\nabla F(x_t),P_t\nabla F(x_t)\rangle\ge p_-\|\nabla F(x_t)\|^2$,
and the Polyak--{\L}ojasiewicz consequence of $\mu$-strong convexity
(Assumption~\ref{ass:convex}) gives
$\|\nabla F(x_t)\|^2\ge 2\mu V_t$, we obtain
\[
    \mathbb{E}[V_{t+1}\mid\mathcal{F}_{t-1}]
    \;\le\;
    (1 - 2p_-\mu\,\eta_t)\,V_t
    + \tfrac{1}{2}L\,p_+^2\,G^2\,\eta_t^2.
\]
With $\eta_t = \eta_0 t^{-\alpha}$ and $\alpha\in(1/2,1)$ we have
$\sum_t\eta_t^2 < \infty$ and $\sum_t\eta_t = \infty$, so the
Robbins--Siegmund theorem
\citep[Theorem~1]{robbins1971convergence}, applied with $\alpha_t\equiv 0$,
$\beta_t := 2p_-\mu\,\eta_t V_t$, and
$\gamma_t := \tfrac{1}{2}L\,p_+^2\,G^2\,\eta_t^2$, yields both the
a.s.\ convergence of $V_t$ to a random limit $V_\infty\ge 0$ and the
pathwise summability
$\sum_t\eta_t V_t < \infty$ a.s.  On the event
$\{V_\infty > 0\}$ one would have $\eta_t V_t \sim \eta_t V_\infty$ for
large $t$, so $\sum_t \eta_t V_t = \infty$ on that event, contradicting
the summability; hence, $V_\infty = 0$ a.s.  By strong convexity
$V_t\ge(\mu/2)\|\Delta_t\|^2$, so
\begin{equation}\label{eq:rs_prelim}
    \Delta_t \to 0 \quad \text{almost surely (SA-AdaGrad and SA-RMSProp).}
\end{equation}

\textit{SA-AdaGrad.}
Set $\bar W := S + \epsilon I$ and decompose $W_t = \bar W + Z_t + b_t$,
where
\[
    Z_t := W_t - \mathbb{E}[W_t \mid \mathcal{F}_{t-1}],
    \qquad
    b_t := \mathbb{E}[W_t \mid \mathcal{F}_{t-1}] - \bar W
         = \nabla F(x_t)\nabla F(x_t)^\top + (\Sigma(x_t) - S).
\]
We record two properties.
\begin{enumerate}
  \item[(a)] \emph{Bounded martingale differences.}
    $\{Z_t\}$ is an $\mathcal{F}_t$-martingale difference sequence with
    $\|Z_t\|_{\mathrm{op}} \le 2(G^2 + \epsilon)$ a.s., since
    $\|W_t\|_{\mathrm{op}} \le G^2 + \epsilon$ under $\|g_t\| \le G$.
  \item[(b)] \emph{Vanishing bias.}  By the local expansion
    $\|\nabla F(x_t)\| \le \|H\|_{\mathrm{op}}\|\Delta_t\|
    + L_R\|\Delta_t\|^2$ and \eqref{eq:rs_prelim},
    $\|\nabla F(x_t)\nabla F(x_t)^\top\|_{\mathrm{op}} \to 0$ a.s.;
    by continuity of $\Sigma(\cdot)$ at~$x^*$,
    $\|\Sigma(x_t) - S\|_{\mathrm{op}} \to 0$ a.s.  Hence,
    $\|b_t\|_{\mathrm{op}} \to 0$ almost surely.
\end{enumerate}

\textit{Telescoping identity (case $c=1$).}
With $\rho_t = 1/t$ and $C_0 = \epsilon I$, the recursion~\eqref{eq:sa_general_stab} rewrites as
$t\,C_t = (t-1)\,C_{t-1} + W_t$; induction on $t$ gives the exact identity
\[
    C_n
    \;=\; \frac{1}{n}\sum_{t=1}^n W_t
    \;=\; \bar W \;+\; \frac{1}{n}\sum_{t=1}^n Z_t
    \;+\; \frac{1}{n}\sum_{t=1}^n b_t
    \qquad (n \ge 1).
\]

\emph{Martingale term.}  Fix entry indices $(i,j)$ and set
$y_t := (Z_t)_{ij}/t$.  By~(a), $\{y_t\}$ is a scalar martingale
difference sequence with
$\sum_{t=1}^\infty \mathbb{E}\,y_t^2
\le 4(G^2+\epsilon)^2 \sum_{t=1}^\infty t^{-2} < \infty$,
so Doob's $L^2$ martingale convergence theorem gives a.s.\ convergence
of $\sum_{t=1}^\infty y_t$.  Kronecker's lemma then yields
$n^{-1}\sum_{t=1}^n (Z_t)_{ij} \to 0$ a.s.  Since $d$ is finite and
operator and entrywise norms are equivalent,
$\|n^{-1}\sum_{t=1}^n Z_t\|_{\mathrm{op}} \to 0$ a.s.

\emph{Bias term.}  By~(b), $\|b_t\|_{\mathrm{op}} \to 0$ a.s., and the
Ces\`aro lemma gives
$\|n^{-1}\sum_{t=1}^n b_t\|_{\mathrm{op}} \to 0$ a.s.

Combining, $C_n \to \bar W = S + \epsilon I$ almost surely.

\textit{General $c \in (0,1]$.}
For general $c$, induction on the recursion gives the closed form
\[
    C_n \;=\; \pi_{0,n}\,C_0
    \;+\; \sum_{t=1}^n w_{t,n}\,W_t,
    \qquad
    w_{t,n} := \frac{c\,\pi_{t,n}}{t},
    \quad
    \pi_{s,n} := \prod_{r=s+1}^n\!\bigl(1 - \tfrac{c}{r}\bigr),
\]
with the convention $\pi_{n,n} = 1$ (empty product).  From
$\log(1 - c/r) = -c/r + \mathcal{O}(r^{-2})$ and harmonic-sum
telescoping,
\[
    \log \pi_{t,n}
    \;=\; \sum_{r=t+1}^n \log(1 - c/r)
    \;=\; -c\,\log(n/t) + \mathcal{O}(1)
    \quad \text{uniformly in } 1\le t\le n,
\]
so $\pi_{t,n}\asymp (t/n)^c$; in particular $\pi_{0,n}\to 0$, and the
weights $\{w_{t,n}\}$ satisfy the Toeplitz row-sum conditions
\[
    \sum_{t=1}^n w_{t,n} \;\to\; 1,
    \qquad
    \max_{1\le t\le n} w_{t,n} \;\to\; 0.
\]
Note the exact partition-of-unity identity
\[
    \pi_{0,n} \;+\; \sum_{t=1}^n w_{t,n} \;=\; 1
    \qquad \text{for every } n\ge 1,
\]
obtained by specializing the closed form to the constant input
$W_t\equiv I$ with $C_0 = I$ (in which case $C_n\equiv I$).  Inserting
the decomposition $W_t = \bar W + Z_t + b_t$ into the closed form and
collapsing the deterministic contributions via this identity,
\[
    C_n - \bar W
    \;=\; \pi_{0,n}(C_0 - \bar W)
    \;+\; \sum_{t=1}^n w_{t,n}\,Z_t
    \;+\; \sum_{t=1}^n w_{t,n}\,b_t;
\]
the first (deterministic) term vanishes as $n\to\infty$ since
$\pi_{0,n}\to 0$.

\emph{Bias term.}  By (b), $\|b_t\|_{\mathrm{op}} \to 0$ a.s., and
Lemma~\ref{lem:pathwise_toeplitz} applied entrywise (using the Toeplitz
row-sum conditions on $\{w_{t,n}\}$ displayed above) yields
$\|\sum_{t=1}^n w_{t,n}\,b_t\|_{\mathrm{op}}\to 0$ a.s.

\emph{Martingale term.}  Fix entry indices $(i,j)$ and set
$\zeta_t := (Z_t)_{ij}$; by (a), $\{\zeta_t, \mathcal{F}_t\}$ is a
bounded scalar martingale-difference sequence with
$|\zeta_t|\le 2(G^2+\epsilon)$ a.s.  Define the auxiliary scalar
martingale
\[
    N_t \;:=\; \sum_{s=1}^t \frac{\zeta_s}{s},
    \qquad N_0 := 0.
\]
Its predictable quadratic variation is bounded:
$\sum_{s=1}^\infty \mathbb{E}(\zeta_s/s)^2
\le 4(G^2+\epsilon)^2\sum_{s=1}^\infty s^{-2} < \infty$, so Doob's $L^2$
martingale convergence theorem gives
$N_t \to N_\infty$ a.s.\ to a finite $\mathcal{F}_\infty$-measurable
limit.  From $\pi_{t,n} = (1 - c/(t+1))\,\pi_{t+1,n}$ we read off the
forward difference
\[
    \pi_{t+1,n} - \pi_{t,n} \;=\; \frac{c\,\pi_{t+1,n}}{t+1}.
\]
Using the identity $\zeta_t/t = N_t - N_{t-1}$ and Abel's summation by
parts,
\begin{align*}
    \sum_{t=1}^n w_{t,n}\,\zeta_t
    &\;=\; c\sum_{t=1}^n \pi_{t,n}\,(N_t - N_{t-1}) \\
    &\;=\; c\,\pi_{n,n}\,N_n
         \;-\; c\sum_{t=1}^{n-1} N_t\,(\pi_{t+1,n} - \pi_{t,n}) \\
    &\;=\; c\,N_n
         \;-\; c\sum_{s=2}^n w_{s,n}\,N_{s-1},
\end{align*}
where the last equality uses $\pi_{n,n}=1$, the forward-difference
identity above, and the reindexing $s=t+1$.  Since $N_n\to N_\infty$
a.s., the boundary term satisfies $c N_n\to c N_\infty$ a.s.  For the
weighted sum, the sequence $N_{s-1}$ (with $N_0 := 0$) converges a.s.\
to $N_\infty$, and the weights $\{w_{s,n}\}$ satisfy the Toeplitz
conditions displayed above, so Lemma~\ref{lem:pathwise_toeplitz} gives
$\sum_{s=1}^n w_{s,n}\,N_{s-1} \to N_\infty$ a.s.\ (the $s=1$ contribution
is $w_{1,n}\cdot 0 = 0$, so the sum from $s=2$ has the same limit).
Therefore,
\[
    \sum_{t=1}^n w_{t,n}\,\zeta_t
    \;\longrightarrow\;
    c\,N_\infty - c\,N_\infty
    \;=\; 0
    \qquad \text{almost surely.}
\]
Since $d$ is finite and the operator norm is equivalent to the maximum
entrywise magnitude (up to a factor depending only on $d$),
$\|\sum_{t=1}^n w_{t,n}\,Z_t\|_{\mathrm{op}}\to 0$ a.s.

Combining the three terms, $C_n \to \bar W = S + \epsilon I$ almost
surely for every $c\in(0,1]$.

\textit{SA-RMSProp.}
Using the preliminary \eqref{eq:rs_prelim} already established for both
SA-AdaGrad and SA-RMSProp, the same decomposition and telescoping
argument applies coordinatewise to the diagonal recursion
$v_t = (1-\rho_t)\,v_{t-1} + \rho_t(g_t\odot g_t + \epsilon\mathbf{1})$,
with $\bar w_i := (s_{\mathrm{diag}} + \epsilon\mathbf{1})_i$ in the
$i$-th coordinate and scalar analogues of properties (a)--(b): the
scalar martingale differences $(Z_t)_i$ are bounded by $2(G^2+\epsilon)$
under $\|g_t\|\le G$, and the scalar bias $(b_t)_i\to 0$ a.s.\ by the
same local-expansion and continuity argument.  Doob's $L^2$
convergence, Kronecker's lemma (for $c=1$), and Abel summation with
Lemma~\ref{lem:pathwise_toeplitz} (for general $c\in(0,1]$) yield
$\mathrm{Diag}(v_t)
\to \mathrm{Diag}(\mathrm{diag}(S) + \epsilon\mathbf{1})$ a.s.

\textit{Passage to $M_t$.}
The map $\varphi : A \mapsto A^{-1/2}$ is continuous on
$\{A \in \mathbb{S}^d_{++} : A \succeq \epsilon I\}$ (indeed locally
Lipschitz by the spectral calculus of a smooth function), so the
continuous mapping theorem applied pathwise gives
$\varphi(C_{t-1}) \to \varphi(\bar W) = (S+\epsilon I)^{-1/2}$ a.s.,
and, hence,
\[
    M_t \;=\; \bigl(\varphi(C_{t-1})\,H\bigr)^{-1}
    \;\to\; H^{-1}(S+\epsilon I)^{1/2} \qquad \text{a.s.}
\]
The diagonal analogue holds for SA-RMSProp.

\textit{Operator-factor expressions.}
Under the full hypotheses~(a)--(b) of the respective primitive
corollary assumed in each branch, the pathwise uniform ellipticity of
$P_t$ supplied by Corollaries~\ref{cor:adaptive_iterate}
and~\ref{cor:ons_iterate} yields the pathwise operator bound
$\sup_t\|M_t\|_{\mathrm{op}} \le C_P$ a.s.\ (this is the constant~$C_P$
appearing in the hypothesis of Proposition~\ref{prop:tight}), and
Theorem~\ref{thm:sa_stab}(iv) combined with
Corollary~\ref{cor:threshold} yields the pathwise coupled one-step
variation
$\|\eta_t^{-1}M_t - \eta_{t-1}^{-1}M_{t-1}\|_{\mathrm{op}}
= \mathcal{O}(t^{\alpha-1})$ a.s.  The pathwise hypotheses of
Proposition~\ref{prop:tight} are, therefore, met.  Since $M_t \to M_\infty$
almost surely with the limits identified above, the \emph{asymptotic
operator factor}
$\bar C_P := \|M_\infty\|_{\mathrm{op}}$---the $t\to\infty$ limit of
$\|M_t\|_{\mathrm{op}}$, whose pathwise uniform bound supplies the
constant $C_P$ in Proposition~\ref{prop:tight}---is identified.  Substituting
$M_\infty$ for each construction:
SA-ONS gives $M_\infty = I$, so
$\bar C_P^{\mathrm{ONS}} = 1$;
SA-AdaGrad gives
$M_\infty = H^{-1}(S+\epsilon I)^{1/2}$;
SA-RMSProp gives
$M_\infty = H^{-1}\mathrm{Diag}(\mathrm{diag}(S)+\epsilon\mathbf{1})^{1/2}$.
The ratio identity
$\bar C_P^{\mathrm{ONS}} / \bar C_P^{\mathrm{Id}} = \lambda_{\min}(H)$
is immediate from the definitions.
\end{proof}

\subsection{Proof of Theorem~\ref{thm:clt}}
\label{app:proof_thm_clt}
\begin{proof}[Proof of Theorem~\ref{thm:clt}]
By Lemma~\ref{lem:decomp}:
$\sqrt{n}\,\overline{\Delta}_n = \widetilde{\Xi}_n + \sqrt{n}\,T_n
+ \sqrt{n}\,R_n(\{P_t\})$, where
$\widetilde{\Xi}_n := -\frac{1}{\sqrt{n}}H^{-1}\sum_{t=1}^n \xi_t$ and
$T_n := -\frac{1}{n}H^{-1}\sum_{t=1}^n u_t$.

We show $\widetilde{\Xi}_n \xrightarrow{d} \mathcal{N}(0, H^{-1}SH^{-1})$.
Define the triangular-array increments
\[
    Z_{n,t} := -\frac{1}{\sqrt{n}}H^{-1}\xi_t,
    \qquad 1 \le t \le n,
\]
so that $\widetilde{\Xi}_n = \sum_{t=1}^n Z_{n,t}$.
Since $x_t$ is $\mathcal{F}_{t-1}$-measurable and $\zeta_t$ is independent of
$\mathcal{F}_{t-1}$, Assumption~\ref{ass:mg} gives
\[
    \mathbb{E}[Z_{n,t}\mid\mathcal{F}_{t-1}] = 0,
    \qquad
    \mathbb{E}[\xi_t\xi_t^\top\mid\mathcal{F}_{t-1}]
    = \mathbb{E}[\xi(x,\zeta_t)\xi(x,\zeta_t)^\top]\big|_{x=x_t}
    = \Sigma(x_t),
\]
so $\{Z_{n,t},\mathcal{F}_t\}_{1 \le t \le n}$ is a martingale-difference
array.  Its predictable quadratic variation is, therefore,
\[
    V_n
    :=
    \sum_{t=1}^n
    \mathbb{E}[Z_{n,t}Z_{n,t}^\top \mid \mathcal{F}_{t-1}]
    =
    H^{-1}\left(
        \frac{1}{n}\sum_{t=1}^n \Sigma(x_t)
    \right)H^{-1}.
\]
By Assumption~\ref{ass:iterate}, $\mathbb{E}\|\Delta_t\|^2 \to 0$, hence,
$x_t \to x^*$ in probability.  The continuity of $\Sigma(\cdot)$ at $x^*$,
therefore, implies $\Sigma(x_t) \to S$ in probability.  Moreover,
\[
    \|\Sigma(x_t)\|_{\mathrm{op}}
    \le
    \mathbb{E}\!\left[\|\xi_t\|^2 \mid \mathcal{F}_{t-1}\right],
\]
and Jensen's inequality gives
\[
    \sup_{t \ge 1}
    \mathbb{E}\!\left[
        \mathbb{E}\!\left(\|\xi_t\|^2 \mid \mathcal{F}_{t-1}\right)^{1+\delta/2}
    \right]
    \le
    \sup_{t \ge 1}\mathbb{E}\|\xi_t\|^{2+\delta}
    < \infty.
\]
Thus, $\{\Sigma(x_t)\}_{t\ge1}$ is uniformly integrable, so
$\Sigma(x_t) \to S$ in $L^1$, and Ces\`aro averaging yields
\[
    \frac{1}{n}\sum_{t=1}^n \Sigma(x_t)
    \xrightarrow{L^1}
    S.
\]
Consequently,
\[
    V_n \xrightarrow{L^1} H^{-1}SH^{-1}.
\]

It remains to verify the Lindeberg condition.  For any $\varepsilon > 0$, let
\[
    L_n(\varepsilon)
    :=
    \sum_{t=1}^n
    \mathbb{E}\!\left[
        \|Z_{n,t}\|^2
        \mathbf{1}_{\{\|Z_{n,t}\|>\varepsilon\}}
        \,\middle|\, \mathcal{F}_{t-1}
    \right].
\]
Using $\|Z_{n,t}\| = n^{-1/2}\|H^{-1}\xi_t\|$ and the inequality
$a^2\mathbf{1}_{\{a>b\}} \le a^{2+\delta}b^{-\delta}$, we obtain
\[
    \mathbb{E}[L_n(\varepsilon)]
    \le
    \frac{\|H^{-1}\|_{\mathrm{op}}^{2+\delta}}
         {\varepsilon^\delta n^{1+\delta/2}}
    \sum_{t=1}^n \mathbb{E}\|\xi_t\|^{2+\delta}
    \le
    \frac{C}{n^{\delta/2}}
    \xrightarrow[n\to\infty]{} 0
\]
for a finite constant $C$.  Hence, $L_n(\varepsilon) \to 0$ in $L^1$ and,
therefore, in probability.  Applying the scalar martingale central
limit theorem of Hall and Heyde \citep[Theorem~3.2]{hall1980martingale}
to the scalar martingale-difference array $\{\theta^\top Z_{n,t}\}$ for
each fixed $\theta\in\mathbb{R}^d$---whose predictable quadratic
variation $\theta^\top V_n\theta$ converges in $L^1$ to
$\theta^\top H^{-1}SH^{-1}\theta$ and whose Lindeberg condition is
inherited from that of $\{Z_{n,t}\}$ via
$|\theta^\top Z_{n,t}|\le\|\theta\|\,\|Z_{n,t}\|$, so that
$\sum_{t=1}^n
\mathbb{E}[|\theta^\top Z_{n,t}|^2
\mathbf{1}_{\{|\theta^\top Z_{n,t}|>\varepsilon\}}
\mid\mathcal{F}_{t-1}]
\le\|\theta\|^2\,L_n(\varepsilon/\|\theta\|)\to 0$---yields
$\theta^\top\widetilde\Xi_n\xrightarrow{d}
\mathcal{N}(0,\theta^\top H^{-1}SH^{-1}\theta)$.  The Cram\'er--Wold
device then gives
\[
    \widetilde{\Xi}_n = \sum_{t=1}^n Z_{n,t}
    \xrightarrow{d}
    \mathcal{N}(0,H^{-1}SH^{-1}).
\]

For the Taylor remainder,
Proposition~\ref{prop:taylor} gives
$\|\sqrt{n}\,T_n\|_{L^1} = \mathcal{O}(n^{1/2-\alpha}) \to 0$,
since $\alpha > 1/2$.  Hence, $\sqrt{n}\,T_n \xrightarrow{p} 0$.

For the dynamic remainder, $\sqrt{n}\,R_n \to 0$ in probability
holds by hypothesis.  (In the default tier, the $L^2(\mathrm{op})$
route of Proposition~\ref{prop:l2_remainder} gives
$\mathbb{E}\|\sqrt{n}\,R_n\| = \mathcal{O}(n^{(\alpha-1)/2})\to 0$,
i.e.\ convergence in $L^1$ and, hence, in probability; alternatively,
in the refinement tier, under the pathwise hypotheses of
Proposition~\ref{prop:tight}, one has
$\mathbb{E}\|\sqrt{n}\,R_n\|^2 = n \cdot \mathbb{E}\|R_n\|^2
\leq C_R\, n^{\alpha-1} \to 0$, giving convergence in $L^2$ and,
hence, in probability.)

Combining the three terms via Slutsky's theorem:
since $\widetilde{\Xi}_n \xrightarrow{d} \mathcal{N}(0, H^{-1}SH^{-1})$ and
$\sqrt{n}\,T_n \xrightarrow{p} 0$ and $\sqrt{n}\,R_n \xrightarrow{p} 0$,
Slutsky's theorem gives that
$\widetilde{\Xi}_n + \sqrt{n}\,T_n + \sqrt{n}\,R_n$
converges in distribution to $\mathcal{N}(0, H^{-1}SH^{-1})$.
\end{proof}

\subsection{Proof of Corollary~\ref{cor:gauss_approx}}
\label{app:proof_thm_gauss_approx}
\begin{proof}[Proof of Corollary~\ref{cor:gauss_approx}]
By Lemma~\ref{lem:decomp},
\[
  W_n
  :=
  \sqrt{n}\,\overline{\Delta}_n
  =
  \widetilde{\Xi}_n + E_n + D_n,
\]
where $\widetilde{\Xi}_n := -\frac{1}{\sqrt{n}}H^{-1}\sum_{t=1}^n \xi_t$
is the rescaled martingale term from the theorem statement,
\[
  E_n := \sqrt{n}\,T_n,
  \qquad
  D_n := \sqrt{n}\,R_n.
\]
Let $Z \sim \mathcal{N}(0,V)$.

By the dual representation~\eqref{eq:wasserstein_dist}, for any
1-Lipschitz function $h$,
\begin{align*}
  \left|\mathbb{E}h(W_n) - \mathbb{E}h(Z)\right|
  &\le
  \left|\mathbb{E}h(\widetilde{\Xi}_n) - \mathbb{E}h(Z)\right|
  +
  \left|\mathbb{E}h(\widetilde{\Xi}_n + E_n + D_n)
        - \mathbb{E}h(\widetilde{\Xi}_n)\right| \\
  &\le
  \Delta_{\mathrm{mart},W}(n)
  +
  \mathbb{E}\|E_n + D_n\| \\
  &\le
  \Delta_{\mathrm{mart},W}(n)
  +
  \mathbb{E}\|E_n\|
  +
  \mathbb{E}\|D_n\|.
\end{align*}
Taking the supremum over all such $h$ gives
\[
  d_W\!\left(
    \mathcal{L}(W_n),
    \Phi_V
  \right)
  \le
  \Delta_{\mathrm{mart},W}(n)
  +
  \mathbb{E}\|E_n\|
  +
  \mathbb{E}\|D_n\|.
\]

For the Taylor remainder, Proposition~\ref{prop:taylor} gives
\[
  \mathbb{E}\|E_n\|
  =
  \sqrt{n}\,\mathbb{E}\|T_n\|
  =
  \mathcal{O}(n^{1/2-\alpha}).
\]

For the dynamic remainder, by Cauchy--Schwarz and the assumed bound
$\mathbb{E}\|R_n\|^2 \le C_R n^{\alpha-2}$,
\[
  \mathbb{E}\|D_n\|
  \le
  \left(\mathbb{E}\|D_n\|^2\right)^{1/2}
  =
  \left(n\,\mathbb{E}\|R_n\|^2\right)^{1/2}
  =
  \mathcal{O}(n^{(\alpha-1)/2}).
\]

Substituting the Taylor and dynamic remainder bounds into the Wasserstein
comparison yields~\eqref{eq:gauss_approx}.
\end{proof}

\section{Additional Experiments Diagnostics}
\label{sec:experiments}

This supplement collects additional experiments diagnostics that complement
the main text.  Unless otherwise noted, the experiments use the same setup
described in Section~\ref{sec:simulation}.

\subsection*{CLT Convergence Diagnostics}

For a finer look at the CLT convergence mechanism, we examine the error
decomposition from Lemma~\ref{lem:decomp}.  Since the synthetic objective
is exactly quadratic (linear regression, $\lambda = 0$), the Taylor
remainder vanishes ($T_n = 0$) and the decomposition simplifies to
\[
    \overline{x}_n - x^*
    = -\frac{1}{n}\,H^{-1} \sum_{t=1}^{n} \xi_t + R_n,
\]
where $\xi_t = g_t - H(x_t - x^*)$ is a martingale difference and $R_n$
is the remainder driven by the non-stationary dynamic preconditioner.
(The implementation includes a safety gradient clip at $\|g_t\| = 500$.
The diagnostic quantities plotted below are computed from the clipped
iterates; the theoretical decomposition above applies exactly to the
unclipped recursion, but the two coincide whenever clipping is inactive.)
The CLT $\sqrt{n}(\overline{x}_n - x^*) \xrightarrow{d}
\mathcal{N}(0,\, H^{-1} S H^{-1})$ requires $\sqrt{n}\,R_n \to 0$ in
probability, which Theorem~\ref{thm:sharp_threshold} guarantees when the
stabilization rate exceeds the threshold; we examine this condition both
visually (Figure~\ref{fig:diagnostics}) and numerically via the empirical
log--log slope of $\sqrt{n}\,\|R_n\|$ versus~$n$
(Table~\ref{tab:scaled_slopes}).

\begin{figure}[tbp]
\centering
\includegraphics[width=\textwidth]{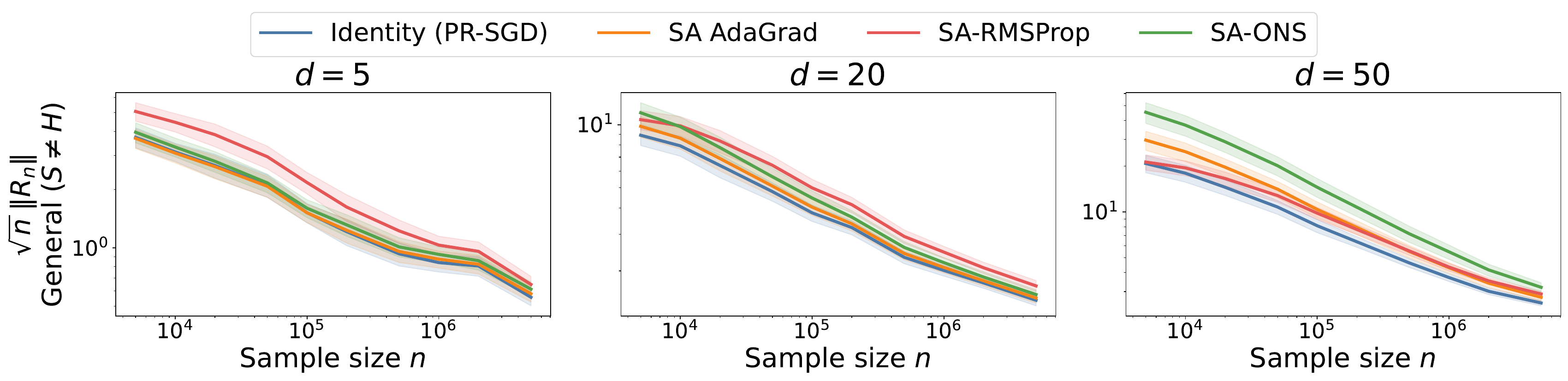}
\caption{CLT convergence diagnostics in the general regime
  ($S \neq H$).  Each curve corresponds to one of the four methods
  (Identity, SA-AdaGrad, SA-RMSProp, SA-ONS).  Columns correspond to
  dimensions $d = 5, 20, 50$; the horizontal axis is the sample
  size~$n$.  The vertical axis shows the scaled dynamic remainder
  $\sqrt{n}\,\|R_n\|$---the preconditioner-specific quantity whose
  vanishing is the key condition for the CLT (Theorem~\ref{thm:clt};
  see Theorem~\ref{thm:sharp_threshold} for the threshold)---on a log
  scale.  Bands are pointwise 95\% confidence intervals across 50
  replications.}
\label{fig:diagnostics}
\end{figure}

Figure~\ref{fig:diagnostics} shows that $\sqrt{n}\,\|R_n\|$ declines
toward zero for all four methods in every dimension.  At $d = 5$ and
$d = 20$ the curves have largely flattened by $n \approx 10^5$,
indicating that the remainder is no longer the dominant contribution
to the error at these sample sizes.  At $d = 50$ the curves are still
descending visibly across the full sample-size range and separate by
method along the estimation-burden ordering already visible in
Figure~\ref{fig:overview}: SA-ONS carries the largest dynamic
remainder, SA-AdaGrad and SA-RMSProp an intermediate one, and the
identity baseline the smallest.

\begin{table}[tbp]
\centering
\small
\caption{Empirical log--log slope of $\sqrt{n}\,\|R_n\|$ versus~$n$ in
the general regime ($S \neq H$), computed by ordinary least-squares
regression of $\log(\sqrt{n}\,\|R_n\|)$ on $\log n$ over the full
sample-size range, where $\|R_n\|$ is the mean across 50 replications
at each~$n$.  A slope near
$-0.50$ indicates a pre-asymptotic $\|R_n\| \propto n^{-1}$ regime in
which the remainder is still decaying; shallower slopes indicate that
the remainder has largely vanished and the $n^{-1/2}$ martingale term
dominates.}
\label{tab:scaled_slopes}
\begin{tabular}{lccc}
\toprule
\textbf{Method} & $d=5$ & $d=20$ & $d=50$ \\
\midrule
Identity (PR-SGD) & $-0.28$ & $-0.28$ & $-0.32$ \\
SA-AdaGrad        & $-0.27$ & $-0.29$ & $-0.36$ \\
SA-RMSProp        & $-0.31$ & $-0.29$ & $-0.31$ \\
SA-ONS            & $-0.27$ & $-0.30$ & $-0.40$ \\
\bottomrule
\end{tabular}
\end{table}

\noindent Table~\ref{tab:scaled_slopes} makes the visual pattern
quantitative.  At $d = 5$ and $d = 20$ the slopes cluster between
$-0.27$ and $-0.31$, well above the pre-asymptotic $-0.50$ value: the
remainder has shrunk enough that the $n^{-1/2}$ martingale term drives
the trend, consistent with the near-nominal coverage in
Figure~\ref{fig:overview} at these dimensions.  At $d = 50$ the slopes
are steeper ($-0.31$ to $-0.40$) and ordered by the method's
estimation burden---SA-ONS most negative, SA-AdaGrad next, then
Identity and SA-RMSProp---indicating that the remainder is still
decaying over the observed range, which is precisely what the larger
coverage gap and NMSE offset at $d = 50$ in the main figure reflect.

\subsection*{Threshold Violation Experiment Details}

The threshold violation experiment (Figure~\ref{fig:threshold}) uses the same
$d=5$ Toeplitz($0.4$) setup as the main synthetic experiment, with
$\eta_t = 2.0\,t^{-0.7}$ and ridge regularization~$+I$.  The SA-RMSProp
baseline updates the diagonal preconditioner with gain $\rho_t = 1/t$,
which yields stabilization rate $\beta = 1$ under bounded inputs
(Corollary~\ref{cor:sa_stab}); here the covariates are Gaussian and
unbounded, so $\beta = 1$ serves as a heuristic label.
The two constant-EMA
variants replace $\rho_t$ with a fixed $\rho \in \{0.5, 0.999\}$,
which heuristically do not admit a positive stabilization rate in the sense of
Definition~\ref{def:beta} (we do not formally establish this).  We use 100~replications and a log-spaced grid of
sample sizes up to $n = 5 \times 10^6$.

Figure~\ref{fig:threshold_supp} shows coverage and normalized MSE
(the scaled remainder is in Figure~\ref{fig:threshold} of the main text).
At $n = 5 \times 10^6$, all methods achieve coverage near $0.95$ and
$\mathrm{NMSE} \approx 1$, suggesting that the inferential impact of the
threshold violation is moderate at these sample sizes.

\begin{figure}[tbp]
\centering
\includegraphics[width=\textwidth]{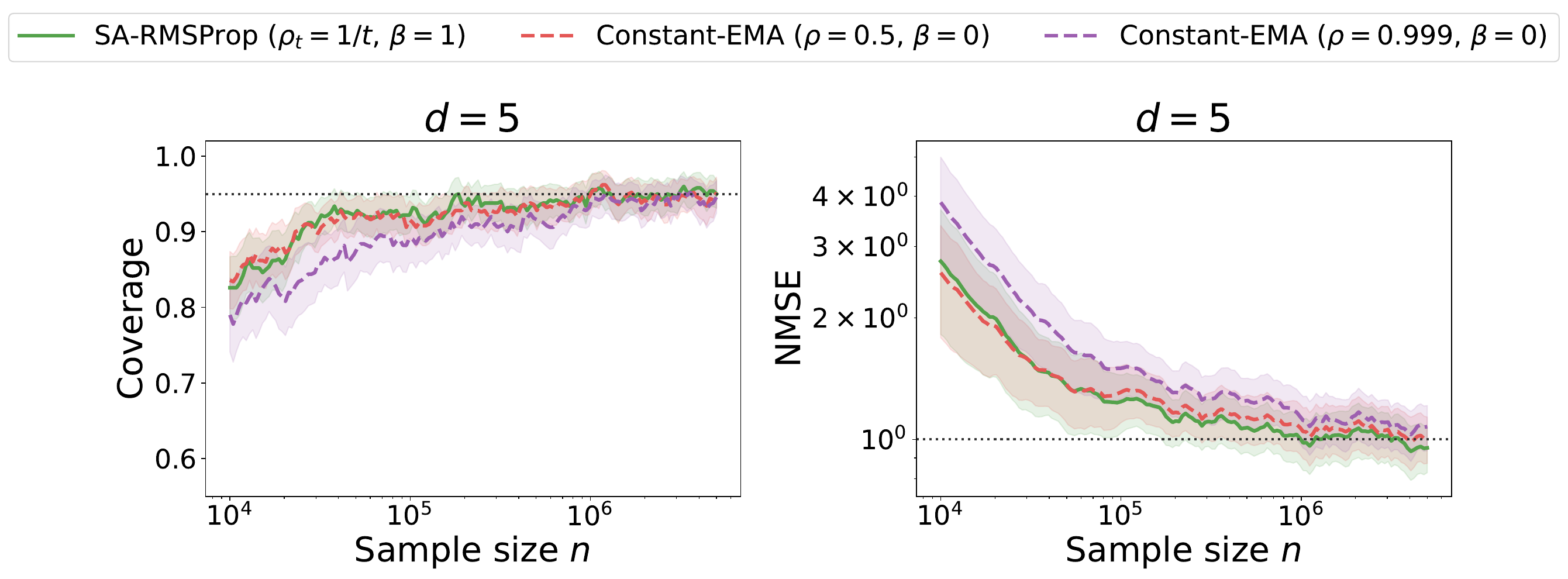}
\caption{Coverage and NMSE for the threshold violation experiment ($d = 5$),
  comparing SA-RMSProp ($\rho_t = 1/t$) against two constant-EMA variants
  ($\rho \in \{0.5, 0.999\}$).
  \textbf{Left:}~coverage (dashed line at the nominal $0.95$ level).
  \textbf{Right:}~NMSE on a log scale (dashed line at the target value~$1$).
  Both panels use a log-scaled horizontal axis for sample size~$n$.
  Bands are pointwise 95\% CIs over 100 replications.
  The scaled remainder panel is in Figure~\ref{fig:threshold}.}
\label{fig:threshold_supp}
\end{figure}

Table~\ref{tab:threshold_slopes} reports the log--log slopes of
$\sqrt{n}\,\|R_n\|$ computed from the second half of the sample-size range.
SA-RMSProp has a clearly negative slope ($-0.21$), consistent with
$\sqrt{n}\,\|R_n\| \to 0$.  The constant-EMA methods have near-zero slopes
($-0.02$ to $-0.03$), consistent with a remainder that does not vanish
over the observed range.  The 3--5$\times$ gap in remainder magnitude
between SA and EMA at $n = 5 \times 10^6$ (SA: $0.21$;
EMA $\rho = 0.5$: $0.72$; EMA $\rho = 0.999$: $1.04$) illustrates
the practical significance of the stabilization-rate threshold.

\begin{table}[tbp]
\centering
\small
\caption{Threshold violation experiment ($d = 5$, 100 replications).
The $\beta$ column lists the stabilization rate from
Definition~\ref{def:beta} (``---'' indicates no positive rate).
The slope is obtained by OLS regression of
$\log(\sqrt{n}\,\|R_n\|)$ on $\log n$ over the second half of the
sample-size range.  A negative slope is consistent with a vanishing
remainder; a near-zero slope suggests the remainder persists.  The last
column gives $\sqrt{n}\,\|R_n\|$ at the largest sample size.}
\label{tab:threshold_slopes}
\begin{tabular}{lccc}
\toprule
\textbf{Method} & $\boldsymbol{\beta}$ & \textbf{Slope} &
  $\sqrt{n}\,\|R_n\|$ at $n{=}5{\times}10^6$ \\
\midrule
SA-RMSProp ($\rho_t = 1/t$) & $1$ & $-0.21$ & $0.21$ \\
Constant-EMA ($\rho = 0.5$) & --- & $-0.02$ & $0.72$ \\
Constant-EMA ($\rho = 0.999$) & --- & $-0.03$ & $1.04$ \\
\bottomrule
\end{tabular}
\end{table}

\end{document}